\documentclass[11pt]{article}
\usepackage[margin = 1in]{geometry}

\usepackage[affil-it]{authblk}
\usepackage{lipsum}
\usepackage{amsfonts}
\usepackage{graphicx}
\graphicspath{ {images/} }
\usepackage{epstopdf}
\usepackage{dsfont}
\usepackage{mathtools}
\usepackage{empheq}
\usepackage{amsfonts}
\usepackage{amsgen,amsthm,amsmath,amstext,amsbsy,amsopn,amssymb,subfigure,stmaryrd,mathabx,mathrsfs}
\usepackage{comment}

\usepackage{blkarray}

\usepackage{url}
\usepackage{longtable}
\usepackage{mathtools}
\usepackage{multirow}
\usepackage{array}

\usepackage{enumitem}
\usepackage{hyperref}
\usepackage{natbib}
\usepackage{booktabs}
\usepackage{algorithm}
\usepackage{algpseudocode}
\usepackage{xcolor}
\usepackage[utf8]{inputenc} 
\usepackage[T1]{fontenc}

\ifpdf
  \DeclareGraphicsExtensions{.eps,.pdf,.png,.jpg}
\else
  \DeclareGraphicsExtensions{.eps}
\fi
\newcommand{\0}{{\mathbf{0}}}
\newcommand{\1}{{\mathbf{1}}}

\newcommand{\cG}{\mathcal{G}}
\newcommand{\cC}{\mathcal{C}}
\newcommand{\cH}{\mathcal{H}}

\newcommand{\cM}{\mathcal{M}}

\newcommand{\cD}{\mathcal{D}}
\newcommand{\cF}{\mathcal{F}}

\newcommand{\hW}{\widehat{W}}

\newcommand{\tA}{\widetilde{A}}

\newcommand{\hM}{\widehat{M}}
\newcommand{\hZ}{\widehat{Z}}
\newcommand{\tM}{\widetilde{M}}
\newcommand{\tC}{\widetilde{C}}

\newcommand{\barZ}{\widebar{Z}}

\newcommand{\bbN}{\mathbb{N}}
\newcommand{\bbP}{\mathbb{P}}
\newcommand{\bbE}{\mathbb{E}}
\newcommand{\bbR}{\mathbb{R}}

\newcommand{\bbO}{\mathbb{O}}

\newcommand{\balpha}{\boldsymbol{\alpha}}
\newcommand{\bbeta}{\boldsymbol{\beta}}

\newcommand{\rmvec}{{\rm{vec}}}
\renewcommand{\exp}{{\rm{exp}}}

\newcommand{\Bern}{{\rm{Bern}}}

\newcommand{\argmin}{\mathop{\rm arg\min}}

\newcommand{\rank}{{\rm rank}}

\newcommand{\F}{{\rm F}}

\newcommand{\Corr}{{\textrm Corr}}
\newcommand{\indi}{{\mathds{1}}}
\newcommand{\BC}{{\textnormal{BC}}}
\newcommand{\SBM}{{\textnormal{SBM}}}
\newcommand{\mean}{{\rm mean}}
\newcommand{\USVT}{{\rm USVT}}

\newtheorem{Definition}{Definition}
\newtheorem{Theorem}{Theorem}

\newtheorem{Lemma}{Lemma}
\newtheorem{Remark}{Remark}
\newtheorem{Corollary}{Corollary}

\newtheorem{Proposition}{Proposition}

\DeclareMathAlphabet\mathbfcal{OMS}{cmsy}{b}{n}

\hypersetup{
    colorlinks=true,
    linkcolor=blue,
    filecolor=blue,      
    urlcolor=cyan,
    citecolor=blue
}

\makeatletter
\newcommand*{\rom}[1]{\expandafter\@slowromancap\romannumeral #1@}
\makeatother

\allowdisplaybreaks

\begin{document}
	\title{Computational Lower Bounds for Graphon Estimation via Low-degree Polynomials}
	
	\author{Yuetian Luo$^1$ ~ and ~ Chao Gao$^2$ }
	\affil{University of Chicago}
	\date{}
	\maketitle

	\footnotetext[1]{Email: \texttt{yuetian@uchicago.edu}.}
	\footnotetext[2]{Email: \texttt{chaogao@uchicago.edu}. The research of CG is supported in part by NSF
Grants ECCS-2216912 and DMS-2310769, NSF Career Award DMS-1847590, and an Alfred Sloan fellowship. }

\begin{abstract}
Graphon estimation has been one of the most fundamental problems in network analysis and has received considerable attention in the past decade. From the statistical perspective, the minimax error rate of graphon estimation has been established by Gao et al (2015) for both stochastic block model (SBM) and nonparametric graphon estimation. The statistical optimal estimators are based on constrained least squares and have computational complexity exponential in the dimension. From the computational perspective, the best-known polynomial-time estimator is based on universal singular value thresholding (USVT), but it can only achieve a much slower estimation error rate than the minimax one. It is natural to wonder if such a gap is essential. The computational optimality of the USVT or the existence of a computational barrier in graphon estimation has been a long-standing open problem. In this work, we take the first step towards it and provide rigorous evidence for the computational barrier in graphon estimation via low-degree polynomials. Specifically, in SBM graphon estimation, we show that for low-degree polynomial estimators, their estimation error rates cannot be significantly better than that of the USVT under a wide range of parameter regimes and in nonparametric graphon estimation, we show low-degree polynomial estimators achieve estimation error rates strictly slower than the minimax rate. Our results are proved based on the recent development of low-degree polynomials by Schramm and Wein (2022), while we overcome a few key challenges in applying it to the general graphon estimation problem. By leveraging our main results, we also provide a computational lower bound on the clustering error for community detection in SBM with a growing number of communities and this yields a new piece of evidence for the conjectured Kesten-Stigum threshold for efficient community recovery. Finally, we extend our computational lower bounds to sparse graphon estimation and biclustering with additive Gaussian noise, and provide discussion on the optimality of our results.    
\end{abstract}


\begin{sloppypar}
\section{Introduction} \label{sec: introduction}
Network analysis has gained considerable research interest in the last couple of decades \citep{goldenberg2010survey,bickel2009nonparametric,girvan2002community,wasserman1994social}. A key task in network analysis is to estimate the underlying network generating process. It is useful for many important applications such as studying network evolution \citep{pensky2019dynamic}, predicting missing links \citep{miller2009nonparametric,airoldi2013stochastic,gao2015rate}, learning user preferences in recommender systems \citep{li2019nearest} and correcting errors in crowd-sourcing systems \citep{shah2018reducing}. In this paper, we are interested in the question: when could the underlying network generating process be estimated in a computationally efficient way? 

A general representation for {the generating process of} unlabelled exchangeable networks was first introduced by \cite{aldous1981representations,hoover1979relations} and was further developed and named {\it graphon} in  \cite{lovasz2006limits,diaconis2008graph,borgs2008convergent}. Specifically, in the graphon model, we observe an undirected graph of $n$ nodes and the associated adjacency matrix $A \in \{0,1\}^{n \times n}$. The value of $A_{ij}$ stands for the presence or the absence of an edge between the $i$th and the $j$th nodes. The sampling process of $A$ is determined as follows: conditioning on $(\xi_1, \ldots, \xi_n)$,
\begin{equation} \label{eq: graphon-model}
	 \text{ for all } 1\leq i < j \leq n, \quad  A_{ij} = A_{ji}\sim \Bern(M_{ij}), \quad \text{ where } \,  M_{ij} = f(\xi_i, \xi_j).
\end{equation} Here the sequence $\{\xi_i\}$ are i.i.d. random variables sampled from an unknown distribution $\bbP_\xi$ supported on $[0,1]$. A common choice for $\bbP_\xi$ is the uniform distribution on $[0,1]$. In this paper, we allow $\bbP_\xi$ to be arbitrary so that the model \eqref{eq: graphon-model} can be studied to its full generality. Conditioning on $(\xi_1, \ldots, \xi_n)$, $A_{ij}$'s are mutually independent across all $1\leq i < j \leq n$, and we adopt the convention that $A_{ii} = M_{ii} = 0$ for all $i \in [n]$. The function $f: [0,1] \times [0,1] \mapsto [0,1]$, which is assumed to be symmetric, is called graphon. In this work, we focus on this general graphon model and consider the problem of estimating $f$ given $A$. 

The concept of graphon plays a significant role in network analysis. It was originally developed as a limit of a sequence of graphs with growing sizes \citep{diaconis2008graph,lovasz2006limits,lovasz2012large}, and has been applied to various network analysis problems ranging from testing graph properties to characterizing distances between two graphs \citep{borgs2008convergent,borgs2012convergent,lovasz2012large}. The general graphon model in \eqref{eq: graphon-model} captures many special models of interest. For example, when $f$ is a constant function, it gives rise to the Erd\H{o}s-R{\'e}nyi random graph; when $f$ is a blockwise constant function or $\bbP_\xi$ has a discrete support, it specializes to the stochastic block model (SBM) \citep{holland1983stochastic}.    

One challenge in graphon estimation is the non-identifiability of $f$ due to the fact that the latent random variables $\{\xi_i\}$ are unobservable. To overcome this, we follow the prior work \cite{gao2015rate} and consider estimating $f$ under the empirical loss: 
\begin{equation} \label{eq: empirical-loss}
	\ell(\hM,  M_f) := \frac{1}{ {n \choose 2} } \sum_{1 \leq i < j \leq n} ( \hM_{ij} - (M_f)_{ij})^2,
\end{equation} where $\hM\in \bbR^{n \times n}$ and $(M_f)_{ij} := f(\xi_i,\xi_j)$.

There has been great interest in graphon estimation in the last decade \citep{wolfe2013nonparametric,airoldi2013stochastic,chan2014consistent,gao2015rate,klopp2017oracle} and we refer readers to Section \ref{sec: literature review} for detailed discussion. From the statistical perspective, \cite{gao2015rate} provided the first characterization for the minimax error rate in graphon estimation. In particular, for the SBM with $k$ blocks, the minimax estimation error rate is
\begin{equation} \label{eq: uncostrained-alg-rate}
	\text{SBM class}: \quad \inf_{\hM} \sup_{M \in \cM_k} \bbE \left( \ell( \hM,M) \right) \asymp \frac{k^2}{n^2}+\frac{\log k}{n},
\end{equation} where $\cM_k$ denotes the set of connectivity probability matrices in SBM with $k$ communities and its exact definition is given Section \ref{sec: computational-limit-graphon-estimation}. The minimax upper bound is achieved by a constrained least-squares estimator which needs to search over all possible graphon matrices in $\cM_k$ and is computationally inefficient, i.e., with runtime exponential in $n$.

When $f$ belongs to a H\"older space with smoothness index $\gamma$, the minimax estimation error rate is shown to be \citep{gao2015rate}
\begin{equation}\label{eq: uncostrained-alg-rate-holder}
	\begin{split}
		\text{H\"older class}: \quad \inf_{  \hM } \sup_{f \in \cH_\gamma(L) } \sup_{\bbP_\xi } \bbE \left(  \ell(\hM, M_f) \right) \asymp \left\{ \begin{array}{c c}
			n^{ - \frac{2\gamma}{\gamma+1} } &  0 < \gamma < 1,\\
			\frac{\log n}{n} &  \gamma \geq 1,
		\end{array} \right.
	\end{split}
\end{equation} where $\cH_\gamma(L)$ denotes the H\"older class to be introduced in Section \ref{sec: nonparametric-graphon-est}. Again, computing the minimax optimal estimator is expensive as it is based on first approximating a $\gamma$-smooth graphon with a blockwise constant matrix and then applying the constrained least-squares estimator.

From the computational perspective, the problem appears to be far less well-understood. The best polynomial-time estimator so far for graphon estimation is the universal singular value thresholding (USVT) \citep{chatterjee2015matrix}, and its sharp error bound was obtained by \cite{klopp2019optimal,xu2018rates},
\begin{equation} \label{eq: poly-time-alg-rate}
\begin{split}
	\text{SBM class}: & \quad 	\sup_{M \in \cM_k} \bbE \left( \ell(\hM_{\USVT},M)  \right) \leq C\frac{k}{n},\\
 \text{H\"older class}: &\quad \sup_{f \in \cH_\gamma(L)} \sup_{\bbP_\xi} \bbE \left( \ell(\hM_{\USVT}, M_f ) \right) \leq C n^{ - \frac{2\gamma}{2\gamma +1} },
\end{split}
\end{equation} for some constant $C > 0$ independent of $n$ and $k$.

Comparing \eqref{eq: uncostrained-alg-rate} and \eqref{eq: uncostrained-alg-rate-holder} with \eqref{eq: poly-time-alg-rate}, we see that there is a big gap between the estimation error rate achieved by the USVT and the minimax rate. It has been conjectured in \cite{xu2018rates} that the error rates in \eqref{eq: poly-time-alg-rate} are optimal within the class of polynomial-time algorithms, but no rigorous evidence is provided there. The fundamental computational limits for graphon estimation have been a long-standing open problem in the community \citep{xu2018rates,gao2021minimax,wu2021statistical}. In particular, in a recent survey about the statistical and computational limits for statistical problems with planted structures, \cite{wu2021statistical} explicitly highlight ``computational hardness of graphon estimation'' in their Section 5 as one of the six prominent open problems in the field. 

The gap on the performance of polynomial-time algorithms and unconstrained-time algorithms is quite common in high-dimensional statistical problems. There has been a flurry of progress in the statistics and theoretical computer science communities towards understanding the general ``statistical-computational trade-offs'' phenomenon. This topic focuses on the gap between signal-to-noise ratio (SNR) requirements under which the problem is information-theoretically solvable v.s. polynomial-time solvable. As the SNR increases, such problems often exhibit three phases of interest: (1) statistically unsolvable; (2) statistically solvable but computationally expensive, e.g., with runtime exponential in the input dimension; (3) easily solvable in polynomial-time. Many frameworks such as average-case reduction, statistical query (SQ), sum-of-squares (SoS) hierarchy, optimization landscape, and low-degree polynomials, have been proposed to study this phenomenon, and we refer readers to Section \ref{sec: literature review} for a thorough discussion. Based on these frameworks, rigorous evidence for the computational barrier has been provided for a wide class of statistical problems, such as planted clique, sparse PCA, submatrix detection, tensor PCA, robust mean estimation, and many others \citep{barak2019nearly,berthet2013complexity,ma2015computational,zhang2018tensor,brennan2018reducibility,diakonikolas2017statistical}. 

Despite all these successes, the graphon estimation problem is a rare example where to our best knowledge essentially no progress has been made under any framework. We think there are two major challenges in establishing the computational lower bound for graphon estimation: (1) in this problem, we want to establish a computational lower bound for {\it estimation} error rate, while most existing frameworks are mainly designed for hypothesis testing. Two natural hypothesis testing problems associated with graphon estimation do not have computational barriers, as we will discuss in Appendix \ref{sec: testing-problem-no-hardness}; (2) in contrast to the classical problems, such as planted clique or sparse PCA, there is no such a canonical SNR quantity in graphon estimation, though it is often critical to understand this quantity in order to apply existing frameworks.

 In this work, we overcome the above challenges and provide the first rigorous piece of evidence for the computational barrier in graphon estimation. The contributions of the paper are summarized below.

\subsection{Our Contributions} \label{sec: contributions}

The main result of the paper is given by the following theorem.
\begin{Theorem} \label{th: intro-theorem} Suppose $2 \leq k \leq \sqrt{n}$. For any $D \geq 1$, there exists a universal constant $c > 0$ such that
	\begin{equation} \label{ineq: intro-theorem-bound}
		\inf_{\hM \in \bbR[A]^{n \times n}_{\leq D} } \sup_{M \in \cM_k} \bbE( \ell(\hM,M) ) \geq \frac{ck}{nD^4}. 
	\end{equation} Here the notation $\hM \in \bbR[A]^{n \times n}_{\leq D}$ means that for all $(i,j) \in [n] \times [n]$, $\hM_{ij}$ is a polynomial of $A$ with degree no more than $D$.	
\end{Theorem}

It has been widely conjectured in the literature that for a broad class of high-dimensional problems, degree-$D$ polynomials are as powerful as the class of $n^{D}$ (up to $\log n$ factors in the exponent) runtime algorithms \citep{hopkins2018statistical}. Therefore, by setting $D = \log^{1 + \epsilon } n$ for any $\epsilon > 0$, Theorem \ref{th: intro-theorem} provides firm evidence that the best estimation error achieved by polynomial-time algorithms for graphon estimation under the SBM class cannot be faster than $\tilde{\Theta}(k/n)$. Up to logarithmic factors, this matches the upper bound achieved by USVT in \eqref{eq: poly-time-alg-rate}. 

We also establish a low-degree polynomial lower bound for graphon estimation under the H\"older class by approximating a smooth graphon via an SBM. See Theorem \ref{th: computational-lower-bound-nonparametric-graphon} in Section \ref{sec: nonparametric-graphon-est}. Again, the statistical error rate in \eqref{eq: uncostrained-alg-rate-holder} is strictly faster than the one achieved by low-degree polynomial algorithms. Combining the two results, we make a step in resolving the open problem regarding the computational lower bounds for graphon estimation raised by \cite{xu2018rates,gao2021minimax,wu2021statistical}.

\subsection{From Community Detection to Graphon Estimation}

Theorem \ref{th: intro-theorem} is proved by leveraging the recent advancement of low-degree polynomials developed by \cite{schramm2020computational}. Compared with previous work \citep{hopkins2017bayesian,hopkins2018statistical} on hypothesis testing, the low-degree polynomial lower bound in \cite{schramm2020computational} is directly established for \textit{estimation} problems under some prior distribution, and is thus particularly suitable for graphon estimation. Sharp computational lower bounds have been derived for several important examples in \cite{schramm2020computational} including the planted submatrix problem and the planted dense subgraph problem. However, unlike the examples in \cite{schramm2020computational}, the graphon estimation problem does not have a natural prior distribution and SNR, and therefore it is unclear how the general theorem of \cite{schramm2020computational} can be applied to such a setting.

To address this challenge, we consider another problem in network analysis called community detection. The goal of community detection is to recover the clustering structure of a network. For this purpose, a canonical model is the $k$-class SBM with within-class and between-class homogeneous connectivity probabilities, i.e., for two nodes from the same community, the connectivity probability is set to be $p$ and for two nodes from different communities, the connectivity probability is set to be $q$ \citep{mossel2015reconstruction,abbe2015exact}. Unlike the general SBM that has $\frac{k(k-1)}{2}$ model parameters, the SBM used for community detection has only $2$ parameters ($p$ and $q$) and can be viewed as a subset. For this subset, not only all the joint cumulants required by the theorem of \cite{schramm2020computational} can be computed, but we also have a nature SNR that quantifies the statistical-computational gap.

By applying \cite{schramm2020computational}, we show that a non-trivial clustering error cannot be achieved by low-degree polynomial algorithms below the generalized Kesten-Stigum threshold \citep{kesten1966additional,decelle2011asymptotic,chen2016statistical}. This result is of independent interest, and complements the recent progress by \cite{hopkins2017bayesian,bandeira2021spectral,banks2021local,brennan2020reducibility} on the computational limits of community detection. More importantly, the low-degree polynomial lower bound for community detection immediately implies the desired rate (\ref{ineq: intro-theorem-bound}) for graphon estimation by carefully choosing a least-favorable pair of $p$ and $q$.

This connection between graphon estimation and community detection from the perspective of computational limit is quite surprising. Without any computational constraint, the statistical limits of the two problems are derived from very different arguments in the literature. While the minimax rate of graphon estimation is polynomial \citep{gao2015rate,klopp2017oracle,gao2016optimal}, the minimax rate of community detection is exponential \citep{zhang2016minimax,fei2020achieving}, and one cannot be derived from the other. In contrast, we show that the low-degree polynomial lower bounds for the two problems can be established through the same argument. Detailed discussion on the connection between the two problems will be given in Section \ref{sec: computational-limit-graphon-estimation} and Section \ref{sec: community-detection}.

\subsection{Related Prior Work} \label{sec: literature review}
{\it Graphon estimation} has received considerable attention in the past decade \citep{wolfe2013nonparametric,yang2014nonparametric,airoldi2013stochastic,olhede2014network,chan2014consistent,borgs2015private,chatterjee2015matrix,gao2015rate,klopp2017oracle,gao2016optimal,zhang2017estimating,klopp2019optimal}. The minimax error rates for a variety of graphon estimation problems, including (sparse) SBM graphon estimation, nonparametric graphon estimation, graphon estimation with missing entries have been established in \cite{gao2015rate,klopp2017oracle,gao2016optimal,klopp2019optimal}. A number of efficient estimators for graphon estimation have been proposed \citep{airoldi2013stochastic,chatterjee2015matrix,chan2014consistent,zhang2017estimating,li2019nearest,gaucher2021optimality}. In the SBM setting, \cite{gaucher2021optimality} showed that a tractable estimator based on variational inference can achieve the minimax rate under appropriate assumptions on the connectivity probability matrix and the clustering labels. Without these additional assumptions, the best polynomial-time estimators for SBM/nonparametric graphon estimation are provided and analyzed in \cite{chatterjee2015matrix,klopp2019optimal,xu2018rates}, but they are far from optimal. Recently, graphon estimation in a bipartite graph, private graphon estimation, and stochastic block smooth graphon model have also been considered in  \cite{choi2017co,donier2023graphon,borgs2015private,sischka2022stochastic}.

\subsubsection{Statistical-computational Trade-offs} There has been a long line of work on studying the statistical-computational trade-offs in high-dimensional statistical problems. One powerful approach to establish the computational lower bounds is based on the average-case reduction \citep{berthet2013complexity, gao2017sparse,wang2016statistical,ma2015computational, cai2017computational,hajek2015computational,brennan2018reducibility,brennan2020reducibility,luo2020tensor,pananjady2022isotonic}, and it requires a distribution over instances in a conjecturally hard problem to be mapped precisely to the target distribution. Once the reduction is done, all hardness results from the conjectured hard problem can be automatically inherited to the target problem. On the other hand, the conclusions rely on conjectures that have not been proved yet. For this reason, many recent literature aims to show computational hardness results under some restricted models of computation, such as sum-of-squares \citep{ma2015sum,hopkins2017power,barak2019nearly}, statistical query (SQ) \citep{feldman2017statistical,diakonikolas2017statistical,diakonikolas2019efficient,feldman2018complexity}, class of circuit \citep{rossman2008constant}, convex relaxation \citep{chandrasekaran2013computational}, local algorithms \citep{gamarnik2014limits},  low-degree polynomials \citep{hopkins2017bayesian,kunisky2019notes} and others.

\subsubsection{Why the Low-degree Polynomial Framework} Among various ways to establish computational lower bounds, the low-degree polynomial framework is both clean and general. It has already been applied to many important high-dimensional problems and always leads to the same computational limits as conjectured in the literature. Compared with the low-degree polynomial method, the statistical query (SQ) framework is typically applied when the observed data consists of i.i.d. samples, but it is not clear how to cast graphon estimation into this form. The sum-of-squares (SoS) lower bounds provide strong evidence for the average-case hardness, but it is important to note that SoS lower bounds show hardness of certification problems. It does not necessarily imply hardness of estimation/recovery \citep{bandeira2020computational,banks2021local}. Average-case reduction is often applied to hypothesis testing problems \citep{berthet2013complexity,brennan2018reducibility}. To show the hardness of estimation from hypothesis testing, one often needs to further perform an extra reduction from estimation to testing. However, as we will see in Appendix \ref{sec: testing-problem-no-hardness}, two natural hypothesis testing problems associated with graphon estimation do not have a statistical-computational gap. 

\subsubsection{More Literature on Low-degree Polynomials} The idea of using low-degree polynomials to predict the statistical-computational gaps was recently developed in a line of work on studying the SoS hierarchy \citep{hopkins2017bayesian,hopkins2018statistical,barak2019nearly}. Many state-of-art algorithms such as spectral methods and approximate messaging messaging (AMP) \citep{donoho2009message} can be represented as low-degree polynomials \citep{kunisky2019notes,gamarnik2020low,montanari2022equivalence} and the ``low'' here typically means logarithmic in the dimension. In comparison to SoS computational lower bounds, the low-degree polynomial method is simpler to establish and appears to always yield the same results for natural average-case hardness problems. The majority of the existing low-degree polynomial hardness results are established for hypothesis testing problems based on the notion of {\it low-degree likelihood ratio}. Examples include unsupervised problems such as planted clique detection \citep{hopkins2018statistical, barak2019nearly}, community detection in SBM \citep{hopkins2017bayesian, hopkins2018statistical,jin2022phase}, spiked tensor model \citep{hopkins2017power, hopkins2018statistical, kunisky2019notes}, spiked Wishart model \citep{bandeira2020computational}, sparse PCA \citep{ding2019subexponential}, spiked Wigner model \citep{kunisky2019notes}, clustering in Gaussian mixture models \citep{loffler2020computationally,lyu2022optimal}, planted vector recovery \citep{mao2021optimal}, independent component analysis \citep{auddy2023large} as well as supervised learning problems such as tensor regression \citep{luo2022tensor}, and mixed sparse linear regression \citep{arpino2023statistical}. Very recently, the low-degree polynomial method has also been extended to establish computational hardness in statistical estimation/recovery problems \citep{schramm2020computational,koehler2022reconstruction,wein2023average,mao2023detection} and random optimization \citep{gamarnik2020low,wein2022optimal,bresler2022algorithmic}. It is gradually believed that the low-degree polynomial method is able to capture the essence of what makes sum-of-squares algorithms, and more generally, polynomial-time algorithms succeed or fail \citep{hopkins2018statistical, kunisky2019notes}. However, there are a couple of important examples where the low-degree polynomials can not predict the right computational threshold, such as the random 3-XOR-SAT problem \citep{kunisky2019notes}. In those settings, low-degree polynomials can be outperformed by some "brittle" algebraic methods with almost no noise tolerance, we refer readers to \cite{holmgren2021counterexamples,zadik2022lattice,diakonikolas2022non} for more discussions. Finally, it is worth mentioning that although we focus on the low-degree polynomial framework, it has been demonstrated that this framework is closely related to many other frameworks, such as SoS, SQ, free-energy landscape, and approximate message passing, from various perspectives \citep{hopkins2017power,barak2019nearly,brennan2021statistical,bandeira2022franz,montanari2022equivalence}.

\subsection{Organization of the Paper} \label{sec: organization}
After the introduction of notation and preliminaries of low-degree polynomials in Section \ref{sec:notations}, we present our main results on the low-degree polynomial lower bounds for graphon estimation in SBM and nonparametric graphon estimation in Section \ref{sec: computational-limit-graphon-estimation} and Section \ref{sec: nonparametric-graphon-est}, respectively. The low-degree polynomial lower bound for community detection in SBM is given in Section \ref{sec: community-detection}. Extensions of the main results to sparse graphon estimation and biclustering are given in Section \ref{sec: extension}. The proofs of the main results are presented in Section \ref{sec: proof-of-main-result} and the rest of the proofs are deferred to appendices.

\section{Notation and Preliminaries}\label{sec:notations}

Define $\bbN = \{0,1,2, \ldots \}$ and $[N] = \{1,\ldots,N\}$ for an integer $N$. For $\balpha \in \bbN^N$, define $|\balpha| = \sum_{i=1}^N \balpha_i$, $\balpha! = \prod_{i=1}^N \balpha_i!$, and for $X \in \bbR^N$, define $X^{\balpha} = \prod_{i=1}^N X_i^{\balpha_i}$. Given $\balpha,\bbeta \in \bbN^N$, we use $\balpha \geq \bbeta$ to mean $\balpha_i \geq \bbeta_i$ for all $i$. The operations $\balpha + \bbeta$ and $\balpha - \bbeta$ are performed entrywise. The notation $\bbeta \lneq \balpha $ means $\bbeta \leq \balpha$ and $\bbeta \neq \balpha$ (but not necessarily $\bbeta_i < \balpha_i$ for every $i$). Furthermore, for $\balpha \geq \bbeta$, we define ${\balpha \choose \bbeta} = \prod_{i=1}^N {\balpha_i \choose \bbeta_i}$. Sometimes, given $n \geq 1$ and $N = n(n-1)/2$, we will view $\balpha \in \bbN^N$ as a multigraph (without self-loops) on vertex set $[n]$, i.e., for each $i < j$, we let $\balpha_{ij}$ represent the number of edges between vertices $i$ and $j$. In this case, $V(\balpha) \subseteq [n]$ denotes the set of vertices spanned by the edges of $\balpha$. For any vector $v$, define its $\ell_2$ norm as $\|v\|_2 = \left(\sum_i |v_i|^2\right)^{1/2}$. For any matrix $D \in \mathbb{R}^{p_1\times p_2}$, the matrix Frobenius and spectral norms are defined as $ \|D\|_\F = \left(\sum_{i,j} D_{ij}^2\right)^{1/2}$ and $\|D\| = \max_{u\in \mathbb{R}^{p_2}}\|D u\|_2/\|u\|_2$, respectively. The notation $I_r$ represents the $r$-by-$r$ identity matrix and $\1_n$ is an all $1$ vector in $\bbR^n$. For any two sequences of numbers, say $\{a_n\}$ and $\{b_n\}$, denote $a_n \asymp b_n$ or $a_n = \Theta (b_n)$ if there exists uniform constants $c, C>0$ such that $ca_n \leq b_n\leq Ca_n$ for all $n$; $a_n \lesssim b_n$ means that $a_n \leq C b_n$ holds for some constant $C > 0$ independent of $n$ and $a_n = \tilde{\Theta}(b_n)$ if $a_n/ b_n$ and $b_n/a_n$ are both bounded by poly$\log(n)$, i.e., $a_n$ and $b_n$ are on the same order up to poly$\log(n)$ factors. Finally, throughout the paper, let $c,c',c'', C$ be some constants independent of $n$ and $k$, whose actual values may vary from line to line.

\subsection{Computational Lower Bounds for Estimation via Low-degree Polynomials} \label{sec: low-degree-preliminary}

Consider the {\it general binary observation model} and suppose the signal $X \in [\tau_0,\tau_1]^N$ with $0 \leq \tau_0 < \tau_1 \leq 1$ is drawn from an arbitrary but known prior. We observe $Y \in \{0,1 \}^N$ where $\bbE[Y_i | X_i] = X_i$ and $\{Y_i \}_{i=1}^N$ are conditionally independent given $X$. Let $\bbR[Y]_{\leq D}$ denote the space of polynomials $g: \bbR^N \to \bbR$ of degree at most $D$ of $Y$. Suppose the goal is to estimate a scalar quantity $x \in \bbR$, which is a function of $X$, then we have the following estimation lower bound for low-degree polynomial estimators.
\begin{Proposition}[\cite{schramm2020computational}] \label{prop: schramm-wein-binary}
	In the general binary model described above, denote $\bbP$ as the joint distribution of $x$ and $Y$. Then for any $D \geq 1$, we have 
	\begin{equation*}
	\inf_{g \in \bbR[Y]_{\leq D}} \bbE_{(x, Y) \sim \bbP} (g(Y) - x )^2 = \bbE(x^2) - \Corr_{\leq D}^2,
\end{equation*} where the degree-$D$ correlation $\Corr_{\leq D}$ is defined as
\begin{equation} \label{def: Corr}
	\Corr_{\leq D}:= \sup_{ \substack{g \in \bbR[Y]_{\leq D} \\ \bbE_{\bbP}[g^2(Y)] \neq 0 }  }  \frac{\bbE_{(x, Y) \sim \bbP} [ g(Y) \cdot x] }{\sqrt{ \bbE_{ \bbP}[ g^2(Y) ] }},  
\end{equation}
and satisfies the property
\begin{equation*}
	\Corr_{\leq D}^2 \leq \sum_{ \balpha \in \{0,1 \}^N, 0 \leq | \balpha | \leq D } \frac{\kappa_{\balpha}^2(x, X)}{( \tau_0 (1-\tau_1) )^{|\balpha|}}.
\end{equation*} Here $\kappa_{\balpha}(x,X)$ is defined recursively by 
\begin{equation} \label{eq: kappa-recursive-relation}
	\kappa_0(x, X) = \bbE(x) \text{ and }\kappa_{\balpha}(x, X) = \bbE(x X^{\balpha}) - \sum_{0 \leq \bbeta \lneq \balpha } \kappa_{\bbeta}(x, X) {\balpha \choose \bbeta } \bbE[X^{\balpha - \bbeta}] \text{ for } \balpha \text{ such that } |\balpha| \geq 1.
\end{equation}
\end{Proposition} 
We note that Proposition \ref{prop: schramm-wein-binary} provides a general $\ell_2$ estimation error lower bound for low-degree estimators of degree at most $D$. To show the low-degree polynomial lower bound in a specific problem, we then have to bound $\sum_{ \balpha \in \{0,1 \}^N, 0 \leq | \balpha | \leq D } \frac{\kappa_{\balpha}^2(x, X)}{( \tau_0 (1-\tau_1) )^{|\balpha|}}$, but to our knowledge, there is no easy and unified way to do that. One important interpretation for $\kappa_{\balpha}(x, X)$ is the following: if we view $\balpha$ as a multiset $\{ a_1, \ldots, a_m  \}$ with $m = \sum_{i=1}^N \balpha_i$, which contains $\balpha_i$ copies of $i$ for all $i \in [N]$, then $\kappa_{\balpha}(x, X)$ is the joint cumulant of a multiset of entries of the signal \citep[Claim 2.14]{schramm2020computational}:
\begin{equation} \label{eq: kappa-cumulant-connection}
	\kappa_{\balpha}(x, X) = \kappa(x, X_{a_1}, \ldots, X_{a_m}),
\end{equation} where $\kappa(\cdots)$ denotes the joint cumulant of a set of random variables and its formal definition and properties are provided in Appendix \ref{sec: cumulant}. This fact about $\kappa_{\balpha}(x, X)$ will be crucially used in our proofs of bounding $\Corr_{\leq D}$ for graphon estimation.

A similar result as Proposition \ref{prop: schramm-wein-binary} holds under the general additive Gaussian noise model as well. We defer the result in that setting to Appendix \ref{sec: comp-lower-bound-gaussian-model}.

\section{Computational Limits for Graphon Estimation in Stochastic Block Model}\label{sec: computational-limit-graphon-estimation}
We first define the parameter space of interest in SBM,
\begin{equation} \label{def: Mk}
\begin{split}
	\cM_k = \Big\{ & M = (M_{ij}) \in [0,1]^{  n \times n}:  M_{ii} = 0 \text{ for } i \in [n],  M_{ij} = M_{ji} = Q_{z_i z_j} \text{ for } i \neq j,\\ & \text{ for some }Q = Q^\top \in [0,1]^{k \times k}, z \in [k]^{n}\Big\}.	
\end{split}
\end{equation}
In other words, the connectivity probability between the $i$th and the $j$th nodes, $M_{ij}$, only depends on $Q$ through their clustering labels $z_i$ and $z_j$.
Given $M \in \cM_k$, we observe a random graph with adjacency matrix $A \in \{0,1 \}^{  n \times n}$ and its generative progress is given in \eqref{eq: graphon-model}. The minimax rate of estimating $M \in \cM_k$ is given in \eqref{eq: uncostrained-alg-rate}. It was shown in \cite{gao2015rate} that the minimax rate can be achieved by the solution of the following constrained least-squares optimization,
\begin{equation}
\min_{ M \in \cM_k}\|A-M\|_\F^2, \label{eq: constrained-least-squares}
\end{equation}
which, by the definition of $\cM_k$, is equivalent to
\begin{eqnarray*}
	 && \min_{ z\in [k]^n} \min_{Q \in \bbR^{k \times k}} \sum_{a,b\in[k]} \sum_{\substack{(i,j) \in z^{-1}(a) \times z^{-1}(b)\\ i\neq j}}  (A_{ij} - Q_{ab})^2 \\
	 &=& \min_{ z\in [k]^n} \sum_{a,b\in[k]} \sum_{\substack{(i,j) \in z^{-1}(a) \times z^{-1}(b)\\i\neq j}}  (A_{ij} - \widebar{A}_{ab}(z))^2,
\end{eqnarray*}
where $z^{-1}(a):=\{i\in[n]:z_i=a\}$, and
$$\widebar{A}_{ab}(z) := \begin{cases}\frac{1}{|z^{-1}(a)| |z^{-1}(b)|} \sum_{i \in z^{-1}(a) } \sum_{j \in z^{-1}(b)} A_{ij} & a\neq b ,\\
\frac{1}{|z^{-1}(a)|(|z^{-1}(a)|-1)}\sum_{\substack{i,j\in z^{-1}(a)\\ i\neq j}}A_{ij} & a=b.
\end{cases}$$
Unfortunately, since the optimization problem involves searching over all clustering patterns, it is computationally expensive to solve and has runtime exponential in $n$. 

This motivates a line of work on searching for polynomial-time algorithms. Among many of them, a prominent one is the universal singular value thresholding (USVT) proposed in \cite{chatterjee2015matrix}. It is a simple and versatile method for structured matrix estimation and has been applied to a variety of different problems such as low-rank matrix estimation, distance matrix completion, graphon estimation and ranking \citep{chatterjee2015matrix,shah2016stochastically}. In particular, given the SVD of $A = U \Sigma V^\top = \sum_{i=1}^n \sigma_i(A) u_i v_j^\top$, where $\sigma_i(A)$ denotes the $i$-th largest singular value of $A$, USVT estimates $M$ by
\begin{equation}
\hM_{\USVT}(\tau) = \sum_{i: \sigma_i(A) > \tau} \sigma_i(A) u_i v_i^\top,\label{eq:USVT-def}
\end{equation}
where $\tau$ is a carefully chosen tuning parameter.  
In the original paper by \cite{chatterjee2015matrix}, it was proved that USVT achieves the error rate $\sqrt{k/n}$ in estimating $M\in\cM_k$. Later, the error rate of USVT was improved to $k/n$ via a sharper analysis \citep{klopp2019optimal,xu2018rates}. Other polynomial-time algorithms in the literature \citep{airoldi2013stochastic,chan2014consistent,zhang2017estimating,li2019nearest,borgs2021consistent,gaucher2021optimality} either achieve error rates no better than $k/n$ or require additional assumptions on the matrix $M$. In this section, we will show that $k/n$ is the best possible error rate that can be achieved by low-degree polynomial algorithms.

In order to apply the general tool given by Proposition \ref{prop: schramm-wein-binary}, one needs to find a prior distribution supported on $\cM_k$ and compute all the joint cumulants under this prior. It turns out that the analysis of the cumulants is intractable under the least favorable prior constructed by \cite{gao2015rate} to prove the minimax lower bound. We need a simpler prior to apply Proposition \ref{prop: schramm-wein-binary}. To this end, we introduce a special class of SBM models considered in the community detection literature, denoted by $\cM_{k,p,q}$ ($0 \leq q < p \leq 1$), whose definition is given by
\begin{equation} \label{def: Mkpq}
\begin{split}
	\cM_{k,p,q} = \Big\{ & M = (M_{ij}) \in [0,1]^{  n \times n}:  M_{ii} = 0 \text{ for } i \in [n], \\
	& M_{ij} = M_{ji} = p \indi(z_i = z_j) + q \indi(z_i \neq z_j) \text{ for } i \neq j\text{ for some }z \in [k]^{n}\Big\}.
\end{split}
\end{equation}
Since $\cM_{k,p,q}$ is much simpler than $\cM_k$, it is not clear that it would lead to a sharp computational lower bound. However, we will show that when the algorithms are restricted within the class of low-degree polynomials, graphon estimation under $\cM_{k,p,q}$ can be as difficult as that under $\cM_k$.
 
We consider the following natural prior distribution $\bbP_{\SBM (p,q)}$ supported on $\cM_{k,p,q}$. In particular, $M \sim \bbP_{\SBM (p,q)}$ can be generated as follows: first, sample $z \in [k]^{n}$ according to $z_i \overset{i.i.d.}\sim \text{Unif}\{1,\ldots, k\}$ for all $i \in [n]$; then let $M_{ij} = p \indi(z_i = z_j) + q \indi(z_i \neq z_j)$ for all $1 \leq i < j \leq n$ and $M_{ii} = 0$ for all $i \in [n]$. Our first main result shows that when the SNR of $\cM_{k,p,q}$ is smaller than a certain threshold, then the estimation error of any low-degree polynomial estimator can be bounded from below. 
\begin{Theorem} \label{th: comp-limit-graphon-est} For any $0 < r <1$ and $D \geq 1$, if 
\begin{equation} \label{eq: SNR-condition}
	\frac{(p-q)^2}{q(1-p)} \leq \frac{r}{(D(D+1))^2} \left(\frac{k^2}{n} \wedge 1 \right),
\end{equation} then we have
\begin{equation} \label{ineq: low-degree-estimation-lower-bound-main}
	\inf_{\hM \in \bbR[A]^{n \times n}_{\leq D} } \bbE_{A, M \sim \bbP_{\SBM (p,q)}} (\ell(\hM, M)) \geq \frac{(p-q)^2}{k} - (p-q)^2 \left( \frac{1}{k^2} + \frac{r(2-r)}{(1-r)^2 n}  \right).
\end{equation} Here the notation $\hM \in \bbR[A]^{n \times n}_{\leq D}$ means that for all $(i,j) \in [n] \times [n]$, we have $\hM_{ij} \in \bbR[A]_{\leq D}$.
\end{Theorem}

To understand the result of Theorem \ref{th: comp-limit-graphon-est}, let us consider the special case $k\leq \sqrt{n}$ and ignore the second term on the right-hand side of (\ref{ineq: low-degree-estimation-lower-bound-main}). Then, Theorem \ref{th: comp-limit-graphon-est} indicates that whenever
\begin{equation}
\frac{n(p-q)^2}{k^2q(1-p)} \ll 1, \label{eq:simp-ks}
\end{equation}
the graphon estimation error cannot be better than $\frac{(p-q)^2}{k}$. We remark that $\frac{(p-q)^2}{k}$ is in fact a trivial error under the prior distribution $M \sim \bbP_{\SBM (p,q)}$, since it can be achieved by the constant estimator $\hM_{ij}=q$ for all $i\neq j$. One may recognize that the SNR condition (\ref{eq:simp-ks}) is related to the well-known {\it Kesten-Stigum} threshold \citep{kesten1966additional,decelle2011asymptotic} in the literature of community detection (See Section \ref{sec: community-detection} for more details). With arguments from statistical physics, it was conjectured that when the number of communities $k$ is a constant, non-trivial community detection is possible in polynomial time whenever
\begin{equation}
\frac{n(p-q)^2}{k(p+(k-1)q)} > 1, \label{eq:ks}
\end{equation}
at least under the asymptotic regime $p=a/n$ and $q=b/n$ for some constants $a>b$.  For general $p$ and $q$ such that $p\lesssim q< p<0.99$, the two SNRs on the left-hand sides of (\ref{eq:simp-ks}) and (\ref{eq:ks}) are of the same order. In fact, \eqref{eq:simp-ks} could be regarded as an asymptotic extension or generalized version of \eqref{eq:ks} when $k$ grows \citep{brennan2020reducibility} and \cite{chen2016statistical} conjectures (see their Conjecture 9) that it is the computational limits for community detection in SBM with a growing number of communities. Hence, Theorem \ref{th: comp-limit-graphon-est} simply says non-trivial graphon estimation is not possible below the {\it generalized Kesten-Stigum} threshold under the parameter space $\cM_{k,p,q}$.
	
To find a tight computational lower bound for graphon estimation under the original SBM class $\cM_k$, we define
\begin{equation}
\cM_k' = \bigcup_{0 \leq q \leq p \leq 1} \cM_{k,p,q}. \label{def: Mk-prime}
\end{equation}
Observe that $\cM_k'\subset \cM_k$, and we have
\begin{eqnarray}
\nonumber \inf_{\hM \in \bbR[A]^{n \times n}_{\leq D} } \sup_{M \in \cM_k} \bbE( \ell(\hM,M) ) &\geq& \inf_{\hM \in \bbR[A]^{n \times n}_{\leq D} } \sup_{M \in \cM_k'} \bbE( \ell(\hM,M) ) \\
\label{eq:later}&\geq& \inf_{\hM \in \bbR[A]^{n \times n}_{\leq D} } \bbE_{A, M \sim \bbP_{\SBM (p,q)}} (\ell(\hM, M)).
\end{eqnarray}
Since the above inequality holds for arbitrary $0 \leq q \leq p \leq 1$, we can find a pair of $p$ and $q$ to maximize the right-hand side of (\ref{ineq: low-degree-estimation-lower-bound-main}) under the SNR constraint (\ref{eq: SNR-condition}). This immediately leads to the following result.

\begin{Corollary} \label{coro: graphon-estimation-final-lower-bound}
Suppose $k \geq 2$. For any $D \geq 1$, there exists a universal constant $c > 0$ such that
	\begin{equation*}
		\inf_{\hM \in \bbR[A]^{n \times n}_{\leq D} } \sup_{M \in \cM_k} \bbE( \ell(M, \hM) ) \geq \frac{c}{D^4}\left( \frac{k}{n} \wedge \frac{1}{k} \right). 
	\end{equation*}
\end{Corollary}
When $k \leq \sqrt{n}$, the result in Corollary \ref{coro: graphon-estimation-final-lower-bound} reduces to Theorem \ref{th: intro-theorem}. Under the low-degree polynomial conjecture \citep{hopkins2018statistical} with $D = \log^{1+\epsilon} n$, the lower bound $\frac{k}{nD^4}$ matches the rate (\ref{eq: poly-time-alg-rate}) achieved by USVT up to some logarithmic factors. This is a bit surprising since Corollary \ref{coro: graphon-estimation-final-lower-bound} is actually proved for a much smaller parameter space $\cM_k'$ than the original one $\cM_k$. This provides a valuable insight that in the regime $k \leq \sqrt{n}$, the simple SBM prior $\bbP_{\SBM (p,q)}$ provides "computationally" a least favorable prior for graphon estimation.

When $k>\sqrt{n}$, however, the rate $\frac{1}{kD^4}$ does not match the performance of the USVT. This may result from the fact that the computational limits of the two spaces $\cM_k'$ and $\cM_k$ are different when $k$ is large. We will verify in the following Section \ref{sec: upper-bound-Mkpq} that the rate $\frac{1}{kD^4}$ is actually sharp if we consider the smaller space $\cM_k'$.

\subsection{Optimality of Theorem \ref{th: comp-limit-graphon-est}} \label{sec: upper-bound-Mkpq}

Our main result Theorem \ref{th: comp-limit-graphon-est} leads to the lower bound rate $\frac{k}{n}\wedge \frac{1}{k}$ in Corollary \ref{coro: graphon-estimation-final-lower-bound} for graphon estimation under the SBM class $\cM_k$. When $k>\sqrt{n}$, this rate becomes $\frac{1}{k}$, and does not match the upper bound achieved by USVT. In fact, since $\frac{1}{k}$ is even smaller than the minimax rate (\ref{eq: uncostrained-alg-rate}) when $k>n^{2/3}$, it cannot be the sharp. We will argue in this section that the sub-optimal rate $\frac{1}{k}$ is due to the choice of the subset $\cM_k'$ instead of an artifact of the proof of Theorem \ref{th: comp-limit-graphon-est}. The Bayes risk of Theorem \ref{th: comp-limit-graphon-est} with respect to the prior $\bbP_{\SBM (p,q)}$ (supported on $\cM_k'$) is optimal, and an improvement of the rate $\frac{1}{k}$ must involve a different subset.

Recall the definition $\cM_k' = \bigcup_{0 \leq q \leq p \leq 1} \cM_{k,p,q}$. Theorem \ref{th: comp-limit-graphon-est} and the inequality (\ref{eq:later}) imply
\begin{eqnarray*}
&& \inf_{\hM \in \bbR[A]^{n \times n}_{\leq D} } \sup_{M \in \cM_k'} \bbE( \ell(\hM, M) ) \\
&\geq& \inf_{\hM \in \bbR[A]^{n \times n}_{\leq D} } \sup_{0\leq q\leq p\leq 1}\bbE_{A, M \sim \bbP_{\SBM (p,q)}} (\ell(\hM, M)) \\
&\geq& \frac{c}{D^4} \left( \frac{k}{n} \wedge \frac{1}{k}  \right).
\end{eqnarray*}
The above lower bound cannot be improved. To see this, consider the following algorithm,
\begin{equation}
		\begin{split}
			\hM = \left\{\begin{array}{c c}
				\hM_{\USVT}(\tau) & k \leq \sqrt{n},\\
				\hM_{\mean} & k \geq \sqrt{n},
			\end{array}  \right.
		\end{split}\label{eq:yo}
	\end{equation}
where $(\hM_{\mean})_{ij} = (\hM_{\mean})_{ji} = \sum_{1 \leq u < v \leq n} A_{uv}/{n \choose 2}$. When $k \leq \sqrt{n}$, the USVT with $\tau \asymp \sqrt{n}$ achieves the rate $\frac{k}{n}$ \citep{xu2018rates}. When $k>\sqrt{n}$, a straightforward calculation (see Appendix \ref{app: guarantee-of-M-mean}) leads to
$$\sup_{0\leq q\leq p\leq 1}\bbE_{A, M \sim \bbP_{\SBM (p,q)}} (\ell(\hM_{\mean}, M))\leq C\frac{1}{k}.$$
In addition, for any $M\in\cM_k'$, if we further assume $\frac{n}{\beta k}\leq\sum_{i=1}^n\indi((z_M)_i = a)\leq \frac{n\beta}{k}$ for all $a\in[k]$ for some constant $\beta>1$, we also have $\bbE_A( \ell(\hM_{\mean}, M) ) \leq C\frac{1}{k}$. In other words, the estimator (\ref{eq:yo}) achieves the rate $\frac{k}{n}\wedge \frac{1}{k}$, and thus the lower bound cannot be improved.

As we have discussed in Section \ref{sec: computational-limit-graphon-estimation}, our low-degree polynomial lower bounds for graphon estimation are derived by the connection to community detection. When $k>\sqrt{n}$, it is likely that the computational limits of the two problems are very different. A sharp lower bound for graphon estimation probably requires the construction of a very different subset than $\cM_k'$. We leave this problem open.

\subsection{A Matching Low-degree Polynomial Upper Bound for $\bbP_{\SBM (p,q)}$} \label{sec:low-degree-up-bound}
Though the error rate of USVT matches our low-degree polynomial lower bound when $k \leq \sqrt{n}$, it is not strictly a low-degree polynomial algorithm, i.e., its entry cannot be written as a polynomial of entries of $A$. In this section, we provide a rigorous low-degree polynomial algorithm with near-optimal guarantees. For technical convenience, we consider the setting $M  \sim \bbP_{\SBM (p,q)}$ as defined in Section \ref{sec: computational-limit-graphon-estimation}. \footnote{In fact, when $k\leq \sqrt{n}$, one can show, via similar arguments in \cite{gao2015rate}, that the information-theoretically optimal rate under $M  \sim \bbP_{\SBM (p,q)}$ is still $\tilde{\Theta}(1/n)$ for the least favorable pair of $(p,q)$.} The algorithm is described below.
		
		\begin{algorithm}[!h] \caption{Low-degree Polynomial Algorithm for SBM Graphon Estimation}\label{alg:low-degree-alg}
			\begin{algorithmic}[1]
				\State \textbf{Input:} $A,p,q$, $k$, $r$, $t_1$ and $t_2$.
				\State (Fill the diagonal and transform the data) Let $\Lambda \in \bbR^{n \times n}$ be a diagonal matrix with i.i.d. $\Bern(p)$ entries on its diagonal and they are independent of $A$; let $\tA = A + \Lambda - q \1_n \1_n^\top$.
				\State (Power iteration) Generate an independent random matrix $B \in \bbR^{p \times r}$ with i.i.d. $N(0,1)$ entries; compute $\tA^{t_1} B$.
				\State (Gradient descent) Run $t_2$ iterations of gradient descent (GD) with zero initialization on the objective $\min_{W \in \bbR^{r \times n}} \|\tA^{t_1} B W - \tA\|_\F^2$, i.e.,
				for $l = 0$ to $t_2-1$, compute $$W_{l+1} = W_l - \eta B^\top \tA^{t_1}( \tA^{t_1} B W_l - \tA ) \quad \textnormal{ with } \quad W_0 = \0.$$				\State \textbf{Output:} $\hM = \tA^{t_1} B W_{t_2} + q \1_n \1_n^\top$.
			\end{algorithmic}
		\end{algorithm}
The main idea of Algorithm \ref{alg:low-degree-alg} is to simulate SVD via power iteration. However, power iteration does not lead to the right scaling without additional normalization, and this motivates us to run a further least-squares optimization to normalize the matrix. Least-squares is not a low-degree algorithm since it involves matrix inverse, and this is simulated via gradient descent. 

By simple counting, one can show that each entry of $\hM$ is a polynomial of entries of $A,\Lambda,B$ with degree at most $2t_1 t_2$. The guarantee of $\hM$ returned by Algorithm \ref{alg:low-degree-alg} is given as follows.
\begin{Theorem} \label{th:low-degree-up-bound}
 Take $r = 2k$, $t_1 = t_2 = C' \log n$ and the stepsize of GD to be $\eta = \frac{1}{C''\left( (\frac{n(p-q)}{k} + C'' \sqrt{n} )^{2t_1} k \vee (C'' n)^{t_1 + 1}  \right)} $ for some large $C', C'' > 0$ in Algorithm \ref{alg:low-degree-alg}. Then there exist $c,C,\bar{C}  > 0$ depending only on $C',C''$ such that when $n \geq C k \log^3 n$, we have with $\bbP_{\SBM (p,q)}$-probability at least $1-n^{-\bar{C}}$, the $\hM$ in Algorithm \ref{alg:low-degree-alg} satisfies $\ell(\hM, M) \leq \frac{c (k+\log n) \log^2 n}{n}$. 
\end{Theorem}
To summarize, $\hM$ is a $O(\log^2 n)$-degree polynomial estimator that achieves the $k/n$ error rate up to logarithmic factors. One important feature of Algorithm \ref{alg:low-degree-alg} is that it works for any $p, q \in [0,1]$ and it automatically adapts between situations with a spectral gap or not. In addition, we note that in order to make the algorithm work, it is important to choose $r$ to satisfy $r/k > 1$ for the sketching matrix $B$ in Step 3. With this choice, the least-squares optimization is well-conditioned and gradient descent achieves a linear rate of convergence in the high SNR regime when there is a spectral gap. In the low SNR regime without a spectral gap, gradient descent after $t_2=O(\log n)$ iterations stays close to the zero initialization, which still works for our purpose.

Compared with the low-degree upper bounds in \cite{schramm2020computational} where a single power iteration is needed in planted submatrix and dense subgraph problems, we have to run a logarithmic number of power iterations followed by a logarithmic number of iterations of gradient descent. The logarithmic number of power iterations seems to be necessary for us to extract the subspace information of $A$. In general, the proposed algorithm can understood as a principled way of simulating spectral algorithms via low-degree polynomials. On the other hand, we note that even though our algorithm is polynomial-time, it has degree $O(\log^2 n)$. It will be interesting to find a $O(\log n)$-degree algorithm to simulate spectral algorithms.

\begin{Remark}
	Careful readers may notice that our estimator $\hM$ is a low-degree polynomial of entries of $A$ as well as independently generated $\Lambda$ and $B$, while our low-degree polynomial lower bounds in Section \ref{sec: computational-limit-graphon-estimation} are proved for the class of deterministic polynomials. However, this is not an issue since the low-degree polynomial lower bounds will continue to hold if we consider the class of polynomials of $A,\Lambda$ and $B$. This is due to the fact that cumulants on two groups of independent random variables are zero (see Proposition \ref{prop: cumulant-independent-prop} in Appendix \ref{sec: preliminary-low-degree}). The same issue has also been dealt with in Claim A.1 by \cite{schramm2020computational}. 
\end{Remark}

\section{Computational Limits for Nonparametric Graphon Estimation}\label{sec: nonparametric-graphon-est}

Let us proceed to nonparametric graphon estimation. We first introduce a class of H\"older smooth graphon. Since graphons are symmetric functions, we only need to consider functions on $\cD = \{(x,y) \in [0,1] \times [0,1]: x \geq y \}$. Define the derivative operator by
\begin{equation*}
	\nabla_{jk} f (x,y) = \frac{\partial^{j+k}}{(\partial x)^j (\partial y)^k} f(x,y),
\end{equation*}
and we adopt the convention $\nabla_{00} f(x,y) = f(x,y)$. Given a $\gamma > 0$, the H\"older norm of $f$ is defined as
\begin{equation*}
	\|f\|_{\cH_{\gamma}} = \max_{j +k \leq \lfloor \gamma \rfloor  } \sup_{(x,y) \in \cD} \left| \nabla_{jk} f(x,y) \right| + \max_{j +k = \lfloor \gamma \rfloor } \sup_{(x,y) \neq (x',y') \in \cD} \frac{\left| \nabla_{jk} f(x,y) - \nabla_{jk}f(x', y') \right|}{( |x-x'| + |y-y'|  )^{\gamma - \lfloor \gamma \rfloor }},
\end{equation*} and the H\"older class with smoothness parameter $\gamma > 0$ and radius $L > 0$ is defined as
\begin{equation*}
	\cH_{\gamma}(L) = \{\|f\|_{\cH_{\gamma}} \leq L: f(x,y) = f(y,x) \text{ for } x \geq y \}.
\end{equation*} 
Finally, the class of smooth graphon of interest is 
\begin{equation*}
	\cF_{\gamma}(L) = \{ 0 \leq f \leq 1: f \in \cH_{\gamma}(L) \}.
\end{equation*}

The minimax rate of estimating $f \in \cF_{\gamma}(L)$ is given by \eqref{eq: uncostrained-alg-rate-holder}. Note that this rate can also be written as
\begin{equation} 
	\min_{k} \left( \frac{k^2}{n^2} +\frac{\log k}{n} + k^{-2(\gamma\wedge 1)} \right) \asymp \left\{ \begin{array}{c c}
			n^{ - \frac{2\gamma}{\gamma+1} } & 0 < \gamma < 1,\\
			\frac{\log n}{n} &  \gamma \geq 1,
		\end{array} \right. \label{eq:bvt-sbm-minimax}
\end{equation}
where the first term $\frac{k^2}{n^2} +\frac{\log k}{n}$ is the minimax rate of graphon estimation under the SBM class $\cM_k$, and the second term $k^{-2(\gamma\wedge 1)}$ is the error of approximating a nonparametric graphon $f \in \cF_{\gamma}(L)$ by an SBM with $k$ blocks (Lemma 2.1 of \cite{gao2015rate}). A rate-optimal estimator can be constructed by the same constrained least-squares optimization (\ref{eq: constrained-least-squares}) with $k$ chosen to be $\lceil n^{\frac{1}{1+ \gamma \wedge 1}} \rceil$, i.e., the solution to the bias-variance tradeoff (\ref{eq:bvt-sbm-minimax}). Despite its statistical optimality, solving (\ref{eq: constrained-least-squares}) is computationally intractable.

In terms of polynomial-time algorithms, it was proved by \cite{xu2018rates} that the USVT estimator (\ref{eq:USVT-def}) with tuning parameter $\tau\asymp\sqrt{n}$ achieves the rate (\ref{eq: poly-time-alg-rate}). Just as (\ref{eq:bvt-sbm-minimax}), the sub-optimal rate (\ref{eq: poly-time-alg-rate}) can also be written in the form of bias-variance tradeoff,
\begin{equation}
	\min_{k} \left( \frac{k}{n} + k^{-2\gamma} \right) \asymp n^{-2\gamma/(2\gamma + 1)},\label{eq:USVT-tradeoff}
\end{equation}
where $\frac{k}{n}$ is the error rate of estimating a rank-$k$ matrix, and $k^{-2\gamma}$ is the error of approximating a nonparametric graphon $f \in \cF_{\gamma}(L)$ by a rank-$k$ matrix (Proposition 1 of \cite{xu2018rates}). The optimal choice of $k$ is given by $\lceil n^{\frac{1}{1+ 2\gamma}} \rceil$. Other polynomial-time algorithms in the literature \citep{airoldi2013stochastic,chan2014consistent,zhang2017estimating,li2019nearest,borgs2021consistent} either achieve error rates no better than $n^{-2\gamma/(2\gamma + 1)}$ or require additional assumptions on $f$. In the following result, we provide a lower bound for nonparametric graphon estimation within the class of low-degree polynomials. 

\begin{Theorem} \label{th: computational-lower-bound-nonparametric-graphon}
	Suppose $\gamma > 0.5$. For any $D \geq 1$, there exists $c >0$ only depending on $L$ and $\gamma$ such that
	\begin{equation*}
		\inf_{\hM \in \bbR[A]^{n \times n}_{\leq D}} \sup_{f \in \cF_\gamma (L)} \sup_{\bbP_\xi} \bbE \left( \ell(\hM, M_f) \right) \geq c n^{- \frac{2\gamma+1}{2\gamma + 2} }/D^4.
	\end{equation*}
\end{Theorem}
Theorem \ref{th: computational-lower-bound-nonparametric-graphon} is proved by similar arguments that lead to Theorem \ref{th: intro-theorem}. A simple calculation shows that the low-degree polynomial lower bound $n^{- \frac{2\gamma+1}{2\gamma + 2} }$ is strictly slower than the statistical rate (\ref{eq:bvt-sbm-minimax}) by a factor scales polynomially in $n$ whenever $\gamma > 0.5$. It confirms that the minimax rate in nonparametric graphon estimation cannot be achieved by the class of low-degree polynomials when $\gamma > 0.5$, providing rigorous evidence for the statistical-computational gap.

Careful readers may notice the gap between the low-degree polynomial lower bound and the upper bound achieved by USVT. We believe this is due to the fact that Theorem \ref{th: computational-lower-bound-nonparametric-graphon} is proved based on Theorem \ref{th: comp-limit-graphon-est}, where we use the SBM model class $\cM_k'$, i.e., SBM class with two parameters $(p,q)$, to approximate a H\"older smooth graphon. To be specific, the optimal choice of $p,q$ in Theorem \ref{th: comp-limit-graphon-est} would satisfy $p - q \asymp \frac{k}{\sqrt{n}}$. At the same time, to guarantee that SBM$(p,q)$ is a $\gamma$-H\"older smooth graphon, we need the condition $p-q \lesssim 1/k^{\gamma}$ (see Proposition \ref{prop:smooth-f} in Appendix \ref{sec: proof-nonparametric-graphon}), i.e., $ \frac{k^2}{n} \lesssim k^{-2\gamma}$. So our choice of $k$ is from the tradeoff between $\frac{k}{n}$ and $k^{-2\gamma -1}$, which is different from the tradeoff (\ref{eq:USVT-tradeoff}) for USVT. To close this gap, we believe that a more sophisticated SBM class is needed to approximate H\"older smooth graphons, which is beyond the scope of the paper and we leave it as an interesting future direction. 
\section{Computational Limits for Community Detection in SBM}\label{sec: community-detection}

The key to the derivation of the computational lower bound for graphon estimation is the understanding of community detection under the distribution $\bbP_{\SBM (p,q)}$ supported on $\cM_{k,p,q}$. For any $M\in\cM_{k,p,q}$ with $p>q$, there exists a unique $z\in[k]^n$ such that
$$M_{ij}=p \indi(z_i = z_j) + q \indi(z_i \neq z_j).$$
We write such $z$ as $z_M$ to emphasize its dependence on $M$. The membership matrix $Z_M$ is defined by: for $i \in [n]$, $(Z_M)_{ii} = 0$, for all $i \neq j$,
\begin{equation}
(Z_M)_{ij}=\indi((z_M)_i = (z_M)_j)=\frac{M_{ij}-q}{p-q}.\label{eq:relation-Z-M}
\end{equation}
The goal of the community detection is to recover the clustering labels $z_M$ or the membership matrix $Z_M$.

The problem of community detection has been widely studied in the literature \citep{bickel2009nonparametric,rohe2011spectral,lei2015consistency,jin2015fast}. When $k=2$, ground-breaking work by \cite{mossel2015reconstruction,mossel2018proof,massoulie2014community} shows that non-trivial community detection (better than random guess) is possible if and only if $\frac{n(p-q)^2}{2(p+q)} > 1$. Sharp SNR thresholds have also been derived for partial recovery and exact recovery \citep{mossel2014consistency,abbe2015exact}. We refer the readers to \cite{abbe2017community,moore2017computer} for extensive reviews on the topic. 

It turns out that the problem starts to exhibit a statistical-computational gap as $k$ gets larger. When $k$ is a large {\it constant}, with arguments from statistical physics, it was conjectured in the literature that non-trivial community detection is possible in polynomial time whenever the SNR exceeds the {\it Kesten-Stigum} threshold \citep{kesten1966additional,decelle2011asymptotic}, which is sharply characterized by $\frac{n(p-q)^2}{k(p+(k-1)q)} > 1$ at least under the asymptotic regime $p=a/n$ and $q=b/n$ for some constants $a>b>0$. In contrast, the information-theoretic limit only requires $\frac{n(p-q)^2}{pk\log k}$ to be large for non-trivial community detection \citep{banks2016information,zhang2016minimax}, so there is a (constant level) statistical-computational gap. The algorithmic side of this conjecture has been resolved in \cite{abbe2018proof}, while rigorous evidence of the computational lower bound has been much more elusive and was partially provided by \cite{hopkins2017bayesian,bandeira2021spectral,banks2021local}. There is also a statistical-computational gap for the detection version of the problem and statistical/computational thresholds for detection and recovery problems are the same when $k$ is a constant \citep{bandeira2021spectral,banks2021local}.

In this section, we focus on the problem of community detection with a potentially growing $k$ as $n$ grows. Different from the constant $k$ regime, in Appendix \ref{sec: testing-problem-no-hardness}, we illustrate that two natural hypothesis testing problems associated with SBM do not have a statistical-computational gap when $k$ grows (at least there is not a statistical-computational gap scaling polynomially in $n$). However, it was conjectured in \cite{chen2016statistical,brennan2020reducibility} that there is still a statistical-computational gap for the recovery problem in SBM with a growing number of communities and the computational limit is given by the generalized Kesten-Stigum threshold \eqref{eq:simp-ks}.

Our goal of this section is to present a low-degree polynomial lower bound for recovery in SBM with growing $k$ under the following loss function, 
\begin{equation*}
	\ell(\hZ,Z) = \frac{1}{ {n \choose 2} } \sum_{1 \leq i < j \leq n} ( \hZ_{ij} - Z_{ij})^2.
\end{equation*}
Compared with Hamming loss of estimating the clustering labels, the above loss for estimating the membership matrix avoids the identifiability issue due to label switching. Under the distribution $M\sim\bbP_{\SBM (p,q)}$, it is easy to show that a trivial error of community detection is
$$\bbE_{A, M \sim \bbP_{\SBM (p,q)}} ( \ell(\hZ, Z_M) )=\frac{1}{k} - \frac{1}{k^2},$$
achieved by $\hZ = \frac{1}{k} \mathbf{1}_{n \times n}$ where $\mathbf{1}_{n \times n}$ denotes a $n \times n$ matrix with all $1$ in its entries. Random guess would achieve a slightly worse error $\frac{2}{k}(1- \frac{1}{k})$ under the same setting. Therefore, we say that an algorithm $\hZ$ can achieve non-trivial community detection if its error is much smaller than $\frac{1}{k} -\frac{1}{k^2} $. When $\frac{n(p-q)^2}{pk^2}>C$ for some sufficiently large constant $C>0$, non-trivial community detection is possible, and polynomial-time algorithms including spectral clustering \citep{chin2015stochastic,abbe2018proof} and semi-definite programming (SDP) \citep{guedon2016community,li2021convex} would work.\footnote{For completeness, the performance of SDP under the model $M\sim\bbP_{\SBM (p,q)}$ is given in Appendix \ref{sec:SDP}} Next, we provide a result in the other direction.

\begin{Theorem} \label{th: clustering-error-comp-limit}
	For any $D \geq 1$, suppose \begin{equation*}
	\frac{(p-q)^2}{q(1-p)} \leq \frac{1}{2(D(D+1))^2} \left(\frac{k^2}{n} \wedge 1 \right),
\end{equation*} then 
	\begin{equation} \label{ineq: comp-lower-bound-clustering}
		\inf_{\hZ \in \bbR[A]_{\leq D}^{n \times n}}  \bbE_{A, M \sim \bbP_{\SBM (p,q)}} ( \ell(\hZ, Z_M) ) \geq \frac{1}{k} - \frac{1}{k^2} - \frac{3}{n} .
	\end{equation}
	In particular, when $k \leq \sqrt{n}$ and
	\begin{equation} \label{ineq: SNR-regime1}
		\frac{n(p-q)^2}{k^2q(1-p)} \leq \frac{1}{2(D(D+1))^2},
	\end{equation}
  the lower bound \eqref{ineq: comp-lower-bound-clustering} holds.
\end{Theorem}

Theorem \ref{th: clustering-error-comp-limit} shows that when the SNR $\frac{n(p-q)^2}{k^2q(1-p)}$ is small, no low-degree polynomial algorithm can achieve non-trivial community detection, which provides firm evidence of the conjecture of the {\it generalized Kesten-Stigum} threshold for community detection in SBM with a growing number of communities. In fact, Theorem \ref{th: clustering-error-comp-limit} can be viewed as a rearrangement of Theorem \ref{th: comp-limit-graphon-est}. Given the relation (\ref{eq:relation-Z-M}), the loss functions of graphon estimation and community detection can be linked through $\ell(\hM, M)=(p-q)^2\ell(\hZ, Z_M)$.

As we have mentioned above, there are a couple of existing pieces of evidence for the computational limits of community detection in SBM when $k$ is a constant \citep{hopkins2017bayesian,bandeira2021spectral,banks2021local}. While when the number of communities grows, to our knowledge, there is only one piece of evidence for the hardness of recovery in SBM via average-case reduction from secrete-leakage planted clique \cite[Section 14.1]{brennan2020reducibility}. They considered establishing the computational lower bound for a testing problem where the null is the Erd\H{o}s-R{\'e}nyi random graph and the alternative is a variant of imbalanced SBM (ISBM) with two features: first, the averaged number of degrees under the null and alternative are matched; second, the ISBM under the alternative is a mean-field analogy of the original SBM so that the testing problem becomes harder and it matches the hardness of the recovery problem. The reduction result is significant as all existing computational hardness evidence for secrete-leakage planted clique can be inherited to the testing problem they consider. The limitation is that they do not directly handle the estimation problem under the original SBM model; moreover, their reduction only works when $k = o(n^{1/3})$, while our computational lower bound is valid as long as $k \leq \sqrt{n}$.

\section{Extensions and Discussion} \label{sec: extension}

Our main results can also be extended to the following settings. In Section \ref{sec: sparse-graphon-estimation}, we consider sparse graphon estimation, and present a corresponding low-degree polynomial lower bound. Section \ref{sec: biclustering} considers the estimation problem under a biclustering structure with additive Gaussian noise, which can be regarded as an extension of the SBM to rectangular matrices.

\subsection{Computational Lower Bound for Sparse Graphon Estimation} \label{sec: sparse-graphon-estimation}
Network observed in practice is often sparse in the sense that the total number of edges is of order $o(n^2)$. The problem of sparse graphon estimation is typically more complex than the dense one and has also been widely considered in the literature \citep{bickel2009nonparametric,bickel2011method,borgs2018lp,borgs2019lp,klopp2017oracle,gao2016optimal,borgs2021consistent}. This section will focus on the sparse SBM model. Given any $0 < \rho < 1$, the class of probability matrices is defined as
\begin{equation} \label{def: sparse-Mk}
\begin{split}
	\cM_{k,\rho} = 	\Big\{ &M = (M_{ij}) \in [0,\rho]^{  n \times n}:  M_{ii} = 0  \text{ for } i \in [n],M_{ij} = M_{ji} = Q_{z_i z_j} \text{ for } i \neq j,  \\ 
	&\text{ for some }z \in [k]^{n}, Q = Q^\top \in [0,\rho]^{k \times k}\Big\}.	
\end{split}
\end{equation} 
The minimax rate for sparse graphon estimation has been derived by \cite{klopp2017oracle,gao2016optimal},
\begin{equation*}
	\inf_{\hM} \sup_{M \in \cM_{k,\rho}} \bbE \left( \ell( \hM,M) \right) \asymp \rho\left(\frac{k^2}{n^2}+\frac{\log k}{n}\right) \wedge \rho^2.
\end{equation*}
By solving a constrained least-squares optimization problem $\min_{ M \in \cM_{k,\rho}}\|A-M\|_\F^2$ similar to (\ref{eq: constrained-least-squares}), one achieves the rate $\rho\left(\frac{k^2}{n^2}+\frac{\log k}{n}\right)$. The other part of the minimax rate $\rho^2$ can be trivially achieved by $\hM = \0$. In terms of polynomial time algorithms, \cite{klopp2019optimal} considered a USVT estimator with tuning parameter $\tau\asymp\sqrt{n\rho}$, and showed that as long as $\rho\geq\frac{\log n}{n}$,
$$\ell(\hM_{\USVT}(\tau), M) \leq C \frac{\rho k}{n},$$
with high probability.\footnote{For completeness, an in-expectation bound is established in Appendix \ref{eq:spec-sparse} for a spectral algorithm.}

The goal of this section is to show that the above rate cannot be improved by a polynomial-time algorithm. This is given by the following theorem.
\begin{Theorem} \label{th: sparse-graphon-estimation-final-lower-bound}
Suppose $2\leq k \leq \sqrt{n}$ and $\rho \geq \frac{ck^2}{n}$ for some small $0<c < 1$. Then for any $D \geq 1$, there exists a universal constant $c' > 0$ such that
	\begin{equation*}
		\inf_{\hM \in \bbR[A]^{n \times n}_{\leq D} } \sup_{M \in \cM_{k,\rho}} \bbE( \ell(M, \hM) ) \geq \frac{c'\rho k}{nD^4}. 
	\end{equation*}
\end{Theorem}

\subsection{Computational Lower Bound for Biclustering} \label{sec: biclustering}
Biclustering is another popular model of interest and has found a lot of applications in the literature \citep{hartigan1972direct,choi2014co,rohe2016co,chi2017convex,mankad2014biclustering}. Similar to SBM, many different problems have been considered for biclustering, such as recovery of the clustering structure, signal estimation and signal detection (detecting whether the signal matrix is zero or not). A line of early work has studied the statistical and computational limits for detection or recovery in biclustering with one planted cluster \citep{balakrishnan2011statistical,kolar2011minimax,butucea2013detection,butucea2015sharp,ma2015computational,cai2017computational,brennan2018reducibility,schramm2020computational} and their extensions to a growing number of clusters have been considered in \cite{chen2016statistical,cai2017computational,banks2018information,brennan2020reducibility,dadon2023detection}. In this section, we are more interested in the latter case.

Define the following parameter space of rectangular matrices with biclustering structure,
\begin{equation*}
	\begin{split}
		\cM_{k_1, k_2} = \Big\{ M \in \bbR^{n_1 \times n_2}: M_{ij} = Q_{z_i z_j} \text{ for some } Q \in \bbR^{k_1 \times k_2}, z_1 \in [k_1]^{n_1}, z_2 \in [k_2]^{n_2} 
		\Big\}.
	\end{split}
\end{equation*}
We observe $Y = M + E$, where $M \in \cM_{k_1, k_2}$ and $E$ has i.i.d. $N(0,1)$ entries. In this section, we are primarily interested in estimating $M$ given $Y$ and the loss of interest is  $\ell(\hM, M) = \frac{1}{n_1 n_2} \sum_{i \in [n_1], j \in [n_2] } (\hM_{ij} - M_{ij})^2 $. The minimax rate has been derived by \cite{gao2016optimal},
\begin{equation} \label{eq: biclustering-minimax}
	\inf_{\hM} \sup_{M \in \cM_{k_1, k_2}} \bbE \left( \ell( \hM,M) \right) \asymp \frac{k_1 k_2}{n_1 n_2} + \frac{\log k_1}{n_2} + \frac{\log k_2}{n_1},
\end{equation} and it is achieved by a constrained least-squares estimator that is computationally intractable. In terms of polynomial-time algorithms, a heuristic two-way extension of the Lloyd's algorithm has been proposed in \cite{gao2016optimal}, but there is no theoretical guarantee. Let us instead consider a simple spectral algorithm, 
\begin{equation*}
	\hM =\argmin_{M: \rank(M) \leq k_1 \wedge k_2 } \| Y - M\|_\F^2.
\end{equation*}
Its theoretical guarantee is given by the following result.
\begin{Proposition} \label{prop: biclustering-upper-bound}
	There exists $C> 0$ such that  $\sup_{M \in \cM_{k_1, k_2}} \bbE \left( \ell( \hM,M) \right)  \leq  C\frac{k_1\wedge k_2}{n_1 \wedge n_2}$.
\end{Proposition} 

Compared with the minimax rate \eqref{eq: biclustering-minimax}, the rate achieved by the spectral algorithm is not optimal. We will show that this rate is indeed the best one that can be achieved by a polynomial-time algorithm, at least in certain regimes of the problem. To this end, consider a subset of $\cM_{k_1,k_2}$, denoted by $\cM_{k_1, k_2, \lambda}$, whose definition is given by
\begin{equation*}
	\begin{split}
		\cM_{k_1, k_2, \lambda} = \Big\{ & M \in \bbR^{n_1 \times n_2}: M_{ij} = Q_{z_i z_j}\text{ for some }z_1 \in [k_1]^{n_1}, z_2 \in [k_2]^{n_2},\\
		& Q \in \bbR^{k_1 \times k_2} \text{ such that } Q_{ii} = \lambda \text{ for all } i \in [k_1 \wedge k_2 ]\text{ and }Q_{ij} = 0 \text{ otherwise }  \Big\}.
	\end{split}
\end{equation*}
We also consider a prior distribution $\bbP_{\BC (\lambda)}$ supported on $\cM_{k_1, k_2, \lambda}$. The sampling process $M\sim \bbP_{\BC (\lambda)}$ is given as follows: first, generate $z_1 \in [k_1]^{n}$, $z_2 \in [k_2]^{n}$ such that $(z_1)_i, (z_2)_j  \overset{i.i.d.} \sim \text{Unif}\{1, \ldots, k_1 \wedge k_2\}$ independently for all $i \in [n_1]$ and  $j \in [n_2]$; then let $M_{ij} = \lambda \indi( (z_1)_i = (z_2)_j ) $. The following result gives a lower bound for the class of low-degree polynomial algorithms.
\begin{Theorem} \label{th: biclustering-comp-lower-bound}
	For any $0 < r < 1$ and $D \geq 1$, if 
	\begin{equation} \label{eq: SNR-condition-biclustering}
		\lambda^2 \leq \frac{r}{(D(D+1))^2} \min \left( 1, \frac{k_1^2\wedge k^2_2}{n_1 \vee n_2} \right)
	\end{equation} holds, then
	\begin{equation}
		\inf_{ \hM \in \bbR[Y]_{\leq D}^{n_1 \times n_2} } \bbE_{Y, M \sim \bbP_{\BC (\lambda)} } ( \ell(\hM, M) ) \geq \frac{\lambda^2}{k_1\wedge k_2} - \left( \frac{\lambda^2}{k_1^2\wedge k_2^2} + \frac{r(2-r)\lambda^2}{(1-r)^2(n_1 \vee n_2)}  \right).\label{eq:bi-low-low-de}
	\end{equation}
\end{Theorem}

Since $\cM_{k_1, k_2, \lambda}\subset \cM_{k_1,k_2}$, the lower bound (\ref{eq:bi-low-low-de}) is also valid for $\cM_{k_1,k_2}$ under the SNR condition (\ref{eq: SNR-condition-biclustering}). To obtain a tight lower bound for $\cM_{k_1,k_2}$, we can maximize the right hand side of (\ref{eq:bi-low-low-de}) under the constraint (\ref{eq: SNR-condition-biclustering}). This leads to the following biclustering lower bound.
\begin{Corollary} \label{coro: biclustering-estimation-final-lower-bound}
Suppose $k_1 \wedge k_2 \geq 2$. For any $D \geq 1$, there exists a universal constant $c > 0$ such that
	\begin{equation*}
	\begin{split}
		\inf_{\hM \in \bbR[Y]^{n_1 \times n_2}_{\leq D} } \sup_{M \in \cM_{k_1, k_2}} \bbE( \ell(\hM,M) ) \geq \frac{c}{D^4} \left( \frac{k_1\wedge k_2}{n_1 \vee n_2} \wedge \frac{1}{k_1\wedge k_2} \right) .
	\end{split}
	\end{equation*}
\end{Corollary}

When $k_1 \wedge k_2 \leq \sqrt{n_1 \vee n_2}$ and $n_1 \asymp n_2$, the lower bound matches the rate of the spectral algorithm.

\section{Proofs of the Main Results in Section \ref{sec: computational-limit-graphon-estimation} } \label{sec: proof-of-main-result}
In this section, we provide proofs for the low-degree polynomial lower bound of graphon estimation in SBM. We will use Proposition \ref{prop: schramm-wein-binary} to prove our results and one of the key parts there is to understand $\kappa_{\balpha}(x, X)$. We will first introduce a few preliminary results regarding $\kappa_{\balpha}(x,X)$, then prove Theorem \ref{th: comp-limit-graphon-est} in Section \ref{sec: proof-main-theorem}, and finally prove Corollary \ref{coro: graphon-estimation-final-lower-bound} in Section \ref{sec: proof-main-corollary}.

Note that $\kappa_{\balpha}(x, X)$ depends on the prior of $X$. Now, let us introduce the following {\it uniform SBM prior with fixed first vertex}.
\begin{Definition} \label{def: uniform-prior-SBM} Consider a $k$-class SBM with $n$ vertices. We say $X \in \bbR^{n (n-1)/2}$ is drawn from the uniform SBM prior with fixed first vertex and parameter $\lambda > 0$ if it is generated as follows: (1) generate a membership vector $z \in [k]^{n}$ such that $z_1 = 1$, $z_j \overset{i.i.d.}\sim \text{Unif}\{1,\ldots, k\}$ for $j = 2,\ldots, n$; (2) let $X = \rmvec( \lambda Z_{ij}: i < j)$, where the symmetric matrix $Z \in \{0,1\}^{n \times n}$ is the corresponding membership matrix of $z$. Here the notation $\rmvec(Z_{ij}: i < j)$ means the vectorization of the upper triangular matrix of $Z$ by column.
\end{Definition}

Then we have the following bounds on $|\kappa_{\balpha}(x, X)|$ when $X$ is drawn from the prior in Definition \ref{def: uniform-prior-SBM}.
\begin{Proposition} \label{prop: kappa-property-SBM}
	Suppose $X$ is generated from the uniform SBM prior with fixed first vertex and parameter $\lambda$. Denote the membership vector of $X$ as $z$ and the first entry of $X$ as $x$, i.e., $x = \indi(z_1 = z_2)$. Then for any multigraph $\balpha$ on $X$ with $|\balpha| \geq 1$, we have
	\begin{itemize}
		\item (i) if $\balpha$ is a disconnected or $\balpha$ is connected but $2 \notin V(\balpha) $, then $\kappa_{\balpha}(x, X) = 0$;
		\item (ii) if $\balpha$ is connected, $2 \in V(\balpha) $ and $1 \notin V(\balpha)$, then $\kappa_{\balpha}(x, X) = 0$;
		\item  (iii) if $\balpha$ is connected, $2 \in V(\balpha) $ and $1 \in V(\balpha)$, then $|\kappa_{\balpha}(x, X)| \leq \lambda^{|\balpha|+1} (1/k)^{|V(\balpha)| - 1} (|\balpha| +1 )^{|\balpha|} $.
	\end{itemize}
\end{Proposition}
\begin{proof} Throughout the proofs, we will view $\balpha$ as a multigraph of $n$ vertices.

	 {\bf Proof of (i).}  By \eqref{eq: kappa-cumulant-connection} we know that $\kappa_{\balpha}(x, X)$ is the joint cumulant of a group of random variables, say $\cG$. For the case $\balpha$ is disconnected, $\cG$ could be divided into $\cG_1$ and $\cG_2$ and $\cG_1$ and $\cG_2$ are independent of each other. Thus, by Proposition \ref{prop: cumulant-independent-prop} in Appendix \ref{sec: cumulant}, $\kappa_{\balpha}(x, X)$ is zero. Similarly, for the case $\balpha$ is connected but $2 \notin V(\balpha)$, we know in the prior for $X$, $z_1$ is known and fixed, so if $2 \notin V(\balpha)$, $x$ will be independent of $X$. By the same argument, $\kappa_{\balpha}(x, X)$ will be zero.
	 
	 {\bf Proof of (ii).} First, for any connected $\balpha$, we have
	 \begin{equation} \label{eq: E-X-alpha-formula}
	 	\begin{split}
	 		\bbE[X^{\balpha}] &= \lambda^{|\balpha|} \bbP( \text{all vertices in } V(\balpha) \text{ are in the same community} ) = \lambda^{|\balpha|} \cdot \left(\frac{1}{k}\right)^{ |V(\balpha)| -1 }, \\
	 		\bbE[xX^{\balpha}] &=  \lambda^{|\balpha| + 1} \bbP( \text{all vertices in } V(\balpha) \cup \{1,2 \} \text{ are in the same community} ) \\
	 		& = \lambda^{|\balpha| + 1}\left(\frac{1}{k}\right)^{ |V(\balpha) \cup \{1,2 \}| -1 }.
	 	\end{split}
	 \end{equation}

	 Next, we prove the claim by induction. When $|\balpha| = 0$, $\kappa_0(x, X) = \bbE(x) = \frac{\lambda}{k}$. Then, for $\balpha$ such that $|\balpha| = 1$, $2 \in V(\balpha) $ and $1 \notin V(\balpha)$, we have 
	 \begin{equation*}
	 	\kappa_{\balpha}(x, X) \overset{ \eqref{eq: kappa-recursive-relation} } = \bbE[x X^{\balpha}] - \kappa_0(x, X) \bbE[X^{\balpha}] \overset{\eqref{eq: E-X-alpha-formula} } = \lambda^{|\balpha| + 1} \left(\frac{1}{k}\right)^{ |V(\balpha)| } - \frac{\lambda}{k} \lambda^{|\balpha|}\left(\frac{1}{k}\right)^{ |V(\balpha)| -1 }  =0.
	 \end{equation*}
	 
	 Now assume that given any $t \geq 2$ and any $\balpha$ with $2 \in V(\balpha) $, $1 \notin V(\balpha)$ and $|\balpha| < t$, $\kappa_{\balpha}(x, X) = 0$. Then for any such $\balpha$ with $|\balpha| = t$, we have
	 \begin{equation*}
	 	\begin{split}
	 		\kappa_{\balpha}(x, X) &\overset{ \eqref{eq: kappa-recursive-relation} } =  \bbE(x X^{\balpha}) - \sum_{0 \leq \bbeta \lneq \balpha } \kappa_{\bbeta}(x, X) {\balpha \choose \bbeta } \bbE[X^{\balpha - \bbeta}] \\
	 		& \overset{(a)}= \bbE(x X^{\balpha}) -\kappa_0(x, X) \bbE[X^{\balpha }] = \lambda^{|\balpha| + 1} \left(\frac{1}{k}\right)^{ |V(\balpha)| } - \frac{\lambda}{k} \lambda^{|\balpha|}\left(\frac{1}{k}\right)^{ |V(\balpha)| -1 }  =0,
	 	\end{split}
	 \end{equation*} where (a) is because for any $\bbeta$ such that $|\bbeta| \geq 1$, $1 \notin V(\bbeta)$ since $\bbeta$ is a subgraph of $\balpha$, and thus $\kappa_{\bbeta}(x, X) = 0$ for either the case $2 \notin V(\bbeta)$ by the result we have proved in part (i) and the case $2 \in V(\bbeta)$ by the induction assumption.
This finishes the induction, and we have that for any $\balpha$ such that $|\balpha| \geq 1$, $2 \in V(\balpha) $ and $1 \notin V(\balpha)$, $\kappa_{\balpha}(x, X) = 0$. 

{\bf Proof of (iii).} First, for any connected subgraph $\bbeta$ of $\balpha$,
\begin{equation}\label{eq: E-X-alpha-beta-formula}
\begin{split}
	 \bbE [X^{\balpha - \bbeta}] &= \lambda^{|\balpha - \bbeta|} \bbP(\text{each connected component in }\balpha - \bbeta \text{ belongs to the same community} ) \\
	& = \lambda^{|\balpha - \bbeta|} \left(\frac{1}{k}\right)^{ |V(\balpha-\bbeta)| - \cC(\balpha - \bbeta) },
\end{split}
\end{equation} where $\cC(\balpha - \bbeta)$ denotes the number of connected components in $\balpha - \bbeta$. 

Next, we prove the claim by induction. Recall that when $|\balpha| = 0$, $\kappa_0(x, X) = \frac{\lambda}{k}$. Then, for $\balpha$ such that $|\balpha| = 1$, $2 \in V(\balpha) $ and $1 \in V(\balpha)$, we have 
	 \begin{equation*}
	 \begin{split}
	 	\kappa_{\balpha}(x, X) &\overset{ \eqref{eq: kappa-recursive-relation} } = \bbE[x X^{\balpha}] - \kappa_0(x, X) \bbE[X^{\balpha}] \overset{\eqref{eq: E-X-alpha-formula} } = \lambda^{|\balpha| + 1} \left(\frac{1}{k}\right)^{ |V(\balpha)|-1 } - \frac{\lambda}{k} \lambda^{|\balpha|}\left(\frac{1}{k}\right)^{ |V(\balpha)| -1 } \\
	 	& = \lambda^{|\balpha| + 1} \left(\frac{1}{k}\right)^{ |V(\balpha)|-1 } - \frac{\lambda}{k} \lambda^{|\balpha|}\left(\frac{1}{k}\right)^{ |V(\balpha)| -1 } = \lambda^{|\balpha| + 1} \left(\frac{1}{k}\right)^{ |V(\balpha)|-1 } ( 1- 1/k ),
	 \end{split}
	 \end{equation*} thus $|\kappa_{\balpha}(x, X)| \leq \lambda^{|\balpha|+1} (1/k)^{|V(\balpha)| - 1} (|\balpha| +1 )^{|\balpha|}$ holds for $|\balpha| = 1$. 

Now assume that given any $t \geq 2$ and any $\balpha$ with $2 \in V(\balpha) $, $1 \in V(\balpha)$ and $|\balpha| < t$, $|\kappa_{\balpha}(x, X)| \leq \lambda^{|\balpha|+1} (1/k)^{|V(\balpha)| - 1} (|\balpha| +1 )^{|\balpha|}$. Then for any such $\balpha$ with $|\balpha| = t$, we have
\begin{equation} \label{ineq: SBM-alpha-prop3-main-arg1}
	\begin{split}
		|\kappa_{\balpha}(x, X)|  & \overset{ \eqref{eq: kappa-recursive-relation} } =  \big|\bbE(x X^{\balpha}) - \sum_{0 \leq \bbeta \lneq \balpha } \kappa_{\bbeta}(x, X) {\balpha \choose \bbeta } \bbE[X^{\balpha - \bbeta}]\big| \\
		& \leq \big|\bbE(x X^{\balpha})\big| + \sum_{0 \leq \bbeta \lneq \balpha } \big|\kappa_{\bbeta}(x, X) \big| {\balpha \choose \bbeta } \bbE[X^{\balpha - \bbeta}] \\
		& \overset{\text{Part }(i)}=  \big|\bbE(x X^{\balpha})\big| + \sum_{0 \leq \bbeta \lneq \balpha: \bbeta \text{ is connected }  } \big|\kappa_{\bbeta}(x, X) \big| {\balpha \choose \bbeta } \bbE[X^{\balpha - \bbeta}] \\
		& \overset{\eqref{eq: E-X-alpha-formula}, \eqref{eq: E-X-alpha-beta-formula} } = \lambda^{|\balpha| + 1} \left(\frac{1}{k}\right)^{ |V(\balpha)|-1 } + \sum_{0 \leq \bbeta \lneq \balpha: \bbeta \text{ is connected } } \big|\kappa_{\bbeta}(x, X) \big| {\balpha \choose \bbeta } \lambda^{|\balpha - \bbeta|} \left(\frac{1}{k}\right)^{ |V(\balpha-\bbeta)| - \cC(\balpha - \bbeta) } \\
		& \overset{(a)}= \lambda^{|\balpha| + 1} \left(\frac{1}{k}\right)^{ |V(\balpha)|-1 } +  \big|\kappa_{0}(x, X) \big|\lambda^{|\balpha|} \left(\frac{1}{k}\right)^{ |V(\balpha)| - 1 } \\
		 &+\sum_{ \substack{0 \lneq \bbeta \lneq \balpha,\\ \bbeta \text{ is connected },\\ 1 \in V(\bbeta), 2 \in V(\bbeta)} } \big|\kappa_{\bbeta}(x, X) \big| {\balpha \choose \bbeta } \lambda^{|\balpha - \bbeta|} \left(\frac{1}{k}\right)^{ |V(\balpha-\bbeta)| - \cC(\balpha - \bbeta) } \\
		 & \overset{(b)} = \lambda^{|\balpha| + 1} \left(\frac{1}{k}\right)^{ |V(\balpha)|-1 } + \lambda^{|\balpha| + 1} \left(\frac{1}{k}\right)^{ |V(\balpha)|} \\
		 & +\sum_{ \substack{0 \lneq \bbeta \lneq \balpha,\\ \bbeta \text{ is connected },\\ 1 \in V(\bbeta), 2 \in V(\bbeta)} } \lambda^{|\bbeta|+1} (1/k)^{|V(\bbeta)| - 1} (|\bbeta| +1 )^{|\bbeta|}  {\balpha \choose \bbeta } \lambda^{|\balpha - \bbeta|} \left(\frac{1}{k}\right)^{ |V(\balpha-\bbeta)| - \cC(\balpha - \bbeta) }.
	\end{split}
\end{equation} where in (a), we separate the term $\bbeta =0$ in the summation and then use the results proved in (i)(ii) of this proposition; (b) is because $\kappa_0(x, X) = \frac{\lambda}{k}$ and by the induction assumption. 

Next,
\begin{equation}\label{ineq: SBM-alpha-prop3-main-arg2}
	\begin{split}
		|\kappa_{\balpha}(x, X)|  \leq & \lambda^{|\balpha| + 1} \left(\frac{1}{k}\right)^{ |V(\balpha)|-1 } + \lambda^{|\balpha| + 1} \left(\frac{1}{k}\right)^{ |V(\balpha)|} \\
		 & +\sum_{ \substack{0 \lneq \bbeta \lneq \balpha,\\ \bbeta \text{ is connected },\\ 1 \in V(\bbeta), 2 \in V(\bbeta)} } \lambda^{|\bbeta|+1} (1/k)^{|V(\bbeta)| - 1} (|\bbeta| +1 )^{|\bbeta|}  {\balpha \choose \bbeta } \lambda^{|\balpha - \bbeta|} \left(\frac{1}{k}\right)^{ |V(\balpha-\bbeta)| - \cC(\balpha - \bbeta) } \\
		 & \overset{(a) }\leq  2\lambda^{|\balpha| + 1} \left(\frac{1}{k}\right)^{ |V(\balpha)|-1 } + \lambda^{|\balpha| + 1}\left(\frac{1}{k}\right)^{ |V(\balpha)|-1 }\sum_{ \substack{0 \lneq \bbeta \lneq \balpha,\\ \bbeta \text{ is connected },\\ 1 \in V(\bbeta), 2 \in V(\bbeta)} }  (|\bbeta| +1 )^{|\bbeta|}  {\balpha \choose \bbeta } \\
		 & \leq \lambda^{|\balpha| + 1} \left(\frac{1}{k}\right)^{ |V(\balpha)|-1 } \left(2 +  \sum_{ \substack{0 \lneq \bbeta \lneq \balpha} } (|\bbeta| +1 )^{|\bbeta|}  {\balpha \choose \bbeta } \right) \\
		 & \overset{(b)}= \lambda^{|\balpha| + 1} \left(\frac{1}{k}\right)^{ |V(\balpha)|-1 } \left(2 +  \sum_{ \ell=1 }^{|\balpha|-1} (\ell +1 )^{\ell}  {|\balpha| \choose \ell } \right) \\
		 & \leq \lambda^{|\balpha| + 1} \left(\frac{1}{k}\right)^{ |V(\balpha)|-1 } \left(2 +  \sum_{ \ell=1 }^{|\balpha|-1} |\balpha|^{\ell}  {|\balpha| \choose \ell } \right) \\
		 & \leq \lambda^{|\balpha| + 1} \left(\frac{1}{k}\right)^{ |V(\balpha)|-1 } \left(  \sum_{ \ell=0 }^{|\balpha|} |\balpha|^{\ell}  {|\balpha| \choose \ell } \right) \\
		 & = \lambda^{|\balpha| + 1} \left(\frac{1}{k}\right)^{ |V(\balpha)|-1 } \left(  |\balpha| + 1 \right)^{|\balpha|},
	\end{split}
\end{equation} where (a) is due to the following Lemma \ref{lm: graph-subgraph-vertex-connection} and in (b) we use the fact $\sum_{\bbeta: |\bbeta| = \ell} {\balpha \choose \bbeta} = {|\balpha| \choose \ell}$ for any $\ell \leq |\balpha|$.

\begin{Lemma} \label{lm: graph-subgraph-vertex-connection}
	Given any connected multigraph $\balpha$. Suppose $\bbeta$ is a connected subgraph of $\balpha$, then 
	\begin{equation*}
		|V(\balpha - \bbeta)| + |V(\bbeta)| - \cC(\balpha - \bbeta) \geq |V(\balpha)|,
	\end{equation*} where $\cC(\balpha - \bbeta)$ denotes the number of connected component in the graph $\balpha - \bbeta$.
\end{Lemma}
\begin{proof} First, it is clear that $V(\balpha - \bbeta) \cup V(\bbeta) \supseteq  V(\balpha)$. 
	Since both $\balpha$ and $\bbeta$ are connected multigraphs, for each connected component in $\balpha - \bbeta$, it must have at least a common vertex with $\bbeta$. Moreover, that common vertex is counted twice in computing $|V(\balpha - \bbeta)| + |V(\bbeta)|$. This finishes the proof.
\end{proof}
This finishes the induction and the proof of this proposition.
\end{proof}

The next lemma counts the number of connected $\balpha$ in (iii) of Proposition \ref{prop: kappa-property-SBM} such that $\kappa_{\balpha}(x,X)$ is nonzero.
\begin{Lemma}\label{lm: num-nonzero-alpha}
	Given any $d \geq 1$, $0 \leq h \leq d-1$, the number of connected $\balpha$ such that $1 \in V(\balpha)$, $2 \in V(\balpha)$, $|\balpha| = d$ and $|V(\balpha)| = d+1 -h$ is at most $n^{d-h-1} d^{d+h}$.
\end{Lemma}
\begin{proof}
	We view $\balpha$ as a multigraph on $[n]$ and count the number of ways to construct such $\balpha$. The counting strategy is the following: we start with adding Vertex $2$ to $\balpha$ and then add $(d-h)$ edges such that for each edge there, it will introduce a new vertex; then we add the rest of $h$ edges on these existing vertices. In the first stage above, we can also count different cases by considering when will Vertex $1$ be introduced in adding new vertices.
	\begin{itemize}
		\item If Vertex $1$ is the first vertex to be added after Vertex $2$, then the number of such choices of $\balpha$ is at most $(nd)^{d-h-1} (d^2)^h$. Here $(nd)^{d-h-1}$ is because for each of the rest of $d-h-1$ edges, the number of choices for the starting vertex is at most $d$ since there are at most $(d+1)$ vertices in $\balpha$ and the number of choices for a newly introduced vertex is at most $n$. $(d^2)^h$ comes from that in the second stage, the choice of each extra edge is at most ${(d+1) \choose 2} =(d+1)d/2 \leq d^2$.
		\item By the same counting strategy, if Vertex $1$ is the second vertex to be added in the first stage, then the number of such choices of $\balpha$ is at most $(nd)^{d-h-1} (d^2)^h$.
		\item $\cdots$
		\item If Vertex $1$ is the $(d-h)$-th vertex to be added in the first stage, then the number of such choices of $\balpha$ is at most $(nd)^{d-h-1} (d^2)^h$.
	\end{itemize}
	  By adding them together, the number of choices of connected $\balpha$ such that $1 \in V(\balpha)$, $2 \in V(\balpha)$, $|\balpha| = d$ and $|V(\balpha)| = d+1 -h$ is at most 
	  \begin{equation*}
	  	(d-h) (nd)^{d-h-1} (d^2)^h \leq d (nd)^{d-h-1} (d^2)^h = n^{d-h-1} d^{d+h}.
	  \end{equation*}
\end{proof}

In the following Proposition \ref{prop: sum-kappa-bound-SBM}, we bound $\sum_{\balpha \in \bbN^N, 0 \leq |\balpha| \leq D }  \frac{\kappa_{\balpha}^2(x,X)}{\balpha!}$ when $X$ is generated from the uniform SBM prior with fixed first vertex.

\begin{Proposition} \label{prop: sum-kappa-bound-SBM}
	Under the same setting as in Proposition \ref{prop: kappa-property-SBM}, for any $D \geq 1$, we have
	\begin{equation*}
		\sum_{\balpha \in \bbN^N, 0 \leq |\balpha| \leq D }  \frac{\kappa_{\balpha}^2(x,X)}{\balpha!} \leq \frac{\lambda^2}{k^2} - \frac{\lambda^2}{n} +\frac{\lambda^2}{n} \sum_{h=0}^d \left( D^2 (D+1)^2 \lambda^2 \right)^h \sum_{d=h}^D \left( D (D+1)^2 \frac{n \lambda^2}{k^2} \right)^{d-h}.
	\end{equation*}
In particular, for any $0 < r < 1$, if $\lambda^2 \leq \frac{r}{(D(D+1))^2} \min\left(1, \frac{k^2}{n} \right)$, then we have 
	\begin{equation*}
		\sum_{\balpha \in \bbN^N, 0 \leq |\balpha| \leq D }  \frac{\kappa_{\balpha}^2(x,X)}{\balpha!} \leq \frac{\lambda^2}{k^2} + \frac{r(2-r)\lambda^2}{(1-r)^2n}.
	\end{equation*}	
\end{Proposition}
\begin{proof}
	First, we have $\kappa_{0}(x,X) = \bbE(x) = \lambda/k $. Then 
	\begin{equation*}
		\begin{split}
			\sum_{\balpha \in \bbN^N, 0 \leq |\balpha| \leq D }  \frac{\kappa_{\balpha}^2(x,X)}{\balpha!} &\leq \sum_{\balpha \in \bbN^N, 0 \leq |\balpha| \leq D } \kappa_{\balpha}^2(x,X) \\
			& \overset{\text{Proposition } \ref{prop: kappa-property-SBM} (i)(ii) }= \kappa_{0}^2(x,X) + \sum_{ \substack{ \balpha \in \bbN^N, 1 \leq |\balpha| \leq D, \\ \balpha \text{ connected }, 1 \in V(\balpha), 2 \in V(\balpha) } } \kappa_{\balpha}^2(x,X) \\
			& \overset{\text{Proposition } \ref{prop: kappa-property-SBM}(iii), \text{Lemma } \ref{lm: num-nonzero-alpha} } \leq \frac{\lambda^2}{k^2} + \sum_{d=1}^D \sum_{h=0}^{d-1} n^{d-h-1} d^{d+h} \left( \lambda^{d+1} (1/k)^{d-h} (d+1)^d  \right)^2 \\
			& =  \frac{\lambda^2}{k^2} + \frac{\lambda^2}{n} \sum_{d=1}^D \sum_{h=0}^{d-1} \left( \frac{nd(d+1)^2 \lambda^2}{k^2} \right)^d \left( \frac{dk^2}{n} \right)^h \\
			& \leq \frac{\lambda^2}{k^2} + \frac{\lambda^2}{n} \sum_{d=1}^D \sum_{h=0}^{d-1} \left( \frac{nD(D+1)^2 \lambda^2}{k^2} \right)^d \left( \frac{Dk^2}{n} \right)^h \\
			& \leq \frac{\lambda^2}{k^2} - \frac{\lambda^2}{n} + \frac{\lambda^2}{n} \sum_{d=0}^D \sum_{h=0}^{d} \left( \frac{nD(D+1)^2 \lambda^2}{k^2} \right)^d \left( \frac{Dk^2}{n} \right)^h\\
			& = \frac{\lambda^2}{k^2} - \frac{\lambda^2}{n} +\frac{\lambda^2}{n} \sum_{h=0}^d \left( D^2 (D+1)^2 \lambda^2 \right)^h \sum_{d=h}^D \left( D (D+1)^2 \frac{n \lambda^2}{k^2} \right)^{d-h}.
		\end{split}
	\end{equation*} This shows the first conclusion.
	
	For any $0 < r < 1$ and $D \geq 1$, if $\lambda^2 \leq \frac{r}{(D(D+1))^2} \min\left(1, \frac{k^2}{n} \right)$, by the above result, we have
	\begin{equation*}
		\begin{split}
			\sum_{\balpha \in \bbN^N, 0 \leq |\balpha| \leq D }  \frac{\kappa_{\balpha}^2(x,X)}{\balpha!} &\leq \frac{\lambda^2}{k^2} - \frac{\lambda^2}{n} +\frac{\lambda^2}{n} \sum_{h=0}^d \left( D^2 (D+1)^2 \lambda^2 \right)^h \sum_{d=h}^D \left( D (D+1)^2 \frac{n \lambda^2}{k^2} \right)^{d-h} \\
			& \leq \frac{\lambda^2}{k^2} - \frac{\lambda^2}{n} +\frac{\lambda^2}{n} \sum_{h=0}^d r^h \sum_{d=h}^D r^{d-h}\leq \frac{\lambda^2}{k^2} - \frac{\lambda^2}{n} + \frac{\lambda^2}{n} \frac{1}{(1-r)^2}\\
			& = \frac{\lambda^2}{k^2} + \frac{r(2-r)\lambda^2}{(1-r)^2n}.
		\end{split}
	\end{equation*}
	This finishes the proof of this proposition.
\end{proof}

\subsection{Proof of Theorem \ref{th: comp-limit-graphon-est} } \label{sec: proof-main-theorem}
\begin{proof}
	First, since $M$ is drawn uniformly at random from $\cM_{k,p,q}$, by symmetry, we have
	\begin{equation} \label{ineq: graphon-main-bound1}
		\begin{split}
			&\quad \inf_{\hM \in \bbR[A]^{n \times n}_{\leq D} } \bbE_{A, M \sim \bbP_{\SBM (p,q)}} [\ell(\hM, M)] \\
			&= \inf_{g \in \bbR[A]_{\leq D} } \bbE_{A, M \sim \bbP_{\SBM (p,q)}} [ ( g(A) - M_{12} )^2 ] \\
			 &= \inf_{g \in \bbR[A]_{\leq D} } \sum_{j=1}^k\bbE_{A, M \sim \bbP_{\SBM (p,q)}} [ ( g(A) - M_{12} )^2 | (z_M)_1 = j ] \bbP((z_M)_1 = j) \\
			 & = \inf_{g \in \bbR[A]_{\leq D} } \bbE_{A, M \sim \bbP_{\SBM (p,q)}} [ ( g(A) - M_{12} )^2 | (z_M)_1 = 1 ]\\
			 &= \inf_{g \in \bbR[A]_{\leq D} } \bbE_{A, M \sim \bbP_{\SBM (p,q)}'} [ ( g(A) - M_{12} )^2 ],
		\end{split}
	\end{equation} where $\bbP_{\SBM (p,q)}'$ is the restriction of $\bbP_{\SBM (p,q)}$ on $\cM_{k,p,q}$ such that $(z_M)_1 = 1$.
	
	The graphon estimation problem in SBM fits in the general binary observation model described in Section \ref{sec: low-degree-preliminary}. Thus we can apply the results in Proposition \ref{prop: schramm-wein-binary}. Following the notation in Section \ref{sec: low-degree-preliminary}, in our context, we have $x = M_{12}$, $X = \rmvec(M_{ij}: i < j)$ encodes the upper triangular entries of $M$ and $Y = \rmvec(A_{ij}: i < j)$ encodes the upper triangular entries of $A$. Thus $N = n(n-1)/2$ and the law of $X$ is supported on $[q,p]$. By Proposition \ref{prop: schramm-wein-binary}, we have
	\begin{equation}\label{ineq: graphon-main-bound2}
		\begin{split}
			\inf_{g \in \bbR[A]_{\leq D} } \bbE_{A, M \sim \bbP_{\SBM (p,q)}'} [ ( g(A) - M_{12} )^2 ] &= \bbE_{A, M \sim \bbP_{\SBM (p,q)}'}(M_{12}^2) - \Corr_{\leq D}^2\\
			& \geq \bbE_{A, M \sim \bbP_{\SBM (p,q)}'}(M_{12}^2) - \sum_{\balpha \in \{0,1 \}^N, 0 \leq |\balpha| \leq D }  \frac{\kappa_{\balpha}^2(M_{12}, X)}{( q (1-p) )^{|\balpha|}}
		\end{split}
	\end{equation} where $\kappa_{\balpha}(M_{12}, X)$ is recursively defined as 
	\begin{equation} \label{eq: cumulant-SBM-graphon}
		\begin{split}
			\kappa_0(M_{12}, X) &= \bbE_{A, M \sim \bbP_{\SBM (p,q)}'} [M_{12}] = p/k + (1-1/k)q = q + (p-q)/k;\\
			\kappa_{\balpha}(M_{12}, X) &= \bbE_{A, M \sim \bbP_{\SBM (p,q)}'}[M_{12} X^{\balpha} ] - \sum_{0 \leq \bbeta \lneq \balpha } \kappa_{\bbeta}(x, X) {\balpha \choose \bbeta } \bbE[X^{\balpha - \bbeta}]\\
			& \overset{(a)}= \bbE_{A, M \sim \bbP_{\SBM (p,q)}'}[M_{12} X^{\balpha} ] - \sum_{0 \leq \bbeta \lneq \balpha } \kappa_{\bbeta}(x, X)  \bbE[X^{\balpha - \bbeta}],\text{ for } \balpha \text{ such that } |\balpha| \geq 1
		\end{split}
	\end{equation} where in (a), we use the fact $\balpha \in \{0,1 \}^N$ so that ${\balpha \choose \bbeta }  = 1$.
	
	Directly computing $\kappa_{\balpha}(M_{12}, X)$ is complicated, here we do a transformation and let $\widebar{X} = (X-q)/\sqrt{q(1-p)}$, $\widebar{M}_{12} = (M_{12} - q)/\sqrt{q(1-p)}$. By the interpretation of $\kappa_{\balpha}(x, X)$ in \eqref{eq: kappa-cumulant-connection}, we have
	\begin{equation} \label{eq: kappa0-bound1}
	\begin{split}
		\kappa_0(\widebar{M}_{12}, \widebar{X}) &\overset{(a) }= \kappa_0\left(\frac{M_{12}}{ \sqrt{q(1-p)} }, X\right) - \frac{q}{\sqrt{q(1-p)}} \\
		& = \frac{1}{\sqrt{q(1-p)}} \kappa_0\left(M_{12}, X\right) - \frac{q}{\sqrt{q(1-p)}} \\
		& \overset{ \eqref{eq: cumulant-SBM-graphon} }= \frac{p-q}{k \sqrt{q(1-p)}},
	\end{split}
	\end{equation} where (a) is by Proposition \ref{prop: cumulant-scaling-prop} in Appendix \ref{sec: cumulant}.
	For any $\balpha$ such that $|\balpha| \geq 1$, by Proposition \ref{prop: cumulant-scaling-prop}, we have
	\begin{equation}\label{eq: kappa0-bound2}
		\kappa_{\balpha}(\widebar{M}_{12}, \widebar{X}) = \frac{1}{ (q(1-p) )^{(|\balpha | + 1)/2 } } \cdot \kappa_{\balpha}(M_{12}, X).
	\end{equation}
	
	By plugging \eqref{eq: kappa0-bound1} and \eqref{eq: kappa0-bound2} into \eqref{ineq: graphon-main-bound1} and \eqref{ineq: graphon-main-bound2}, we have
	\begin{equation} \label{ineq: graphon-lower-bound-arg1}
		\begin{split}
			&\inf_{g \in \bbR[A]_{\leq D} } \bbE_{A, M \sim \bbP_{\SBM (p,q)}'} [ ( g(A) - M_{12} )^2 ] \\
			&\geq \bbE_{A, M \sim \bbP_{\SBM (p,q)}'}(M_{12}^2) - \sum_{\balpha \in \{0,1 \}^N, 0 \leq |\balpha| \leq D }  \frac{\kappa_{\balpha}^2(M_{12}, X)}{( q (1-p) )^{|\balpha|}} \\
			& = q^2 + \frac{p^2 - q^2}{k} -  \kappa^2_0(M_{12}, X) - \sum_{\balpha \in \{0,1 \}^N, 0 \lneq |\balpha| \leq D }  \frac{\kappa_{\balpha}^2(M_{12}, X)}{( q (1-p) )^{|\balpha|}} \\
			& \overset{ \eqref{eq: cumulant-SBM-graphon}, \eqref{eq: kappa0-bound2} }= q^2 + \frac{p^2 - q^2}{k} - (q + (p-q)/k )^2 + q(1-p) \kappa^2_0(\widebar{M}_{12}, \widebar{X}) \\
			& \quad \quad \quad - \sum_{\balpha \in \{0,1 \}^N, 0 \leq |\balpha| \leq D }  q(1-p)\kappa_{\balpha}^2(\widebar{M}_{12}, \widebar{X}) \\
			& \overset{\eqref{eq: kappa0-bound1}}= \frac{(p-q)^2}{k} - q(1-p)\sum_{\balpha \in \{0,1 \}^N, 0 \leq |\balpha| \leq D }  \kappa_{\balpha}^2(\widebar{M}_{12}, \widebar{X}).
		\end{split}
	\end{equation}
	
	Recall that $Z_M$ is the membership matrix associated with $M$. Since $X$ represents the upper triangular entries of $q \1_{n} \1_{n}^\top + (p-q) Z_M $ where $\1_{n }$ represents an all $1$ vector, after the transformation, $\widebar{X}$ encodes upper triangular entries of $ \frac{(p-q) Z_M}{\sqrt{q(1-p)}}$ and $\widebar{M}_{12}$ is the first entry of $\widebar{X}$. 
	
	Notice that $\bbP_{\SBM (p,q)}'$ is exactly the uniform SBM prior with fixed first vertex defined in Definition \ref{def: uniform-prior-SBM} with $\lambda = \frac{(p-q)}{\sqrt{q(1-p)}}$, by Proposition \ref{prop: sum-kappa-bound-SBM}, we have
	\begin{equation}
	\begin{split}
		\inf_{\hM \in \bbR[A]^{n \times n}_{\leq D} } \bbE_{A, M \sim \bbP_{\SBM (p,q)}} [\ell(\hM, M)] &\overset{ \eqref{ineq: graphon-main-bound1} }=\inf_{g \in \bbR[A]_{\leq D} } \bbE_{A, M \sim \bbP_{\SBM (p,q)}'} [ ( g(A) - M_{12} )^2 ] \\
		& \overset{ \eqref{ineq: graphon-lower-bound-arg1} } \geq  \frac{(p-q)^2}{k} - q(1-p)\sum_{\balpha \in \{0,1 \}^N, 0 \leq |\balpha| \leq D }  \kappa_{\balpha}^2(\widebar{M}_{12}, \widebar{X}) \\
		& \overset{(a)}= \frac{(p-q)^2}{k} - q(1-p)\sum_{\balpha \in \{0,1 \}^N, 0 \leq |\balpha| \leq D }  \kappa_{\balpha}^2(\widebar{M}_{12}, \widebar{X})/\balpha! \\
		& \overset{ \text{Proposition } \ref{prop: sum-kappa-bound-SBM} }\geq \frac{(p-q)^2}{k} - q(1-p) \cdot \left( \frac{(p-q)^2}{k^2 q(1-p)} + \frac{r(2-r)(p-q)^2}{(1-r)^2 n q(1-p)}  \right) \\
		& = \frac{(p-q)^2}{k} - (p-q)^2 \left( \frac{1}{k^2} + \frac{r(2-r)}{(1-r)^2 n}  \right).
	\end{split}
	\end{equation} where (a) is because $\balpha \in \{0,1 \}^N$. This finishes the proof of this theorem.
\end{proof}

\subsection{Proof of Corollary \ref{coro: graphon-estimation-final-lower-bound} } \label{sec: proof-main-corollary}
\begin{proof}
	Since $k \geq 2$ and $n \geq k^2 \geq 2k$, by Theorem \ref{th: comp-limit-graphon-est} we have there exists a small enough $r > 0$ such that when $\frac{(p-q)^2}{q(1-p)} \leq \frac{r}{(D(D+1))^2} \min(1,\frac{k^2}{n}) $, we have
	\begin{equation} \label{ineq:coro-mid-ineq}
		\inf_{\hM \in \bbR[A]^{n \times n}_{\leq D} } \bbE_{A, M \sim \bbP_{\SBM (p,q)}} \ell(\hM, M) \geq c_r \frac{(p-q)^2}{k}
	\end{equation} for some constant $c_r > 0$ depends on $r$ only. 
	
	Then, we take $\epsilon \leq q \leq p \leq 1-\epsilon$ for some $\epsilon > 0$ such that $\frac{(p-q)^2}{q(1-p)} \geq c \frac{r}{(D(D+1))^2} \min(1,\frac{k^2}{n})$ for some $1> c > 0$, and we have
	\begin{equation*}
	\begin{split}
		\inf_{\hM \in \bbR[A]^{n \times n}_{\leq D} } \sup_{M \in \cM_k} \bbE( \ell(\hM,M) ) & \geq \inf_{\hM \in \bbR[A]^{n \times n}_{\leq D} } \bbE_{A, M \sim \bbP_{\SBM (p,q)}} \ell(\hM, M) \\
		&\overset{ \eqref{ineq:coro-mid-ineq} }\geq c_r \frac{(p-q)^2}{k} = c'_r \frac{r}{(D(D+1))^2} q(1-p) \left( \frac{k}{n} \wedge \frac{1}{k} \right) \\
		& \geq \frac{c}{D^4} \left( \frac{k}{n} \wedge \frac{1}{k} \right),
	\end{split}
	\end{equation*} where $c$ depends on $\epsilon$ and $r$ only.
\end{proof}

\vskip 0.2in
\bibliographystyle{apalike}
\bibliography{reference}

		\newpage
		\appendix
		
			\begin{center}
			{\LARGE Supplement to "Computational Lower Bounds for Graphon Estimation via Low-degree Polynomials"	
				
			}		
			\medskip
			
			{\large Yuetian Luo \quad and\quad Chao Gao}
			\medskip
		\end{center}
		
		In this supplement, we provide a table of contents and the rest of the technical proofs.
		\tableofcontents

\section{Two Natural Hypothesis Testing Problems Associated with Graphon Estimation/Biclustering} \label{sec: testing-problem-no-hardness}
In this section, we illustrate that two natural testing problems associated with graphon estimation do not have computational barriers. For simplicity, we focus on the additive Gaussian noise model, i.e., the biclustering setting in Section \ref{sec: biclustering}, while we believe similar results hold in the binary model as well. Suppose $n$ divides $k$ and $k$ is allowed to grow with respect to $n$. Let $z \in [k]^{n}$ be a balanced partition of $[n]$ with equal size in each community, i.e., $z^{-1}(t) = n/k$ for all $t \in [k]$, chosen uniformly at random from such partitions. Let $z'$ be independently drawn from the same distribution as $z$. Construct $Z \in \{0,1 \}^{n \times n} $ such that $Z_{ij} = \indi( z_i = z'_j )$ for any $i,j \in [n]$. Then we consider the following two testing problems:
\begin{equation} \label{eq: global-testing-problems}
	\begin{split}
		\text{(Problem 1)}& \quad H_0: Y = E \quad \text{ v.s. } \quad H_1: Y = \lambda Z + E, \\
		\text{(Problem 2)}& \quad H_0: Y = E \quad \text{ v.s. } \quad H_1: Y = \lambda (Z -\frac{1}{k} \1_n \1_n^\top)   + E ,
	\end{split}
\end{equation}  where $\lambda > 0$ is the signal strength, $\1_n \in \bbR^n$ is an all one vector and $E$ has i.i.d. $N(0,1)$ entries. 

Problem 1 and Problem 2 have the same null distribution, but under $H_1$ we have $\bbE (\sum_{i,j} Y_{ij} ) > 0$ in the first problem and $\bbE (\sum_{i,j} Y_{ij} ) = 0$ in the second problem. Both Problem 1 and Problem 2 have been considered in the literature and it turns out both of them do not exhibit statistical-computational gaps. In particular, the statistical separation rate of Problem 1 is $\lambda \asymp \frac{k}{n} $ (see Table 1 in \cite{dadon2023detection} with the parameter regime $m = n/k$, notice that the notation $k$ there is our notation $n/k$) and the statistical separation rate of Problem 2 is $\lambda  \asymp \sqrt{ \frac{k}{n}}$ (see Theorem 2 in \cite{banks2018information} and Lemma 14.8 in \cite{brennan2020reducibility} with the parameter regime $n = rK$ there and the discussion afterward). Moreover, for both problems, statistically optimal testing procedures can be computed within polynomial time. Thus, we cannot hope to obtain the computational hardness of biclustering estimation by considering the above two testing problems. However, it was conjectured in \cite{chen2016statistical,brennan2020reducibility} that there is still a statistical-computational gap for estimation/recovery in both problems and when $k \leq \sqrt{n}$, the conjectured computational threshold for estimation/recovery is $\lambda \asymp \frac{k}{\sqrt{n}}$ in our notation while the statistical threshold is $\lambda \asymp \sqrt{ \frac{k}{n} } $, see Conjecture 19 in \cite{chen2016statistical}, the discussion after Lemma 14.8 in \cite{brennan2020reducibility} and Table 1 in \cite{dadon2023detection}.

\section{More Background and Preliminaries for Low-degree Polynomials} \label{sec: preliminary-low-degree}

\subsection{Cumulant and its Properties} \label{sec: cumulant}
The key quantity $\kappa_{\balpha}$ is closely related to the cumulants. Here we provide a brief overview of cumulants and their basic properties, which will be useful in deducing some properties for $\kappa_{\balpha}$ in the graphon estimation problem. Most of the materials in this section are covered in \cite[Section 2.5]{schramm2020computational}. 

\begin{Definition}
	Let $X_1, \ldots, X_n$ be jointly distributed random variables. Their cumulant generating function is 
	\begin{equation*}
		K(t_1, \ldots, t_n) = \log \bbE \left[ \exp\left( \sum_{i=1}^n t_i X_i \right) \right],
	\end{equation*} and their joint cumulant is the quantity
	\begin{equation*}
		\kappa(X_1, \ldots, X_n) = \left( \left( \prod_{i=1}^n \frac{\partial}{\partial t_i} K(t_1,\ldots, t_n)  \right)  \right) \Big|_{t_1= \cdots = t_n = 0}.
	\end{equation*}
\end{Definition}

The following two properties of cumulant will be frequently used in our proofs.
\begin{Proposition} \label{prop: cumulant-independent-prop}
	If $a, b $ are two positive integers and $X_1, \ldots, X_a, Y_1, \ldots, Y_b$ are random variables with $\{X_i \}_{i \in [a]}$ independent from $\{Y_j \}_{j \in [b]}$, then
	\begin{equation*}
		\kappa(X_1,\ldots, X_a, Y_1, \ldots, Y_b) = 0.
	\end{equation*}
\end{Proposition}

\begin{Proposition}\label{prop: cumulant-scaling-prop}
	The joint cumulant is invariant under constant shifts and is scaled by constant multiplication. That is, if $X_1,\ldots, X_n$ are jointly-distributed random variables and $c$ is any constant, then
	\begin{equation*}
		\kappa(X_1 + c, X_2, \ldots, X_n) = \kappa(X_1, \ldots, X_n) + c \cdot \indi(n = 1),
	\end{equation*} and
	\begin{equation*}
		\kappa(cX_1 , X_2, \ldots, X_n) = c \cdot \kappa(X_1 , X_2, \ldots, X_n).
	\end{equation*}
\end{Proposition}

\subsection{Computational Lower Bounds for Estimation Under the Additive Gaussian Observation Model} \label{sec: comp-lower-bound-gaussian-model}
For the general additive Gaussian noise model, \cite{schramm2020computational} also provided a low-degree polynomial lower bound for estimation for low-degree polynomials. Given any positive integer $N$, suppose we observe $Y = X + E \in \bbR^N$, where $X \in \bbR^N$ is drawn from an arbitrary (but known) prior and $E$ has i.i.d. $N(0,1)$ entries independent from $X$. The goal is again to estimate a scalar quantity $x \in \bbR$, which is a function of $X$. Then we have a similar result as in Proposition \ref{prop: schramm-wein-binary}.
\begin{Proposition}[\cite{schramm2020computational}] \label{prop: schramm-wein-gaussian}
	In the general additive  Gaussian noise model described above, for any $D \geq 1$, we have 
	\begin{equation*}
	\inf_{g \in \bbR[Y]_{\leq D}} \bbE_{(x, Y) \sim \bbP} (g(Y) - x )^2 = \bbE(x^2) - \Corr_{\leq D}^2,
\end{equation*} where $\Corr_{\leq D}$ is defined as in the same way as in \eqref{def: Corr} and satisfies
\begin{equation*}
	\Corr_{\leq D}^2 \leq \sum_{ \balpha \in \bbN^N, 0 \leq | \balpha | \leq D } \frac{\kappa_{\balpha}^2(x, X)}{ \balpha! }.
\end{equation*} Here $\kappa_{\balpha}(x, X)$ is also defined recursively in the same way as in \eqref{eq: kappa-recursive-relation}. 
\end{Proposition}

\section{Additional Proofs in Section \ref{sec: computational-limit-graphon-estimation} }
\subsection{Guarantee for $\hM_{\mean}$ when $k \geq \sqrt{n}$} \label{app: guarantee-of-M-mean}
	For notation simplicity, let us denote $\widebar{A} = \sum_{1 \leq i < j \leq n} A_{ij}/{n \choose 2} $ and $\widebar{M} =\sum_{1 \leq i < j \leq n} M_{ij}/{n \choose 2} $. Then for $1 \leq i < j \leq n$, we have there exists $C_1, C_2 > 0$ such that
	\begin{equation*}
		\begin{split}
			\bbE[( \widebar{A} - M_{ij} )^2] &= \bbE[( \widebar{A} - \widebar{M} )^2] + \bbE[( \widebar{M} - M_{ij} )^2] \\
			&\leq \left\{ \begin{array}{cc}
				\frac{C_1}{n^2} + C_2 (p-q)^2 & \text{ if } M_{ij} = p,\\
				\frac{C_1}{n^2} + \frac{C_2}{k^2} & \text{ if } M_{ij} = q.
			\end{array}  \right.
		\end{split}
	\end{equation*} Thus,
	\begin{equation*}
		\begin{split}
			& \quad \frac{1}{{n \choose 2}}\bbE \sum_{1 \leq i < j \leq n}  [( \widebar{A} - M_{ij} )^2] \\
			&\leq  \frac{1}{{n \choose 2}} \left( (\frac{C_1}{n^2} + C_2 (p-q)^2) \frac{(n/k - 1) n}{2} + (\frac{C_1}{n^2} + \frac{C_2}{k^2}) ( \frac{n(n-1)}{2} -  \frac{(n/k - 1) n}{2}  )  \right)  \\
			& \leq \frac{C'}{n^2} + \frac{C''}{k} + \frac{C'''}{k^2} \leq C/k.
		\end{split}
	\end{equation*}

\subsection{Proof of Theorem \ref{th:low-degree-up-bound}} \label{sec:low-degree-upbd-lemma}

Without loss of generality, we assume $p \geq q$, while the proof still go through if $p < q$.	Given the membership vector $z_M$, let us also introduce $\Pi_{M} \in \{0,1 \}^{n \times k}$ where $(\Pi_M)_{ij} = 1$ if $(z_M)_i = j$ and $(\Pi_M)_{ij} = 0$ otherwise. With this notation, we can write $M = q \1_n \1_n^\top + (p-q) \Pi_{M} \Pi_{M}^\top - p I_n$. Let 
	\begin{equation} \label{eq:equiva-rela}
		\begin{split} 
		&\widebar{M} = M + pI_n, \quad \tM = \widebar{M} - q \1_n \1_n^\top = (p-q) \Pi_{M} \Pi_{M}^\top, \\
			&\widebar{A} = A + \Lambda, \quad \widebar{Z} = \widebar{A} - \widebar{M} = \tA - \tM. 
		\end{split}
	\end{equation} Notice that $\tM$ is a rank $k$ matrix and let us denotes its SVD as $U^* \Sigma^* U^{* \top}$, where $U^* \in \mathbb{O}_{n,k}$ ($\mathbb{O}_{n,k}$ denotes the set of $n$-by-$k$ column orthonormal matrices) and the diagonal entries of $\Sigma^*$ encode the number of member in each individual communities. 
Then
	\begin{equation*}
	\begin{split}
		\ell(\hM, M) = \frac{1}{ {n \choose 2} } \sum_{i < j} (\hM_{ij} - M_{ij})^2 \leq \frac{1}{ {n \choose 2} } \| \hM - \widebar{M}\|_\F^2 = \frac{1}{ {n \choose 2} } \| \tA^{t_1} B W_{t_2} - \tM\|_\F^2.
	\end{split}
	\end{equation*}
So to prove the result, we just need to show $\| \tA^{t_1} B W_{t_2} - \tM\|_\F^2 \leq C n (\log n + k) \log^2 n$.

Suppose $\tA$ has SVD $\tA = U_1 \Sigma_1 V_1^\top + U_2 \Sigma_2 V_2^\top$, where $U_1,V_1 \in \mathbb{O}_{n,k}, U_2,V_2 \in \mathbb{O}_{n,n-k}$, $\Sigma_1 \in \bbR^{k \times k}$ and $\Sigma_2 \in \bbR^{(n-k) \times (n-k)}$. By standard concentration results, we have the following events hold simultaneously with probability $1 - n^{-\bar{C}}$ (where the randomness is taken with respect to $M,A,\Lambda$ and $B$):
\begin{itemize}
	\item (A1) $\forall j \in [k], |\sum_{i=1}^n \indi ((z_M)_i = j) - \frac{n}{k} | \leq C_1 \sqrt{ \frac{n}{k} \log n }$, i.e., $|\Sigma^*_{jj} - \frac{(p-q)n}{k} | \leq C_1 (p-q)\sqrt{ \frac{n}{k} \log n }$;
	\item (A2) $\|\widebar{Z}\| \leq C_2 \sqrt{n}$;
	\item (A3) $c_3 \sqrt{k} \leq \sigma_k( V_1^\top B) \leq \sigma_1(V_1^\top B) \leq C_3 \sqrt{k} $ and $\|V_2^\top B\| \leq C_3 \sqrt{n}$,
\end{itemize} where $\sigma_i(\cdot)$ denotes the $i$-th largest singular value of a given matrix and $C_1, C_2, C_3 > 0$. Here (A1) is by the Bernstein inequality for the sum of Bernoulli random variables; (A2) is by the concentration of spectrum of the noise matrix in sparse random graphs, e.g., see \cite{feige2005spectral} and Lemma 2 in \cite{xu2018rates} with $\rho = 1$; (A3) is by the standard concentration result for the spectrum of Gaussian random matrices, see \cite{vershynin2010introduction}. Given (A1) and (A2) hold, we also have 
\begin{itemize}
	\item (A4) $\|U_1 \Sigma_1 V_1^\top - \tM\|_\F^2 \leq C_4 k n$;
	\item (A5) $\|\Sigma_2\| \leq \min_{X \text{ is of rank bounded by } k} \|\tA - X\| \leq \|\tA - \tM\| = \|\barZ\| \leq C_2 \sqrt{n}$;
	\item (A6) $\|\Sigma_1\| \leq \|\tM\| + \|\barZ\| \leq (p-q)\frac{n}{k} ( 1 + C_1 \sqrt{ \frac{k}{n} \log n } ) + C_2 \sqrt{n}$,
\end{itemize} where (A4) is by a similar argument as in Proposition \ref{prop: biclustering-upper-bound}.

The rest of the argument will be a deterministic argument given the above (A1)-(A6) events hold. It consists of two parts: statistical error analysis and computational error analysis.
\begin{itemize}
	\item  In the statistical error analysis, we want to show the estimation error is small if we can solve the least squares problem exactly. Specifically, suppose $\hW = \argmin_{W \in \bbR^{r \times n}} \|\tA^{t_1} B W - \tA\|_\F^2$, we want to show
	\begin{equation} \label{ineq:stat-error}
		\|\tA^{t_1} B \hW - \tM\|^2_\F \lesssim nk,
	\end{equation} where the notation $a_n \lesssim b_n$ where it means that $a_n \leq C b_n$ for some constant $C >0$ independent of $n$.
	\item In the computational error analysis, we want to show GD approximates the least-squares solution well, i.e., 
	\begin{equation}\label{ineq:comp-error}
		\|\tA^{t_1} B W_{t_2} - \tA^{t_1} B \hW\|_\F^2 \lesssim n (\log n +k) \log^2 n.
	\end{equation}
\end{itemize} Once we show \eqref{ineq:stat-error} and \eqref{ineq:comp-error}, they imply $\|\tA^{t_1} B W_{t_2} - \tM\|_\F^2 \lesssim n (\log n +k) \log^2 n$ and we are done. 
One challenge in performing our analysis is that since we do not assume the spectrum gap, we have to consider both scenarios (with a spectrum gap and without a spectrum gap) at the same time. 
 \vskip.3cm
{\noindent \bf (Statistical error analysis).} Since $\tA = \tM + \barZ$, by Lemma \ref{lm:ls-bound} in Appendix \ref{sec:low-degree-upbd-lemma}, we have for any $W \in \bbR^{r \times n}$,
\begin{equation*}
	\begin{split}
		\|\tA^{t_1} B \hW - \tM\|^2_\F \leq \frac{20}{9} \|\tA^{t_1} B W - \tM\|^2_\F + \frac{110}{9}  \| \barZ_{\max(r)} \|_\F^2.
	\end{split}
\end{equation*} Since $\| \barZ_{\max(r)} \|_\F^2 \leq r \|\barZ\|^2 \lesssim nk$ by (A2), we just need to show $\|\tA^{t_1} B W - \tM\|^2_\F \lesssim nk$. We consider two different cases:
\begin{itemize}
	\item (Case 1: $\frac{(p-q)n}{k} \left( 1 - C_1 \sqrt{ \frac{k}{n} \log n } \right) \geq 5 C_2 \sqrt{n}$) This is the high SNR regime, we have a spectrum gap. Specifically, 
	\begin{equation} \label{ineq:simga1-lower-bound}
		\sigma_k(\Sigma_1 ) \geq \sigma_k(\Sigma^*) - \|\barZ\| \overset{\text{(A1)}}\geq \frac{(p-q)n}{k} \left( 1 - C_1 \sqrt{ \frac{k}{n} \log n } \right) - C_2 \sqrt{n} \geq 4 C_2 \sqrt{n}.  
	\end{equation} 
	
	Then we take $W = (V_1^\top B)^\dagger \Sigma_1^{-t_1 + 1}V_1^\top$. So
	\begin{equation*}
		\begin{split}
			\|\tA^{t_1} B \hW - \tM\|^2_\F &= \|U_1 \Sigma^{t_1}_1 V_1^\top B W + U_2 \Sigma^{t_1}_2 V_2^\top B W - \tM\|_\F^2 \\
			& = \|U_1 \Sigma_1 V_1^\top - \tM + U_2 \Sigma^{t_1}_2 V_2^\top B (V_1^\top B)^\dagger \Sigma_1^{-t_1 + 1}V_1^\top\|_\F^2 \\
			& \leq 2 \|U_1 \Sigma_1 V_1^\top - \tM\|_\F^2 + 2 \| U_2 \Sigma^{t_1}_2 V_2^\top B (V_1^\top B)^\dagger \Sigma_1^{-t_1 + 1}V_1^\top\|_\F^2.
		\end{split}
	\end{equation*} Notice that $\|U_1 \Sigma_1 V_1^\top - \tM\|_\F^2  \lesssim nk$ by (A4) and 
	\begin{equation*}
		\begin{split}
			& \| U_2 \Sigma^{t_1}_2 V_2^\top B (V_1^\top B)^\dagger \Sigma_1^{-t_1 + 1}V_1^\top\|_\F^2 \\
			\leq & k \| U_2 \Sigma^{t_1}_2 V_2^\top B (V_1^\top B)^\dagger \Sigma_1^{-t_1 + 1}V_1^\top\|^2\\
			\leq & k \left( \|\Sigma_2\|^{t_1} \|V_2^\top B\| \frac{1}{\sigma_k(V_1^\top B) \sigma_k(\Sigma_1)^{t_1 -1} } \right)^2 \\
			\overset{(a)}\leq & k \left( \frac{(C_2\sqrt{n})^{t_1} C_3 \sqrt{n} }{c_3\sqrt{k} (4C_2 \sqrt{n} )^{t_1 - 1}  } \right)^2 \leq \left(\frac{ C_2 C_3 n}{c_3} \right)^2 (\frac{1}{16} )^{t_1 - 1} \\
			\overset{(b)}\lesssim & nk,
		\end{split}
	\end{equation*} where (a) is due to (A5),(A3),\eqref{ineq:simga1-lower-bound}; (b) is because $t_1 = C' \log n$ for large $C' > 0$.

	\item (Case 2: $\frac{(p-q)n}{k} \left( 1 - C_1 \sqrt{ \frac{k}{n} \log n } \right) \leq 5 C_2 \sqrt{n}$) This is the low SNR regime, we do not have a spectrum gap. Since $n \geq C k \log^3 n$, we have $$\|\Sigma^*\| \leq \frac{(p-q)n}{k} \left( 1 + C_1 \sqrt{ \frac{k}{n} \log n } \right) \leq 6 C_2 \sqrt{n}.$$
 We take $W = 0$, so 
	\begin{equation*}
		\begin{split}
			\|\tA^{t_1} B W - \tM\|^2_\F = \| \tM\|^2_\F \leq k \|\Sigma^*\|^2 \lesssim nk.
		\end{split}
	\end{equation*}
\end{itemize} So we have finished the proof for the statistical error analysis.

 \vskip.3cm
{\noindent \bf (Computational error analysis).} Suppose we denote the SVD of $\tA^{t_1} B$ as follows
\begin{equation*}
	\begin{split}
		\tA^{t_1} B = P S Q^\top = [P_1 \, P_2] \left[\begin{array}{cc}
			S_1 & \0 \\
			\0 & S_2
		\end{array} \right] \left[\begin{array}{c}
			Q_1^\top \\
			Q_2^\top
		\end{array} \right],
	\end{split}
\end{equation*} where $P \in \bbO_{n,r}, Q \in \bbO_{r,r}, S \in \bbR^{r \times r}$. Without loss of generality, we assume $\tA^{t_1} B$ has rank $r$, while the same analysis still go through when $\tA^{t_1} B$ has rank less than $r$, then we just need to use its economic SVD. Let $P_1 S_1 Q_1^\top$ consists the top $k$ components of $P S Q^\top$ and $P_2 S_2 Q_2^\top$ consists the rest of the components. First, we can write down a close formula for the least-squares solution $\hW = (B^\top \tA^{2t_1} B )^{-1} B^\top \tA^{t_1} \tA = Q S^{-1} P^{\top} \tA$. By \eqref{eq:ls-contraction} in Lemma \ref{lm:ls-bound}, we have
\begin{equation} \label{eq:ls-error-iter}
	\begin{split}
		&\tA^{t_1} B (W_{l+1} -\hW) = (I_n - \eta \tA^{ t_1} B B^\top \tA^{ t_1}  ) \tA^{t_1} B (W_{l} -\hW), \\
		\Longleftrightarrow & P S Q^\top (W_{l+1} -\hW) = (I_n - \eta P S^2 P^\top ) P S Q^\top (W_{l} -\hW),\\
		 \Longleftrightarrow & SQ^\top  (W_{l+1} -\hW) = (I_r - \eta S^2 )S Q^\top (W_{l} -\hW). 
	\end{split}
\end{equation} Notice that 
\begin{equation} \label{ineq:AB-spec-bound}
\begin{split}
	\|\tA^{t_1} B\| & = \|U_1 \Sigma_1^{t_1} V_1^\top B + U_2 \Sigma_2^{t_1} V_2^\top B\| \leq \|\Sigma_1\|^{t_1} \|V_1^\top B\| + \|\Sigma_2\|^{t_1} \|V_2^\top B\| \\
	& \overset{(a)}\leq \left( \frac{(p-q)n}{k} \left( 1 + C_1 \sqrt{ \frac{k}{n} \log n } \right) + C_2 \sqrt{n} \right)^{t_1} C_3 \sqrt{k} + (C_2 \sqrt{n})^{t_1} C_3 \sqrt{n}\\
	& \leq \left( \frac{(p-q)n}{k} + C_2 \sqrt{n} \right)^{t_1}\left(   1 + C_1 \sqrt{ \frac{k}{n} \log n } \right)^{t_1} C_3 \sqrt{k} + (C_2 \sqrt{n})^{t_1} C_3 \sqrt{n} \\
	& \overset{(b)}\lesssim \left( \frac{(p-q)n}{k} + C_2 \sqrt{n} \right)^{t_1} \sqrt{k}  \vee (C_3 \sqrt{n})^{t_1 + 1}.
\end{split}
\end{equation} where (a) is due to (A6),(A2),(A5); (b) is because $n \gtrsim k \log^3 n$. Notice that $\|S\|^2 = \|\tA^{t_1} B\|^2$, so by our choice of stepsize $\eta$ in GD, we have $\0 \preccurlyeq I_r - \eta S^2 \preccurlyeq I_r$ always holds.

 Now let us consider two different cases:
\begin{itemize}
	\item (Case 1: $ \frac{(p-q) n}{k} \leq \tC \sqrt{n} \log n$ for some large $\tC > 0$.) This is the low SNR regime, we do not have a spectrum gap, so the GD can converge slowly. Notice that by the SNR condition, we have $\|\Sigma_1\| \lesssim \sqrt{n} \log n$. Since $I_r - \eta S^2 \preccurlyeq I_r$, we have 
	\begin{equation*}
		\begin{split}
			\|\tA^{t_1} B (W_{t_2} -\hW)\|_\F^2 &\leq \|\tA^{t_1} B (W_{t_2-1} -\hW)\|_\F^2 \leq \cdots \leq \|\tA^{t_1} B (W_{0} -\hW)\|_\F^2 \\
			& \overset{W_0 = \0}= \|P S Q^\top  Q S^{-1} P^{\top} \tA\|_\F^2 \leq r \|\Sigma_1\|^2 \lesssim nk \log^2 n.
		\end{split}
	\end{equation*}
	
	\item (Case 2: $\frac{(p-q) n}{k} \geq \tC \sqrt{n} \log n$.) This is the high SNR regime, we have a spectrum gap. In this case, we would expect GD to converge fast in the top components, so a more refined control on the least squares error is needed. By \eqref{eq:ls-error-iter}, we have $\|\tA^{t_1} B (W_{l+1} -\hW)\|_\F^2 = \|SQ^\top  (W_{l+1} -\hW)\|_\F^2$ and 
	\begin{equation*}
		\begin{split} &SQ^\top  (W_{l+1} -\hW) = (I_r - \eta S^2 )S Q^\top (W_{l} -\hW)\\
			\Longleftrightarrow &\left[\begin{array}{cc}
			S_1 & \0 \\
			\0 & S_2
		\end{array} \right] \left[\begin{array}{c}
			Q_1^\top \\
			Q_2^\top
		\end{array} \right] (W_{l+1} -\hW) = \left[\begin{array}{cc}
			I_k - \eta S^2_1 & \0 \\
			\0 & I_{r-k} - \eta S^2_2
		\end{array} \right] \left[\begin{array}{c}
			Q_1^\top \\
			Q_2^\top
		\end{array} \right] (W_{l} -\hW) \\
		\Longleftrightarrow & \left[\begin{array}{c}
			S_1 Q_1^\top (W_{l+1} -\hW) \\
			S_2 Q_2^\top (W_{l+1} -\hW)
		\end{array} \right] =\left[\begin{array}{c}
			(I_k - \eta S^2_1)S_1 Q_1^\top (W_{l} -\hW) \\
			(I_{r-k} - \eta S^2_2)S_2 Q_2^\top (W_{l} -\hW)
		\end{array} \right].
		\end{split}
	\end{equation*} So we have $$\|\tA^{t_1} B (W_{l+1} -\hW)\|_\F^2 = \|SQ^\top  (W_{l+1} -\hW)\|_\F^2 = \|S_1Q_1^\top  (W_{l+1} -\hW)\|_\F^2 + \|S_2Q_2^\top  (W_{l+1} -\hW)\|_\F^2.$$
	Next, we will show that ({\bf Part I}) $\|S_1Q_1^\top  (W_{l+1} -\hW)\|_\F^2$ decay geometrically fast, so as long as $t_2 = C' \log n$, this term will be small and ({\bf Part II}) $\|S_2Q_2^\top  (W_{l+1} -\hW)\|_\F^2$ will not decay geometrically fast but can be well controlled.
	\vskip.3cm
{\bf (Part I)}  Since $\tC$ is large enough, so by \eqref{ineq:AB-spec-bound}, we have
	\begin{equation*}
		\begin{split}
			\sigma_1(S_1) = \|\tA^{t_1} B\| \leq C  \sqrt{k} \left( \frac{(p-q)n}{k} \right)^{t_1} \left( 1+ \frac{C_2 \sqrt{n} t_1 }{\tC n \log n} \right) \vee (C_3 \sqrt{n})^{t_1 + 1}\leq  C  \sqrt{k} \left( \frac{(p-q)n}{k} \right)^{t_1},
		\end{split}
	\end{equation*} here we can choose $\tC$ to be large enough so that $\tC n \log n$ dominates $C_2 \sqrt{n} t_1$ and the second inequality is because when $t_1 \gtrsim \log n$, the first term will dominate the second term. At the same time,
	\begin{equation*}
		\begin{split}
			\sigma_k(S_1) &=\sigma_k(\tA^{t_1} B)  = \sigma_k( U_1 \Sigma_1^{t_1} V_1^\top B + U_2 \Sigma_2^{t_1} V_2^\top B ) \geq \sigma_k ( U_1 \Sigma_1^{t_1} V_1^\top B ) - \|U_2 \Sigma_2^{t_1} V_2^\top B\|\\
	& \overset{(a)}\geq \left( \frac{(p-q)n}{k} \left( 1 - C_1 \sqrt{ \frac{k}{n} \log n } \right) - C_2 \sqrt{n} \right)^{t_1} C_3 \sqrt{k} - (C_2 \sqrt{n})^{t_1} C_3 \sqrt{n}\\
	& \overset{(b)}\geq \sqrt{k} \left( \frac{(p-q)n}{k} \right)^{t_1} \left( 1- C'_1 t_1\sqrt{ \frac{k}{n} \log n }  - \frac{c'\sqrt{n}t_1}{\frac{(p-q)n}{k}} \right) - (C_2 \sqrt{n})^{t_1} C_3 \sqrt{n} \\
	& \overset{(c)}\geq c \sqrt{k} \left( \frac{(p-q)n}{k} \right)^{t_1} 
		\end{split}
	\end{equation*}where (a) is due to (A1),(A2),(A5); in (b), we use the fact $n \gtrsim k \log^3 n$ and we choose $\tC$ to be large enough so that $\tC n \log n$ dominates $C_2 \sqrt{n} t_1$ so that $\frac{\sqrt{n}t_1}{\frac{(p-q)n}{k}} \leq c' < 1$; (c) is because when $t_1 \gtrsim \log n$, we have $(C_2 \sqrt{n})^{t_1} C_3 \sqrt{n}$ is dominated by the first term.
	
	In summary, in this regime, we have $\|\tA^{t_1} B\| \asymp \sigma_k(\tA^{t_1} B)$. Moreover, by choosing large enough constant $C''$ in $\eta$, we have $\eta$ is on the scale of $1/\|\tA^{t_1} B\|^2$ and $c_1 I_k \preccurlyeq I_k - \eta S^2_1 \preccurlyeq c_2 I_k$ for some $0< c_1 < c_2 < 1$. So $$\|S_1 Q_1^\top (W_{t_2} -\hW)\|^2_\F \leq c^{t_2} \|S_1 Q_1^\top (W_{0} -\hW)\|^2_\F,$$
	for some $0 < c < 1$. In addition, 
	\begin{equation*}
		\begin{split}
			\|S_1 Q_1^\top (W_{0} -\hW)\|^2_\F = \|S_1 Q_1^\top Q S^{-1} P^\top  \tA\|^2_\F = \|P_1^\top  \tA\|^2_\F \leq \|\tA\|^2_\F \leq n^2,
		\end{split}
	\end{equation*} where the last inequality is because each entry of $\tA$ is bounded by $1$ by construction. So as long as $t_2 = C' \log n$, we have $\|S_1 Q_1^\top (W_{t_2} -\hW)\|^2_\F \lesssim nk $.

		\vskip.3cm
{\bf (Part II)}  By the choice of stepsize, we always have contraction, so 
\begin{equation*}
\begin{split}
	\|S_2Q_2^\top  (W_{t_2} -\hW)\|_\F^2 \leq \|S_2Q_2^\top  (W_{t_2-1} -\hW)\|_\F^2  &\leq \cdots \leq \|S_2Q_2^\top  (W_{0} -\hW)\|_\F^2 =\| P_2^\top  \tA\|^2_\F \\
	& \leq \| P_2^\top U_1 \Sigma_1 V_1^\top + P_2^\top U_2 \Sigma_2 V_2^\top\|^2_\F \\
	& \leq 2 \|P_2^\top U_1 \Sigma_1 V_1^\top \|_\F^2 + 2\| P_2^\top U_2 \Sigma_2 V_2^\top\|^2_\F.
\end{split}
\end{equation*} Notice that $\| P_2^\top U_2 \Sigma_2 V_2^\top\|^2_\F \leq (r-k) \|P_2^\top U_2 \Sigma_2 V_2^\top\|^2 \leq (r-k) \|\Sigma_2\|^2 \lesssim nk $ by (A5). So if we can show $\|P_2^\top U_1 \Sigma_1 V_1^\top \|_\F^2 \lesssim n (\log n +k) \log^2 n$, we are done. Let $P_{1\perp} \in \bbO_{n,n-k}$ be the orthogonal complement of $P_1$, then we have
\begin{equation} \label{ineq:last-projection-bound}
	\begin{split}
		\|P_2^\top U_1 \Sigma_1 V_1^\top \|_\F^2 \leq \|P_{1\perp}^\top U_1 \Sigma_1 V_1^\top \|_\F^2 \leq \|P_{1\perp}^\top U_1\|^2 \|\Sigma_1\|_\F^2
	\end{split}
\end{equation} From (A6) and the SNR condition $\frac{(p-q) n}{k} \geq \tC \sqrt{n} \log n$ we considered here, we have $\|\Sigma_1\|_\F^2 \leq k\|\Sigma_1\|^2  \lesssim k \left( \frac{(p-q)n}{k} \right)^2$. Next, we need to bound $\|P_{1\perp}^\top U_1\|$. To this end, we introduce an ideal matrix $\tA^{t_1}_{\text{ideal}} =  \left( \frac{(p-q)n}{k} \right)^{t_1} U_1 V_1^\top$. Notice that the left singular subspace of $\tA^{t_1}_{\text{ideal}} B$ is exactly the one spanned by $U_1$ since $r \geq k$. By Wedin's perturbation bound (the version we use is Theorem 5 in \cite{luo2020schatten}) 
\begin{equation} \label{ineq:subspace-perturb}
	\begin{split}
		\|P_{1\perp}^\top U_1\| \leq \frac{2 \| (\tA^{t_1} - \tA^{t_1}_{\text{ideal}}) B \|}{\sigma_k( \tA^{t_1}_{\text{ideal}} B )}.
	\end{split}
\end{equation} It is easy to see $\sigma_k( \tA^{t_1}_{\text{ideal}} B ) \geq c_3 \left(\frac{(p-q)n}{k} \right)^{t_1} \sqrt{k} $ by (A3). At the same time,
\begin{equation} \label{ineq:noise-bound}
	\begin{split}
		 &\| (\tA^{t_1} - \tA^{t_1}_{\text{ideal}}) B \| \leq \left\|U_1 \left(\Sigma_1^{t_1} - \left(\frac{(p-q)n}{k} \right)^{t_1} I_k \right) V_1^\top B  \right\| + \|U_2 \Sigma_2^{t_1} V_2^\top B\| \\
		\overset{\text{(A3),(A5)}} \leq & C_3 \sqrt{k} \left\| \Sigma_1^{t_1} - \left( \frac{(p-q)n}{k} \right)^{t_1} I_k \right\| + (C_2 \sqrt{n})^{t_1} C_3 \sqrt{n} \\
		\overset{\text{(A1)}}\leq & C_3 \sqrt{k} \left( \frac{(p-q)n}{k} \right)^{t_1} \left[\left(  1 + C_1 \sqrt{\frac{k \log n}{n}} + \frac{C_2 \sqrt{n}}{\frac{(p-q)n}{k}}  \right)^{t_1} - 1 \right]  + (C_2 \sqrt{n})^{t_1} C_3 \sqrt{n}\\
		\overset{(a)}\leq & C \sqrt{k} \left( \frac{(p-q)n}{k} \right)^{t_1} \left( C_1 \sqrt{\frac{k \log n}{n}} t_1 + \frac{C_2 \sqrt{n} t_1}{\frac{(p-q)n}{k}} \right) + (C_2 \sqrt{n})^{t_1} C_3 \sqrt{n}\\
		\overset{(b)}\lesssim &  \sqrt{k} \left( \frac{(p-q)n}{k} \right)^{t_1} \left(\frac{\sqrt{n} t_1}{\frac{(p-q)n}{k}} +  \sqrt{\frac{k \log n}{n}} t_1 \right) 
	\end{split}
\end{equation} where (a) is because $n \gtrsim k \log^3 n$ and we choose $\tC$ to be large enough so that given $0<x <1 \leq  t$ and $xt < c < 1$, we have $(1 + x)^t \leq 1 + C'xt$ for all $x < c/t$; (b) is because when $t_1 = C' \log n$, we have the first term dominates the second term.

By plugging \eqref{ineq:noise-bound} into \eqref{ineq:subspace-perturb}, we have
\begin{equation*}
	\|P_{1\perp}^\top U_1\| \lesssim \frac{\sqrt{n} t_1}{\frac{(p-q)n}{k}} + \sqrt{\frac{k \log n}{n}} t_1 ,
\end{equation*} and combining it with \eqref{ineq:last-projection-bound}, we have
\begin{equation*}
\begin{split}
	&\|P_2^\top U_1 \Sigma_1 V_1^\top \|_\F^2 \leq \|P_{1\perp}^\top U_1\|^2 \|\Sigma_1\|_\F^2 \\
	\lesssim & k \left( \frac{(p-q)n}{k} \right)^2 \times \left( \frac{\sqrt{n} t_1}{\frac{(p-q)n}{k}} \right)^2 + k \left( \frac{(p-q)n}{k} \right)^2 \times \frac{k \log n t_1^2}{n}  \\
	\lesssim & n (k + \log n) \log^2 n.
\end{split}
\end{equation*}
\end{itemize} This finishes the proof of Theorem \ref{th:low-degree-up-bound}.

\subsection{Additional lemmas for Theorem \ref{th:low-degree-up-bound}}
\begin{Lemma}[Least-squares Bound] \label{lm:ls-bound}
	Consider the least-squares problem $\min_{W}\|Y W - X\|_\F^2$ where $Y \in \bbR^{n_1 \times r}, W \in \bbR^{r \times n_2}$ and $X \in \bbR^{n_1 \times n_2}$. Suppose $X$ can be decomposed as $X = M + Z$ and $\hW \in \argmin_{W} \|Y W - X\|_\F^2$, then for any $W \in \bbR^{r \times n_2}$, we have
	\begin{equation} \label{ineq:ls-bound}
		\|Y\hW - M\|_\F^2 \leq \frac{20}{9} \|Y W - M\|_\F^2 + \frac{110}{9} \| Z_{\max(r)} \|_\F^2,
	\end{equation} where for any matrix $(\cdot)$, $(\cdot)_{\max(r)}$ denotes the top-$r$ truncated SVD of it. 
	
Furthermore, if we consider gradient descent with stepsize $\eta$ for solving the problem, i.e.,
	\begin{equation} \label{eq:GD-iterate}
		W_{t+1} = W_t - \eta Y^\top (YW_t - X), \,\, \forall t \geq 0.
	\end{equation} Then we have
	\begin{equation} \label{eq:ls-contraction}
		Y (W_{t+1} - \hW) = (I_{n_1} - \eta Y Y^\top ) Y (W_t - \hW), \quad \forall t \geq 0.
	\end{equation} 
\end{Lemma}
\begin{proof}
Since $\hW$ minimizes the objective, we have $\|Y \hW - X\|_\F^2 \leq \|Y W - X\|_\F^2$ for any $W \in \bbR^{r \times n_2}$. By plugging in $X = M + Z$ and rearranging terms, the previous condition is equivalent to
\begin{equation} \label{ineq:ls-bound-equi}
	\|Y\hW - M\|_\F^2 \leq \|YW - M\|_\F^2 + 2 \langle Y(\hW - W) , Z \rangle.
\end{equation}
Then
\begin{equation} \label{ineq:ls-bound-1}
	\begin{split}
		2\langle Y(\hW - W) , Z \rangle \overset{(a)}
		\leq &  2\|Y(\hW - W)\|_\F \|Z_{\max(r)}\|_\F \\
		\overset{(b)}\leq & 2(\|Y\hW - M\|_\F +\|YW - M\|_\F )  \|Z_{\max(r)}\|_\F\\
		\overset{(c)}\leq & ( \frac{1}{10} \|Y\hW - M\|^2_\F + 10\|Z_{\max(r)}\|^2_\F ) + (\|YW - M\|^2_\F + \|Z_{\max(r)}\|^2_\F),
	\end{split}
\end{equation} where (a) is by Lemma 3 of \cite{luo2020schatten}; (b) is by the triangle inequality; (c) is because $a^2 + b^2 \geq 2|ab|$.

By plugging \eqref{ineq:ls-bound-1} into \eqref{ineq:ls-bound-equi}, we have
\begin{equation*}
	\frac{9}{10} \|Y\hW - M\|_\F^2 \leq 2\|YW - M\|_\F^2 + 11\|Z_{\max(r)}\|^2_\F,
\end{equation*} which implies \eqref{ineq:ls-bound}.

	Now, we prove the second statement. The first order optimality condition yields $Y^\top (Y \hW - X) = 0$. By plugging it into the GD update formula, we get
	\begin{equation*}
		\begin{split}
			&W_{t+1} = W_t - \eta Y^\top (YW_t - Y \hW) \\
		\Leftrightarrow	& Y(W_{t+1} - \hW) = Y(W_t - \hW) - \eta Y Y^\top Y(W_t - \hW)\\
		\Leftrightarrow & \eqref{eq:ls-contraction}
		\end{split}
	\end{equation*}
\end{proof}

\section{Proofs in Section \ref{sec: nonparametric-graphon-est} }\label{sec: proof-nonparametric-graphon}

\subsection{Proof of Theorem \ref{th: computational-lower-bound-nonparametric-graphon}  }
	We are going to reduce the problem to the graphon estimation problem in the stochastic block model. First, we are going to construct a class of graphons that mimics the class $\cM_k$. The construction has also been used in \cite{gao2015rate} to show the statistical lower bound for nonparametric graphon estimation. To construct a function $f \in \cH_\gamma(L)$, we need the following smooth function $K(x)$ that is infinitely differentiable,
	\begin{equation*}
		K(x) = C_K \exp( - \frac{1}{1-64x^2} ) \indi( |x| < 1/8 ),
	\end{equation*}  where $C_K$ is a constant such that $\int K(x) dx = 1$. The function $K$ is a positive symmetric mollifier, based on which we define the following function $\psi(x) = \int_{-3/8}^{3/8}  K(x-y) dy$.
	The function $\psi(x)$ is called a smooth cutoff function. The support of $\psi(x)$ is $(-1/2,1/2)$. Since $K(x)$ is supported on $(-1/8,1/8)$ and the value of its integral is $1$, $\psi(x)$ is $1$ on the interval $[-1/4,1/4]$. More importantly, the smoothness property of $K(x)$ is inherited by $\psi(x)$, i.e., $\psi(x)$ is also infinitely differentiable on $(-1/2,1/2)$. Given small $0< \epsilon < 1/2$, and $\epsilon \leq q < p \leq 1-\epsilon$, let us define the symmetric matrix $Q^{qp} \in [0,1]^{k \times k}$ as follows: $Q^{qp}_{uu} = p$ for all $u \in [k]$ and $Q^{qp}_{uv} = Q^{qp}_{vu} = q$ for all $1\leq u < v \leq k$. Now define 
	\begin{equation} \label{def: f-qp-definition}
		f_{qp}(x,y) = \sum_{a,b \in [k]} \left( Q^{qp}_{ab} - q \right) \psi(kx - a + \frac{1}{2} ) \psi(ky -b + \frac{1}{2} ) + q.
	\end{equation} The following proposition gives the condition when the function defined in \eqref{def: f-qp-definition} belongs to $\cF_\gamma$.
\begin{Proposition}\label{prop:smooth-f}
	For any $\gamma > 0$. Suppose $p,q,k$ satisfies $(p-q)k^\gamma \leq c_{L,\gamma}$ for a small enough constant $c_{L,\gamma}>0$ depending on $L,\gamma$ only. Then we have $f_{qp}(x,y) \in \cF_{\gamma}(L)$.
\end{Proposition} The proof of this proposition is provided in the following Appendix \ref{sec:proof-smooth-f}. Now, let us define a new function class
	\begin{equation}
		\widebar{\cF}_{\gamma}(L) = \{f_{qp}: \epsilon \leq q \leq p \leq 1-\epsilon, p-q \leq c_{L,\gamma}/k^{\gamma} \}. 
	\end{equation} By construction and Proposition \ref{prop:smooth-f}, we have $\widebar{\cF}_{\gamma}(L) \subseteq \cF_{\gamma}(L) $. We also let
\begin{equation*}
\begin{split}
	\widebar{\cM}_{k}(n') = \Big\{& M = (M_{ij}) \in [0,1]^{  n' \times n'}: M_{ii} = 0, M_{ij} = M_{ji} = Q^{qp}_{z_i z_j} \text{ for } i \neq j\\
	& \text{for some }\epsilon \leq q < p \leq 1-\epsilon, p-q \leq c_{L,\gamma}/k^\gamma \text{ and } z \in [k]^{n'} \Big\}.
\end{split}
\end{equation*}

	The definition of $f_{qp}$ implies that for any $a, b \in [k]$,
	\begin{equation} \label{eq: connection-to-SBM}
		f_{qp}(x,y) \equiv Q_{ab}^{qp}, \text{  when } (x,y) \in \left[ \frac{a-3/4}{k}, \frac{a-1/4}{k} \right] \times \left[ \frac{b-3/4}{k}, \frac{b-1/4}{k} \right].
	\end{equation} Therefore, in a sub-domain, $f_{qp}$ is a piecewise constant function. To be specific, define
	\begin{equation*}
		I = \left( \bigcup_{a=1}^k \left[ \frac{(a-3/4)n}{k}, \frac{(a-1/4)n}{k} \right]  \right) \bigcap [n].
	\end{equation*} The values of $f_{qp}(i/n,j/n)$ on $(i,j) \in I \times I$ form a stochastic block model and $|I| \geq n/4$. Recall that $\Pi_n$ is the set of permutations on $[n]$. Define a subset by $\Pi_n' = \{\pi \in \Pi_n: \pi(i) = i \text{ for } i \in [n] \setminus I \}$. In other words, any $\pi \in \Pi_n'$ can be viewed as a permutation on $I$. Note that for any permutation $\pi \in \Pi_n'$, one valid distribution choice for $\bbP_\xi$ is
	\begin{equation*}
		\bbP_{\pi}\left( (\xi_1, \ldots, \xi_n) = (\pi(1)/n, \ldots, \pi(n)/n) \right) = 1.
	\end{equation*}

	 Then we have
	 \begin{equation} \label{ineq: nonparametric-graphon-ineq1}
	 	\begin{split}
	 		& \quad \inf_{\hM \in \bbR[A]^{n \times n}_{\leq D}} \sup_{f \in \cF_\gamma (L)} \sup_{\bbP_\xi} \bbE \left( \ell(\widehat{M}, M_f) \right) \\
	 		 &\geq \inf_{\hM \in \bbR[A]^{n \times n}_{\leq D}} \sup_{f \in \widebar{\cF}_{\gamma}(L)} \sup_{\bbP_\xi} \bbE \left( \ell(\widehat{M}, M_f) \right) \\
	 		& \geq  \inf_{\hM \in \bbR[A]^{n \times n}_{\leq D}} \sup_{f \in \widebar{\cF}_{\gamma}(L)} \sup_{\pi \in \Pi_n'} \bbE \left( \frac{1}{{n \choose 2}} \sum_{1 \leq i < j \leq n}\left(\widehat{M}_{( \pi(i), \pi(j) )}- (M_f)_{( \pi(i), \pi(j) )}\right)^2 \right) \\
	 		&  \geq c\inf_{\hM \in \bbR[A]^{n \times n}_{\leq D}} \sup_{f \in \widebar{\cF}_{\gamma}(L)} \sup_{\pi \in \Pi_n'} \bbE \left( \frac{1}{{|I| \choose 2}} \sum_{\substack{1 \leq i < j \leq n\\ i,j \in I}}\left(\widehat{M}_{( \pi(i), \pi(j) )}- (M_f)_{( \pi(i), \pi(j) )}\right)^2 \right),
	 	\end{split}
	 \end{equation} where in the last inequality, we use the condition $|I| \geq n/4$.
	 Notice that in the sub-domain $I \times I$, the quantity at the end of \eqref{ineq: nonparametric-graphon-ineq1} can be viewed as an SBM graphon estimation problem.
	 
	 By the above interpretation, to further proceed, we now need to specify $q,p$ and the number of communities $k$ in approximating $\gamma$-smooth graphon with SBM. Given any $0 < r < 1$, choose $ \epsilon \leq q < p \leq  1-\epsilon $ satisfying $p-q =  \frac{c}{D(D+1)} k/\sqrt{n} $ for small enough $c > 0$ such that the condition $\frac{(p-q)^2}{q(1-p)} \leq \frac{r}{(D(D+1))^2} \frac{k^2}{|I|} $ holds. In addition, to simultaneously guarantee $p-q \leq c_{L,\gamma}/k^{\gamma}$ so that $\widebar{\cF}_{\gamma}(L) \subseteq \cF_{\gamma}(L)$, we require $k \leq \left(c_{L,\gamma}\frac{D(D+1)}{c}\right)^{ \frac{1}{\gamma+1} } n^{\frac{1}{2(\gamma+1)} }$. 	 
	
	Then by \eqref{ineq: nonparametric-graphon-ineq1} and its interpretation, we can apply Theorem \ref{th: comp-limit-graphon-est} in the sub-domain $I \times I$ with $k$ communities and $|I|$ nodes and the condition $\frac{(p-q)^2}{q(1-p)} \leq \frac{r}{(D(D+1))^2} \frac{k^2}{|I|} $ being active, 
	\begin{equation*}
		\begin{split}
			& \inf_{\hM \in \bbR[A]^{n \times n}_{\leq D}} \sup_{f \in \cF_\gamma (L)} \sup_{\bbP_\xi} \bbE \left( \ell(\widehat{f}, f) \right)\\
			 \overset{ \eqref{ineq: nonparametric-graphon-ineq1} }\geq & c\inf_{\hM \in \bbR[A]^{n \times n}_{\leq D}} \sup_{f \in \widebar{\cF}_{\gamma}(L)} \sup_{\pi \in \Pi_n'} \bbE \left( \frac{1}{{|I| \choose 2}} \sum_{\substack{1 \leq i < j \leq n\\ i,j \in I}}\left(\widehat{M}_{( \pi(i), \pi(j) )}- (M_f)_{( \pi(i), \pi(j) )}\right)^2 \right)\\
			 \geq &  c\inf_{\hM \in \bbR[A]_{\leq D}^{|I| \times |I|} }\sup_{M \in \widebar{\cM}_k(|I|)} \bbE\left( \frac{1}{{|I| \choose 2}} \ell( \hM, (M_f)_{I \times I} )  \right)  \\
			 \overset{ \text{Theorem }\ref{th: comp-limit-graphon-est} }\geq & c' (p-q)^2/k  \overset{(a)}= c' \frac{k}{(D(D+1))^2 n} \overset{(b)}= c' n^{ -(2\gamma +1)/(2 \gamma+2) }/D^{4- \frac{2}{\gamma+1}} \geq c' n^{ -(2\gamma +1)/(2 \gamma+2) }/D^{4},
		\end{split}
	\end{equation*} where (a) is due to the choice of $p-q$ and (b) is due to the fact we want to maximize the lower bound given the $k \leq \left(c_{L,\gamma}\frac{D(D+1)}{c}\right)^{ \frac{1}{\gamma+1} } n^{\frac{1}{2(\gamma+1)} }$ constraint, so we choose $k = \left(c_{L,\gamma}\frac{D(D+1)}{c}\right)^{ \frac{1}{\gamma+1} } n^{\frac{1}{2(\gamma+1)} }$. Notice that here $c'$ only depends on $\gamma$ and $L$. This finishes the proof of this theorem.

\subsection{Proof of Proposition \ref{prop:smooth-f} } \label{sec:proof-smooth-f}
	Since $\psi(x)$ is infinitely differentiable, for any finite positive integer $m$, we have $\max_{j \leq m}|\nabla_j \psi(x)| \leq c_{m}$ for some $c_{m} > 0$ for all $x \in [-1/2,1/2]$. In addition, the derivative of $ \nabla_m \psi(x)$ vanishes at $(-\infty, -1/2]$ and $[1/2, \infty)$ at any order. 
	
	Recall  
	\begin{equation*}
		f_{qp}(x,y) = \sum_{a,b \in [k]} \left( Q^{qp}_{ab} - q \right) \psi(kx - a + \frac{1}{2} ) \psi(ky -b + \frac{1}{2} ) + q.
	\end{equation*} So for any nonnegative integer $j,l$, we have
	\begin{equation*}
		\begin{split}
			\nabla_{jl} f_{qp}(x,y) = \sum_{a,b \in [k]} k^{j + l} \left( Q^{qp}_{ab} - q \right) \nabla_j \psi(kx - a + \frac{1}{2} ) \nabla_l \psi(ky -b + \frac{1}{2} ).
		\end{split}
	\end{equation*} Notice that since $ \nabla_m \psi(x)$ vanishes at $(-\infty, -1/2]$ and $[1/2, \infty)$ at any order, each sum term in $f_{qp}(x,y)$ and $\nabla_{jl} f(qp)(x,y)$ is non-zero only when $a \in (kx, kx + 1)$ and $b \in (ky, ky+1)$.

	 Given any $x, y \in [0,1]$, let us define
	\begin{equation*}
		\begin{split}
			a_{x} = \begin{cases}
			1, & \text{ if } x \in [0,1/k], \\
				\lceil kx \rceil, & \text{ if } x \in (1/k,1],
			\end{cases} \quad 
			b_{y} = \begin{cases}
			1, & \text{ if } y \in [0,1/k], \\
				\lceil ky \rceil, & \text{ if } y \in (1/k,1].
			\end{cases}  
		\end{split}
	\end{equation*} So we can simplify $f_{qp}(x,y)$ and $\nabla_{jl} f_{qp}(x,y)$ as follows:
	\begin{equation} \label{eq:f-simple}
		\begin{split}
			f_{qp}(x,y) &= \left( Q^{qp}_{a_x b_y } - q \right) \psi(kx - a_x + \frac{1}{2} ) \psi(ky -b_y + \frac{1}{2} ) + q\\
			\nabla_{jl} f_{qp}(x,y) &=k^{j + l} \left( Q^{qp}_{a_x b_y} - q \right) \nabla_j \psi(kx - a_x + \frac{1}{2} ) \nabla_l \psi(ky -b_y + \frac{1}{2} ).
		\end{split}
	\end{equation}
	\vskip.5cm
	{\noindent \bf Case 1: show $f_{qp} \in \cF_{\gamma}(L)$ when $\gamma \in (0,1)$.} Let us denote $\cD = [0,1] \times [0,1]$. In this case,
	\begin{equation*}
		\begin{split}
			\|f_{qp}\|_{\cH_{\gamma}} = \sup_{x,y \in \cD} |f_{qp}(x,y)| + \sup_{(x,y) \neq (x',y') \in \cD} \frac{|f_{qp}(x,y) - f_{qp}(x',y') |}{(|x - x'| + |y - y'|)^{\gamma}}.
		\end{split}
	\end{equation*} First, it is easy to check $\sup_{x,y \in \cD} |f_{qp}(x,y)| \leq 1$ by \eqref{eq:f-simple}. In addition 
	\begin{equation*}
		\begin{split}
			& |f_{qp}(x,y) - f_{qp}(x',y') | \\
			= & \left|  ( Q^{qp}_{a_x b_y } - q ) \psi(kx - a_x + \frac{1}{2} ) \psi(ky -b_y + \frac{1}{2} ) - (Q^{qp}_{a_{x'} b_{y'} } - q)  \psi(kx' - a_{x'} + \frac{1}{2} ) \psi(ky' -b_{y'} + \frac{1}{2} ) \right|.
		\end{split}
	\end{equation*} 
	We discuss how to bound the above quantity under different scenarios. The first scenario is that $a_x \neq b_y$ and $a_{x'} \neq b_{y'}$, then $|f_{qp}(x,y) - f_{qp}(x',y') | = 0$ since $Q^{qp}_{a_x b_y } = Q^{qp}_{a_{x'} b_{y'} } = q$. The second scenario is that we have $a_x = b_y$ or $a_{x'} = b_{y'}$, one of them holds. Without loss of generality, let us consider $a_{x'} = b_{y'} = 1$ and $a_x > b_y$. In this scenario, we have
	\begin{equation*}
		\begin{split}
			 |f_{qp}(x,y) - f_{qp}(x',y') | =  (p-q)  \left|  \psi(kx - a_x + \frac{1}{2} ) \psi(ky -b_y + \frac{1}{2} ) \right|.
		\end{split}
	\end{equation*} To bound the ratio $\frac{(p-q)  \left|  \psi(kx - a_x + \frac{1}{2} ) \psi(ky -b_y + \frac{1}{2} ) \right|}{(|x - x'| + |y - y'|)^{\gamma}}$, it is enough to consider $(x,y)$ so that $y = y'$ and $a_{x} = 2$, otherwise the ratio can only be smaller. So 
	\begin{equation*}
		\begin{split}
			&\sup_{(x,y) \neq (x',y') \in \cD} \frac{(p-q)  \left|  \psi(kx - a_x + \frac{1}{2} ) \psi(ky -b_y + \frac{1}{2} ) \right|}{(|x - x'| + |y - y'|)^{\gamma}}\\
			 \leq & \sup_{ x' \in [0,1/k], x \in (1/k,2/k]} \frac{(p-q)  \left|  \psi(kx - a_x + \frac{1}{2} ) \right|}{(|x - x'|)^{\gamma}} = \sup_{ x' \in [0,1/k], x \in (1/k,2/k]} \frac{(p-q)  \left|  \psi(kx - \frac{3}{2} ) \right|}{(|x - x'|)^{\gamma}}\\
			 = & \sup_{ x' \in [0,1/k], x \in (1/k,2/k]} \frac{(p-q)  \left|  \psi(kx - \frac{3}{2} ) - \psi(k\times 1/k - \frac{3}{2}) \right|}{(|x - x'|)^{\gamma}}\\
			 \leq & \sup_{ x' \in [0,1/k], x \in (1/k,2/k]} \frac{c_1(p-q)  k |x - \frac{1}{k} |}{(|x - x'|)^{\gamma}} \leq c_1(p-q)  k |x - 1/k|^{1-\gamma} \leq c_1 (p-q) k (1/k)^{1-\gamma}\\
			 \leq & c_1 (p-q) k^\gamma, 
		\end{split}
	\end{equation*} where $c_1$ denotes the Lipschitz constant of $\psi(x)$
	
	The third scenario is that $a_x = b_y$ and $a_{x'} = b_{y'}$. In this case, we have	\begin{equation*}
		\begin{split}
			& |f_{qp}(x,y) - f_{qp}(x',y') | \\
			= &  (p-q)\left|  \psi(kx - a_x + \frac{1}{2} ) \psi(ky -b_y + \frac{1}{2} ) -  \psi(kx' - a_{x'} + \frac{1}{2} ) \psi(ky' -b_{y'} + \frac{1}{2} ) \right| \\
			\leq & (p-q) \left|  \psi(kx - a_x + \frac{1}{2} ) \psi(ky -b_y + \frac{1}{2} ) - \psi(kx' - a_{x'} + \frac{1}{2} )\psi(ky -b_y + \frac{1}{2} ) \right| \\
			& + (p-q) \left| \psi(kx' - a_{x'} + \frac{1}{2} )\psi(ky -b_y + \frac{1}{2} ) -  \psi(kx' - a_{x'} + \frac{1}{2} ) \psi(ky' -b_{y'} + \frac{1}{2} ) \right|\\
			\leq & (p-q) \left|  \psi(kx - a_x + \frac{1}{2} )  - \psi(kx' - a_{x'} + \frac{1}{2} ) \right| + (p-q) \left| \psi(ky -b_y + \frac{1}{2} ) -  \psi(ky' -b_{y'} + \frac{1}{2} ) \right|.
		\end{split}
	\end{equation*} 
	
	Let us now bound $\left|  \psi(kx - a_x + \frac{1}{2} )  - \psi(kx' - a_{x'} + \frac{1}{2} ) \right|$. Without loss of generality, let us assume $x \in [0,1/k]$ and $a_x = 1$. 
	\begin{itemize}
		\item (Type I) If $x' \in [0,1/k]$ and $a_{x'} = 1$, then $\left|  \psi(kx - a_x + \frac{1}{2} )  - \psi(kx' - a_{x'} + \frac{1}{2} ) \right| \leq c_{1} k |x- x'|$.
		\item (Type II) If $x' \in (1/k,2/k]$, $x' - x \leq 1/k$ and $a_{x'} = 2$, then 
		\begin{equation*}
		\begin{split}
			\left|  \psi(kx - a_x + \frac{1}{2} )  - \psi(kx' - a_{x'} + \frac{1}{2} ) \right| &= \left|  \psi(kx - \frac{1}{2} )  - \psi(kx' - \frac{3}{2} ) \right| \\
			& \overset{(a)}= \left|  \psi(kx - \frac{1}{2} )  - \psi(\frac{3}{2} - kx' ) \right|\\ & \leq c_1|k(x+x') -2| \overset{(b)}\leq c_1 k(x' - x) = c_1 k|x' - x|
		\end{split}
		\end{equation*} where (a) is because $\psi(x)$ is symmetric around zero; (b) is because $-k(x'-x) \leq k(x+x') -2 \leq k(x'-x)$.
		\item (Type III) If $x' \in (1/k, 1]$ and $x' - x \geq 1/k$. We can find $x'' := x' + (a_x - a_{x'})/k = a_x /k + x' - \lceil x'k\rceil/k  \in [0,1]$, so that 
	\begin{equation*}
		a_x = a_{x''} \quad \text{ and } \quad  kx' - a_{x'} + \frac{1}{2}   = kx'' - a_{x''} + \frac{1}{2}.
	\end{equation*} Notice that since $a_{x} = a_{x''}$, we have $ |x - x''| \leq 1/k \leq |x-x'| $. Thus,
	\begin{equation*}
		\begin{split}
			\left|  \psi(kx - a_x + \frac{1}{2} )  - \psi(kx' - a_{x'} + \frac{1}{2} ) \right|& = \left|  \psi(kx - a_x + \frac{1}{2} )  - \psi(kx'' - a_{x''} + \frac{1}{2} ) \right|\\
			& \leq c_1 k|x - x''|.
		\end{split}
	\end{equation*}
	\end{itemize}
	We can obtain similar bounds for $\left| \psi(ky -b_y + \frac{1}{2} ) -  \psi(ky' -b_{y'} + \frac{1}{2} ) \right|$. Now, let us bound $\sup_{(x,y) \neq (x',y') \in \cD} \frac{|f_{qp}(x,y) - f_{qp}(x',y') |}{(|x - x'| + |y - y'|)^{\gamma}}$:
	\begin{itemize}
		\item (Both of $x,x'$ and $y,y'$ fall in Type III) We have
		\begin{equation*}
			\begin{split}
				\frac{|f_{qp}(x,y) - f_{qp}(x',y') |}{(|x - x'| + |y - y'|)^{\gamma}} \leq  \frac{|f_{qp}(x,y) - f_{qp}(x',y') |}{(|x - x''| + |y - y''|)^{\gamma}} &\leq \frac{(p-q)c_1 k(|x-x''| + |y - y''|)}{(|x - x''| + |y - y''|)^{\gamma}}  \\
				& = (p-q)c_1 k(|x-x''| + |y - y''|)^{1- \gamma} \\
				& \leq (p-q)c_1 k (2/k)^{1-\gamma} = c_1 2^{1-\gamma} (p-q) k^{\gamma}.
			\end{split}
		\end{equation*}
		\item (One of $x,x'$ and $y,y'$ falls in Type III, without loss of generality, we assume it is $x,x'$) We have
		\begin{equation*}
			\begin{split}
				\frac{|f_{qp}(x,y) - f_{qp}(x',y') |}{(|x - x'| + |y - y'|)^{\gamma}} \leq  \frac{|f_{qp}(x,y) - f_{qp}(x',y') |}{(|x - x''| + |y - y'|)^{\gamma}} &\leq \frac{(p-q)c_1 k(|x-x''| + |y - y'|)}{(|x - x''| + |y - y'|)^{\gamma}}  \\
				& = (p-q)c_1 k(|x-x''| + |y - y'|)^{1- \gamma} \\
				& \leq (p-q)c_1 k (2/k)^{1-\gamma} = c_1 2^{1-\gamma}(p-q)k^{\gamma}.
			\end{split}
		\end{equation*}
		\item (None of $x,x'$ and $y,y'$ falls in Type III) We have
		\begin{equation*}
			\begin{split}
				\frac{|f_{qp}(x,y) - f_{qp}(x',y') |}{(|x - x'| + |y - y'|)^{\gamma}}  &\leq \frac{(p-q)c_1 k(|x-x'| + |y - y'|)}{(|x - x'| + |y - y'|)^{\gamma}}  \\
				& = (p-q)c_1 k(|x-x'| + |y - y'|)^{1- \gamma} \\
				& \leq (p-q)c_1 k (2/k)^{1-\gamma} = c_1 2^{1-\gamma} (p-q) k^{\gamma}.
			\end{split}
		\end{equation*}
	\end{itemize} In summary, we have $\sup_{(x,y) \neq (x',y') \in \cD}\frac{|f_{qp}(x,y) - f_{qp}(x',y') |}{(|x - x'| + |y - y'|)^{\gamma}} \leq c_1 2^{1-\gamma} (p-q) k^{\gamma} \leq L-1$ as long as $(p-q) k^{\gamma} \leq (L-1)/(c_1 2^{1-\gamma})$. To this end, we have $\|f_{qp}\|_{\cH_{\gamma}} \leq L$.
	\vskip.5cm
	{\noindent \bf (Case 2: show $f_{qp} \in \cF_{\gamma}(L)$ when $\gamma \geq 1$).}  Recall,
\begin{equation*}
	\|f\|_{\cH_{\gamma}} = \max_{j+l \leq \lfloor \gamma \rfloor  } \sup_{(x,y) \in \cD} \left| \nabla_{jl} f(x,y) \right| + \max_{j+l = \lfloor \gamma \rfloor } \sup_{(x,y) \neq (x',y') \in \cD} \frac{\left| \nabla_{jl} f(x,y) - \nabla_{jl}f(x', y') \right|}{( |x-x'| + |y-y'|  )^{\gamma - \lfloor \gamma \rfloor }}.
\end{equation*}

By \eqref{eq:f-simple}, we have
\begin{equation*}
	\begin{split}
		\max_{j+l \leq \lfloor \gamma \rfloor  } \sup_{(x,y) \in \cD} \left| \nabla_{jl} f_{qp}(x,y) \right| &\leq \max_{j +l \leq \lfloor \gamma \rfloor  } \sup_{(x,y) \in \cD} k^{j + l} \left( Q^{qp}_{a_x b_y} - q \right) \nabla_j \psi(kx - a_x + \frac{1}{2} ) \nabla_l \psi(ky -b_y + \frac{1}{2} ) \\
		& \leq (p-q) c^2_{\gamma } k^{\lfloor \gamma \rfloor} \leq (p-q) c^2_{\gamma} k^{\gamma },
	\end{split}
\end{equation*} where $c_{\gamma }$ satisfies $\max_{m =1 ,\ldots, c_{\lfloor \gamma \rfloor}} \sup_x |\nabla_m \psi(x)| \leq c_{\gamma}$. At the same time,
	\begin{equation*}
		\begin{split}
			& | \nabla_{jl} f_{qp}(x,y) -\nabla_{jl}  f_{qp}(x',y') | \\
			= & \Big|  ( Q^{qp}_{a_x b_y } - q )\nabla_{j}  \psi(kx - a_x + \frac{1}{2} ) \nabla_{l} \psi(ky -b_y + \frac{1}{2} ) \\
			& - (Q^{qp}_{a_{x'} b_{y'} } - q) \nabla_{j}  \psi(kx' - a_{x'} + \frac{1}{2} ) \nabla_{l} \psi(ky' -b_{y'} + \frac{1}{2} ) \Big|.
		\end{split}
	\end{equation*} 
	We discuss to bound the above quantity under different scenarios. The first scenario is that $a_x \neq b_y$ and $a_{x'} \neq b_{y'}$, then $|\nabla_{jl} f_{qp}(x,y) - \nabla_{jl} f_{qp}(x',y') | = 0$. The second scenario is that we have $a_x = b_y$ or $a_{x'} = b_{y'}$, one of them holds. Without loss of generality, let us consider $a_{x'} = b_{y'} = 1$ and $a_x > b_y$. In this scenario, we have
	\begin{equation*}
		\begin{split}
			& |\nabla_{jl} f_{qp}(x,y) - \nabla_{jl} f_{qp}(x',y') | \\
			\leq &  (p-q)  \left|  \nabla_{j} \psi(kx - a_x + \frac{1}{2} ) \nabla_{l} \psi(ky -b_y + \frac{1}{2} ) \right| \\
		\end{split}
	\end{equation*} To bound the ratio $\frac{(p-q)  \left| \nabla_{j}  \psi(kx - a_x + \frac{1}{2} )\nabla_{l}  \psi(ky -b_y + \frac{1}{2} ) \right|}{(|x - x'| + |y - y'|)^{\gamma}}$, it is enough to consider $(x,y)$ so that $y = y'$ and $a_{x} = 2$, otherwise the ratio can only be smaller. So 
	\begin{equation*}
		\begin{split}
			&\sup_{(x,y) \neq (x',y') \in \cD} \frac{(p-q)  \left| \nabla_{j}  \psi(kx - a_x + \frac{1}{2} ) \nabla_{l} \psi(ky -b_y + \frac{1}{2} ) \right|}{(|x - x'| + |y - y'|)^{\gamma}}\\
			 \leq & c_\gamma \sup_{ x' \in [0,1/k], x \in (1/k,2/k]} \frac{(p-q)  \left| \nabla_{j}  \psi(kx - a_x + \frac{1}{2} ) \right|}{(|x - x'|)^{\gamma}} = c_\gamma\sup_{ x' \in [0,1/k], x \in (1/k,2/k]} \frac{(p-q)  \left| \nabla_{j}  \psi(kx - \frac{3}{2} ) \right|}{(|x - x'|)^{\gamma}}\\
			\overset{(a)} = &c_\gamma \sup_{ x' \in [0,1/k], x \in (1/k,2/k]} \frac{(p-q)  \left| \nabla_{j}  \psi(kx - \frac{3}{2} ) - \nabla_{j} \psi(k\times 1/k - \frac{3}{2}) \right|}{(|x - x'|)^{\gamma}}\\
			 \leq & c_\gamma \sup_{ x' \in [0,1/k], x \in (1/k,2/k]} \frac{c'_\gamma (p-q)  k |x - \frac{1}{k} |}{(|x - x'|)^{\gamma}} \leq c_\gamma c_\gamma' (p-q)  k |x - 1/k|^{1-\gamma} \leq c_\gamma c_\gamma' (p-q) k (1/k)^{1-\gamma}\\
			 \leq & c_\gamma c_\gamma' (p-q) k^\gamma, 
		\end{split}
	\end{equation*} where $c_\gamma'$ denotes the Lipschitz constant of $\nabla_j \psi(x)$ for any $j \leq \lfloor \gamma \rfloor$ and here in (a) we use the property $\nabla_j \psi(x)$ vanishes at $x = -1/2$. 

	The third scenario is that $a_x = b_y$ and $a_{x'} = b_{y'}$. In this case, we have	\begin{equation*}
		\begin{split}
			& | \nabla_{jl} f_{qp}(x,y) - f_{qp}(x',y') | \\
			= &  (p-q)  k^{j+l}   \left| \nabla_{j} \psi(kx - a_x + \frac{1}{2} ) \nabla_{l} \psi(ky -b_y + \frac{1}{2} ) -  \nabla_{j}\psi(kx' - a_{x'} + \frac{1}{2} ) \nabla_{l}\psi(ky' -b_{y'} + \frac{1}{2} ) \right| \\
			\leq & c_\gamma (p-q) k^{j+l}  \left| \nabla_{j} \psi(kx - a_x + \frac{1}{2} )  - \nabla_{j}\psi(kx' - a_{x'} + \frac{1}{2} ) \right| \\
			& + c_\gamma (p-q) k^{j+l}  \left| \nabla_{l}\psi(ky -b_y + \frac{1}{2} ) -  \nabla_{l}\psi(ky' -b_{y'} + \frac{1}{2} ) \right|.
		\end{split}
	\end{equation*} 
	
	Following a similar analysis as in Case 1, we can get bound 
	\begin{equation*}
		\begin{split}
			 \max_{j+l = \lfloor \gamma \rfloor } \sup_{(x,y) \neq (x',y') \in \cD} \frac{\left| \nabla_{jl} f(x,y) - \nabla_{jl}f(x', y') \right|}{( |x-x'| + |y-y'|  )^{\gamma - \lfloor \gamma \rfloor }} & \leq   k^{\lfloor \gamma \rfloor} c_\gamma c'_{\gamma} (p-q) k  (2/k)^{1 + \lfloor \gamma \rfloor - \gamma } \\
			 & = c_\gamma c'_{\gamma} 2^{1 + \lfloor \gamma \rfloor - \gamma } (p-q) k^\gamma. 
		\end{split}
	\end{equation*}

	In summary, we have
	\begin{equation*}
		\begin{split}
			\|f_{qp}\|_{\cH_\gamma}  & \leq c^2_\gamma (p-q) k^\gamma +  \max_{j+l = \lfloor \gamma \rfloor }\sup_{(x,y) \neq (x',y') \in \cD}\frac{ |\nabla_{jl}f_{qp}(x,y) - \nabla_{jl}f_{qp}(x',y') |}{(|x - x'| + |y - y'|)^{\gamma}} \\
			&\leq c^2_\gamma (p-q) k^\gamma + c_\gamma c'_{\gamma} 2^{1-\gamma} (p-q) k^{\gamma} \leq L
		\end{split}
	\end{equation*}
as long as $(p-q) k^{\gamma}$ is small than a constant which depends on $\gamma$ and $L$ only. We finish the proof of this proposition.

\section{Proofs in Section \ref{sec: community-detection} } \label{sec: proof-community-detection}

\subsection{Proof of Theorem \ref{th: clustering-error-comp-limit}  }
	This theorem is proved by leveraging the result in Theorem \ref{th: comp-limit-graphon-est} and a contradiction argument. First, without loss of generality, we consider the setting where $p,q$ are known since any estimator which is oblivious of $p,q$ will only be less powerful. 
		
		Suppose there is an estimator $\hZ \in \bbR[A]_{\leq D}^{n \times n}$ such that $\bbE_{A, M \sim \bbP_{\SBM (p,q)}} ( \ell(\hZ, Z_M) ) < \frac{1}{k} - \left( \frac{1}{k^2} + \frac{r(2-r)}{(1-r)^2 n}  \right)$. Then we could construct $\hM = q + (p-q) \hZ$ and it is a valid estimator for $M$, and belongs to the class $\bbR[A]_{\leq D}^{n \times n}$. Since for each $i < j$, we have $M_{ij} = q + (p-q) Z_{ij}$ and
		\begin{equation*}
			\ell(\hM, M) = \frac{1}{ {n \choose 2} } \sum_{1 \leq i < j \leq n} (\hM_{ij} - M_{ij})^2 = \frac{(p-q)^2}{ {n \choose 2} } \sum_{1 \leq i < j \leq n} (\hZ_{ij} - Z_{ij})^2 = (p-q)^2 \ell(\hZ, Z).
		\end{equation*}
		Thus
		\begin{equation*}
			\bbE_{A, M \sim \bbP_{\SBM (p,q)}} ( \ell(\hM, M) ) < \frac{(p-q)^2}{k} - (p-q)^2\left( \frac{1}{k^2} + \frac{r(2-r)}{(1-r)^2 n}  \right) 
		\end{equation*} and this contradicts with the result in Theorem \ref{th: comp-limit-graphon-est}. So the assumption $\bbE_{A, M \sim \bbP_{\SBM (p,q)}} ( \ell(\hZ, Z_M) ) < \frac{1}{k} - \left( \frac{1}{k^2} + \frac{r(2-r)}{(1-r)^2 n}  \right)$ does not hold and this finishes the proof of this theorem by setting $r = 1/2$.

\subsection{Analysis of SDP}\label{sec:SDP}

Consider the SBM model $M \sim \bbP_{\SBM (p,q)}$ in Section \ref{sec: community-detection}.
The following SDP was considered by \cite{guedon2016community,li2021convex},
\begin{equation} \label{eq: SDP-program}
		\begin{split}
			\max_{Z} \left\langle Z, A - \frac{p+q}{2} \1_n \1^\top_n \right \rangle
\text{ s.t. } Z \succeq 0, Z_{ij} = Z_{ji} > 0 \, \text{ for all }1 \leq  i < j \leq n \text{ and } Z_{ii} = 1.	\end{split}
	\end{equation}Its guarantee is given as follows.
\begin{Proposition} \label{prop: SDP-estimator}
	Suppose $n$ is sufficiently large and $\epsilon\leq q \leq p \leq 1- \epsilon$ for some small $\epsilon > 0$. Let $\hZ$ be the solution of the above SDP. Then there exists $c > 0$ independent of $n$ and $k$ such that  
	\begin{equation*}
		\begin{split}
			\bbE_{A, M \sim \bbP_{\SBM (p,q)}} ( \ell(\hZ, Z_M) ) \leq c \sqrt{\frac{p}{n(p-q)^2}}
		\end{split} 
	\end{equation*}
As a consequence, non-trivial community detection ($\bbE_{A, M \sim \bbP_{\SBM (p,q)}} ( \ell(\hZ, Z_M) )\ll \frac{1}{k}$) is achieved whenever $\frac{n(p-q)^2}{pk^2}\gg 1$.	
\end{Proposition}
\begin{proof}
First, the weak consistency for the solution of the SDP program in \eqref{eq: SDP-program} has been studied in \cite{guedon2016community,li2021convex}. In particular, the Theorem 5.1 in \cite{li2021convex} shows that for any $M \in \cM_{k,p,q}$, we have with probability at least $1 - 2(e/2)^{-2n}$ such that
	\begin{equation} \label{ineq: SDP-l1-error}
		\|\hZ - Z_M\|_1 \leq 45 \sqrt{pn^3}/(p-q).
	\end{equation}
	Moreover, since $Z \succeq 0$, we have $Z = X X^\top$ for some $X \in \bbR^{n \times n}$. Then for every $i \in [n]$, we have $Z_{ii} = X_{i:} X_{i:}^\top = \|X_{i:}\|^2_2$, which is required to be equal to $1$. Thus, $|Z_{ij}| = | X_{i:} X_{j:}^\top | \leq \|X_{i:}\|_2 \| \|X_{j:}\|_2 = 1$. With the additional $Z_{ij} \geq 0$ constraint, we have $\hZ_{ij} \in [0,1]$ for all $i,j \in [n]$. Therefore, with probability at least $1 - 2(e/2)^{-2n}$,
\begin{equation} \label{ineq: SDP-l2-error}
	\|\hZ - Z_M\|_\F^2 \leq \|\hZ - Z_M\|_1 \leq 45 \sqrt{pn^3}/(p-q).
\end{equation}

Then 
\begin{equation*}
	\begin{split}
		\bbE_{A, M \sim \bbP_{\SBM (p,q)}}  ( \ell(\hZ, Z_M) ) &\overset{(a)}\leq \frac{1}{n^2} \bbE_{A, M \sim \bbP_{\SBM (p,q)}} \|\hZ - Z_M\|_\F^2 \\
		& \overset{(b)} \leq  45 \sqrt{\frac{ p}{n (p-q)^2}} + 2(e/2)^{-2n} \\
		& \overset{(c)}\leq c \sqrt{\frac{ p}{n (p-q)^2}} 
	\end{split}
\end{equation*} here (a) is because of the symmetry and the fact that $\hZ_{ii} = 1$ and $(Z_M)_{ii} = 0$ for all $i \in [n]$; (b) is because of \eqref{ineq: SDP-l2-error} and the fact each entry of $\hZ - Z_M$ belongs to $[-1,1]$, and thus $\|\hZ - Z\|_\F^2$ is at most $n^2$; (c) is because $n$ is sufficiently large. 
\end{proof}

\section{Proofs in Section \ref{sec: extension} } \label{sec: proof-extension}

\subsection{A Spectral Algorithm for Sparse Graphon Estimation} \label{eq:spec-sparse}

We show that a truncated SVD procedure on the denoised adjacency matrix \citep{chin2015stochastic} can also achieve the rate $\frac{\rho k}{n}$ in expectation in sparse graphon estimation. Let us define $T_\tau(A) \in \{0,1 \}^{n \times n}$ as a truncated version of $A$ by replacing the $i$th row and column of $A$ with zeros whenever the degree of $i$th node of $A$ exceeds $\tau$. Then, the truncated SVD estimator is
\begin{equation*}
	\hM = \argmin_{ \rank(M) \leq k } \| T_\tau(A) - M \|_\F^2,
\end{equation*}
and its guarantee is given as follows.
\begin{Proposition} \label{prop: sparse-graphon-upper-bound}
	Suppose $n \rho \geq C_1$ and $\tau = C_2n \rho $ for some large constants $C_1,C_2 > 0$. Then there exists $C_3 > 0$ such that for any $M \in \cM_{k, \rho}$, we have $ \bbE( \ell(\hM, M) ) \leq C_3\frac{\rho k}{n}  $.
\end{Proposition}
\begin{proof}
	First, let $\widetilde{M} \in [0,\rho]^{n \times n}$ be the matrix such that its off-diagonal entries are the same as $M$, but $\widetilde{M}_{ii} = \rho$ for all $i \in [n]$. Then
	\begin{equation} \label{ineq: sparse-graphon-upper-bound-1}
		\begin{split}
			\|\hM - M\|_\F^2 &= \|\hM - \tM + \tM - M\|_\F^2 \leq 2( \|\hM - \tM \|_\F^2 + \|\tM - M\|_\F^2 ) \\
			&\overset{(a)}\leq 2k \|\hM - \tM\|^2 + 2 \rho^2 n \leq 2k \|\hM - T_\tau(A) + T_\tau(A) - \tM\|^2 + 2 \rho^2 n \\
			& \leq 2k( \|\hM - T_\tau(A)\| + \|T_\tau(A) - \tM\| )^2 + 2 \rho^2 n  \\
			&\overset{(b)}\leq  8k \|T_\tau(A) - \tM\|^2 + 2 \rho^2 n 
		\end{split}
	\end{equation} where (a) is because $\hM$ and $\tM$ are all of rank at most $k$, and thus $\rank(\hM - \tM) \leq 2k$; (b) is because the truncated SVD $\hM$ is best rank $k$ approximation of $T_\tau(A)$ in any unitarily invariant norms \citep[Theorem 2]{mirsky1960symmetric} and $\tM$ is of rank at most $k$;	
	
 By \cite[Lemma 20]{gao2017achieving} or \cite[Lemma 27]{chin2015stochastic}, with probability at least $1 - n^{-C''}$ for some large $C''$ depending on $C_2$, when $\tau \asymp \sqrt{n \rho} $,  we have
	\begin{equation} \label{ineq: sparse-graphon-upper-bound-2}
		\|T_\tau(A) - \tM\|^2 \leq C'''(n \rho + 1).
	\end{equation}
	In summary, by combining \eqref{ineq: sparse-graphon-upper-bound-1} and \eqref{ineq: sparse-graphon-upper-bound-2}, we have
	\begin{equation*}
		\begin{split}
			\bbE( \ell(\hM, M) ) &\leq \bbE \left( \frac{1}{2{n \choose 2}} \|\hM - M\|_\F^2 \right) \overset{\eqref{ineq: sparse-graphon-upper-bound-1} } \leq   \frac{1}{2{n \choose 2}} \bbE\left(  8k \|T_\tau(A) - \tM\|^2 + 2 \rho^2 n  \right)  \\
			& \overset{(a)}\leq \frac{1}{2{n \choose 2}} \left( 8C'''n \rho k + 2 \rho^2 n  \right) + \frac{1}{2{n \choose 2}}n^{-C''} n^2  \\
			& \overset{(b)}\leq C_3 \frac{\rho k}{n} + c n^{-C''} \overset{(c)}\leq C_3 \frac{\rho k}{n}
		\end{split}
	\end{equation*} where (a) is because of \eqref{ineq: sparse-graphon-upper-bound-2}, $n \rho \geq C_1$ and the fact that $\| T_\tau(A) - \tM \|^2 \leq \| T_\tau(A) - \tM \|_\F^2 \leq n^2$ by constraint; (b) is because $\rho < 1$; (c) is because $C''$ can be large enough such that $\frac{\rho k}{n}$ dominates $n^{-C''}$. This finishes the proof of this proposition.
	\end{proof}

\subsection{Proof of Theorem \ref{th: sparse-graphon-estimation-final-lower-bound}  }
	First, we apply the results in Theorem \ref{th: comp-limit-graphon-est} and use the new $p,q$ definition here. Given any $0 < \rho < 1,0 \leq q < p \leq 1$, we define $\cM_{k,\rho, p, q}$ as the sparse version of $\cM_{k, p, q}$:
	\begin{equation} \label{def: Mk-rho-pq}
\begin{split}
	\cM_{k,\rho, p,q} = \Big\{ &M = (M_{ij})  \in [0,1]^{  n \times n}: Q = \rho q \1_k \1_k^\top + \rho(p-q) I_k, M_{ii} = 0 \text{ for } i \in [n], \\
	&  M_{ij} = M_{ji} = Q_{z_i z_j} \text{ for all } i \neq j \text{ for some } z \in [k]^{n} \Big\}.
\end{split}
\end{equation}
Then by Theorem \ref{th: comp-limit-graphon-est} with $q, p$ being replaced by $\rho q, \rho p$, we have given any $0 < r < 1$, $2 \leq k \leq \sqrt{n}$ and $D \geq 1$, if  
\begin{equation*} 
	\frac{(p-q)^2 \rho^2 }{\rho q(1- \rho p)} \leq \frac{r}{(D(D+1))^2}\frac{k^2}{n},
\end{equation*} then there exists $c_r > 0$ such that
\begin{equation} \label{ineq: sparse-graphon-ineq1}
\begin{split}
	\inf_{\hM \in \bbR[A]^{n \times n}_{\leq D} } \bbE_{A, M \sim \bbP_{\SBM (\rho, p,q)}} (\ell(\hM, M)) &\geq \frac{(p-q)^2 \rho^2 }{k} - (p-q)^2 \rho^2 \left( \frac{1}{k^2} + \frac{r(2-r)}{(1-r)^2 n}  \right) \\
	& \geq c_r \frac{(p-q)^2 \rho^2 }{k},
\end{split}
\end{equation} where $\bbP_{\SBM (\rho, p,q)}$ denotes the uniform prior on $\cM_{k,\rho,p,q}$. 

	Then, we take $q = \epsilon$ and $p = 1- \epsilon$ with $\epsilon > 0$ being a small enough constant such that $\frac{(p-q)^2 \rho }{q(1-p \rho)} \geq c' \frac{r}{(D(D+1))^2} \frac{k^2}{n} $ for some $0 < c' < 1$. This is achievable as $\rho \geq \frac{c k^2 }{n}$ for some $0<c< 1$ and $\frac{(p-q)^2 }{q(1-p \rho)}$ with $q = \epsilon$ and $p = 1- \epsilon$ is a monotonically increasing function as $\epsilon$ decreases. Then we have
	\begin{equation*}
	\begin{split}
		\inf_{\hM \in \bbR[A]^{n \times n}_{\leq D} } \sup_{M \in \cM_{k,\rho}} \bbE( \ell(M, \hM) ) & \geq \inf_{\hM \in \bbR[A]^{n \times n}_{\leq D} } \bbE_{A, M \sim \bbP_{\SBM (\rho, p,q)}} \ell(\hM, M) \\
		& \overset{ \eqref{ineq: sparse-graphon-ineq1} } \geq c_r \frac{(p-q)^2 \rho^2}{k} = c_r \frac{r}{(D(D+1))^2} q(1-p \rho) \frac{k \rho}{n} \geq \frac{c'' k \rho}{nD^4},
	\end{split}
	\end{equation*} where $c''$ is independent of $k$ and $n$.

\subsection{Proof of Proposition \ref{prop: biclustering-upper-bound}}
	Without loss of generality, we assume $k_1 \leq k_2$. First, we have the following inequality holds almost surely:
	\begin{equation} \label{ineq: biclustering-prop-ineq}
		\begin{split}
			\|\hM - M\|_\F^2 &\overset{(a)}\leq 2k_1 \|\hM - M\|^2 \leq 2k_1 \|\hM - Y + Y - M\|^2 \leq 2k_1( \|\hM - Y\| + \|Y - M\| )^2 \\
			&\overset{(b)}\leq  8k_1 \|Y - M\|^2. 
		\end{split}
	\end{equation} where (a) is because $\hM$ and $M$ are all of rank at most $k_1$, and thus $\rank(\hM - M) \leq 2k_1$; (b) is because the truncated SVD $\hM$ is best rank $k_1$ approximation of $Y$ in any unitarily invariant norms \citep[Theorem 2]{mirsky1960symmetric} and $M$ is of rank equal or smaller than $k_1$.
	
	Thus
		\begin{equation} \label{ineq: biclustering-upp-bd-ineq}
	\begin{split}
		\bbE[\|\hM - M\|_\F^2] &\overset{ \eqref{ineq: biclustering-prop-ineq} }\leq 8 k_1 \bbE[ \|Y - M\|^2 ] \\
		& = 8 k_1 \int_{0}^{\infty} \bbP( \|Y - M\|^2  > t ) dt \\
		& = 8 k_1 \int_{0}^{\infty} \bbP( \|Y - M\|  > \sqrt{t} ) dt \\
		& = 8 k_1 \int_{0}^{( \sqrt{n_1} + \sqrt{n_2} )^2} \bbP( \|Y - M\|  > \sqrt{t} ) dt + 8 k_1 \int_{( \sqrt{n_1} + \sqrt{n_2} )^2}^{\infty} \bbP( \|Y - M\|  > \sqrt{t} ) dt \\
		& \leq 8 k_1 ( \sqrt{n_1} + \sqrt{n_2} )^2 + 8 k_1 \int_{0}^{\infty} 2 (\sqrt{n_1} + \sqrt{n_2} + x) \bbP( \|Y - M\|  > \sqrt{n_1} + \sqrt{n_2} + x )  dx \\
		& \overset{(a)}\leq 16 k_1 (n_1 + n_2) +  16 k_1 \int_{0}^{\infty} 2(\sqrt{n_1} + \sqrt{n_2} + x)\exp(-x^2/2) dx \\
		& \overset{(b)}\leq C  k_1 (n_1 + n_2),
	\end{split}
	\end{equation} 
where (a) is because $(a+b)^2 \leq 2(a^2 + b^2)$ and the standard concentration results for random matrices with i.i.d. $N(0,1)$ entries, see \cite[Corollary 5.35]{vershynin2010introduction}; (b) is because $\int_{0}^\infty x \exp(-x^2/2) = 1$ and $\int_{0}^\infty \exp(-x^2/2) = \sqrt{2\pi}/2$.  Finally, for any $M \in \cM_{k_1, k_2}$, we have $$\bbE[\ell(M, \hM) ] = \frac{1}{n_1 n_2} \bbE[\|\hM - M\|_\F^2] \overset{ \eqref{ineq: biclustering-upp-bd-ineq} }\leq \frac{C k_1}{n_1 \wedge n_2 }.$$

\subsection{Proof of Theorem \ref{th: biclustering-comp-lower-bound}  }
\begin{proof}
	Without loss of generality, we assume $k_1 \leq k_2$ in the proof and to get the results in the statement, we just need to replace $k_1$ by $k_1 \wedge k_2$ in the end. Due to symmetry, we have
	\begin{equation*}
		\begin{split}
			\inf_{ \hM \in \bbR[Y]_{\leq D}^{n_1 \times n_2} } \bbE_{Y, M \sim \bbP_{\BC (\lambda)} } ( \ell(\hM, M) ) & = \inf_{ g \in \bbR[A]_{\leq D} } \bbE_{Y, M \sim \bbP_{\BC (\lambda)} } [( g(Y) - M_{11} )^2] \\
			& = \inf_{ g \in \bbR[A]_{\leq D} } \sum_{j=1}^{k_1} \bbE_{Y, M \sim \bbP_{\BC (\lambda)} } [( g(Y) - M_{11} )^2 | (z_1)_1 = j] \bbP( (z_1)_1 = j ) \\
			& = \inf_{ g \in \bbR[A]_{\leq D} } \bbE_{Y, M \sim \bbP_{\BC (\lambda)} } [( g(Y) - M_{11} )^2 | (z_1)_1 = 1]\\
			& = \inf_{ g \in \bbR[A]_{\leq D} } \bbE_{Y, M \sim \bbP'_{\BC} } [( g(Y) - M_{11} )^2],
		\end{split}
	\end{equation*} where in the last equality, we denote $\bbP_{\BC (\lambda)}'$ as the restriction of $\bbP_{\BC (\lambda)}$ on $\cM_{k_1, k_2}$ such that $\{(z_1)_1 = 1 \}$. The biclustering model falls in the class of additive Gaussian noise model, so we could apply Proposition \ref{prop: schramm-wein-gaussian} in Appendix \ref{sec: comp-lower-bound-gaussian-model}. In the notation described in Appendix \ref{sec: comp-lower-bound-gaussian-model}, $X$ here is the vectorization of $M$, $x = M_{11}$ and $N = n_1 n_2$.
	
	Then by Proposition \ref{prop: schramm-wein-gaussian}, we have
	\begin{equation} \label{ineq: bicluster-proof-ineq1}
		\inf_{ g \in \bbR[A]_{\leq D} } \bbE_{Y, M \sim P'_{\BC} } [( g(Y) - M_{11} )^2] \geq \bbE_{M \sim \bbP_{\BC (\lambda)}'}( M_{11}^2 ) - \sum_{ \balpha \in \bbN^N, 0 \leq | \balpha | \leq D } \frac{\kappa_{\balpha}^2(M_{11}, X)}{ \balpha! }.
	\end{equation}
	
	Notice that $\bbE_{M \sim \bbP_{\BC (\lambda)}'}( M_{11}^2 ) = \lambda^2/k_1$ and $\bbP_{\BC (\lambda)}'$ is exactly the same prior considered in Propositions \ref{prop: biclustering-alpha-property} and \ref{prop: biclustering-bound-sum-of-alpha}, then by Proposition \ref{prop: biclustering-bound-sum-of-alpha}, we have
	\begin{equation}
	\begin{split}
		\inf_{ \hM \in \bbR[Y]_{\leq D}^{n_1 \times n_2} } \bbE_{Y, M \sim \bbP_{\BC (\lambda)} } ( \ell(\hM, M) ) &=\inf_{g \in \bbR[A]_{\leq D} } \bbE_{Y, M \sim P'_{\BC}} [ ( g(Y) - M_{11} )^2 ] \\
		& \overset{ \eqref{ineq: bicluster-proof-ineq1} } \geq  \frac{\lambda^2}{k_1} - \sum_{ \balpha \in \bbN^N, 0 \leq | \balpha | \leq D } \frac{\kappa_{\balpha}^2(M_{11}, X)}{ \balpha! } \\
		& \overset{ \text{Proposition } \ref{prop: biclustering-bound-sum-of-alpha} }\geq \frac{\lambda^2}{k_1} - \left( \frac{\lambda^2}{k_1^2} + \frac{r(2-r)\lambda^2}{(1-r)^2(n_1 \vee n_2)}  \right).
	\end{split}
	\end{equation} 	This finishes the proof of this theorem.
\end{proof}

\subsubsection{Additional Proofs for Theorem \ref{th: biclustering-comp-lower-bound} }
Note that in biclustering, we can also view $\balpha \in \bbN^N$ as a multi-bipartite graph between vertex sets $[n_1]$ and $[n_2]$. In this case, let $V_1(\balpha) \subseteq [n]$ denote the set of vertices spanned by the vertices of $\balpha$ from Group $1$ and $V_2(\balpha) \subseteq [n]$ denote the set of vertices spanned by the vertices of $\balpha$ from Group $2$. We also let $|V(\balpha)| = |V_1(\balpha)| + |V_2(\balpha)|$. Again without loss of generality, we assume $k_1 \leq k_2$ throughout this section. Next, we establish some useful properties regarding $\kappa_{\balpha}(M_{11}, X)$ in the biclustering model.
\begin{Proposition} \label{prop: biclustering-alpha-property}
	Suppose $M$ is generated from $\bbP_{\BC (\lambda)}$ condition on $\{(z_1)_1 = 1\}$. Denote $X \in \bbR^{n_1 n_2}$ as the vectorization of $M$ and $x$ as the first entry of $X$, i.e., $x = M_{11}$. Then we have
	\begin{itemize}
		\item (i) If $\balpha$ is disconnected or $\balpha$ is connected but $1 \notin V_2(\balpha)$, then $\kappa_{\balpha}(x, X) = 0$.
		\item (ii) If $\balpha$ is connected with $|\balpha| \geq 1$, $1 \in V_2(\balpha)$, $1 \notin V_1(\balpha)$, then $\kappa_{\balpha}(x, X) = 0$.

		\item (iii) If $\balpha$ is connected with $|\balpha| \geq 1$, $1 \in V_2(\balpha)$, $1 \in V_1(\balpha)$, then $|\kappa_{\balpha}(x, X)| \leq \lambda^{|\balpha|+1} (1/k_1)^{|V(\balpha)| - 1} (|\balpha| +1 )^{|\balpha|} $.
	\end{itemize}
\end{Proposition}
\begin{proof}
	{\bf Proof of (i)}. By \eqref{eq: kappa-cumulant-connection}, we know that $\kappa_{\balpha}(x, X)$ is the joint cumulant of a group of random variables, say $\cG$. For the case $\balpha$ is disconnected, $\cG$ could be divided into $\cG_1$ and $\cG_2$ and $\cG_1$ and $\cG_2$ are independent of each other. Thus, by Proposition \ref{prop: cumulant-independent-prop}, $\kappa_{\balpha}(x, X)$ is zero. Similarly, for the case $\balpha$ is connected but $1 \notin V_2(\balpha)$, we know in the prior for $X$, $(z_1)_{1}$ is known and fixed, so if $1 \notin V_2(\balpha)$, $x$ will be independent of $X$. By Proposition \ref{prop: cumulant-independent-prop}, $\kappa_{\balpha}(x, X)$ will be zero.
	
	{\bf Proof of (ii)}. First, for any connected $\balpha$, we have
	 \begin{equation} \label{eq: E-X-alpha-formula-biclustering}
	 	\begin{split}
	 		\bbE[X^{\balpha}] &= \lambda^{|\balpha|} \bbP\Big(  \text{all vertices in } V_1(\balpha) \cup V_2(\balpha) \text{ are in the same community} \Big) \\
	 		 & = \lambda^{|\balpha|} \cdot \left(\frac{1}{k_1}\right)^{ |V_1(\balpha)| + |V_2(\balpha)| -1 } = \lambda^{|\balpha|} \cdot \left(\frac{1}{k_1}\right)^{ |V(\balpha)|  -1 }, \\
	 		\bbE[xX^{\balpha}] &=  \lambda^{|\balpha| + 1} \bbP( \text{all vertices in } V_1(\balpha) \cup \{1 \} \text{ and } V_2(\balpha) \cup \{1 \}  \text{ are in the same community} ) \\
	 		& = \lambda^{|\balpha| + 1}\left(\frac{1}{k_1}\right)^{ |V_1(\balpha) \cup \{1 \}| + |V_2(\balpha) \cup \{1 \}| -1 }.
	 	\end{split}
	 \end{equation}

	 Next, we prove the claim by induction. When $|\balpha| = 0$, $\kappa_0(x, X) = \bbE(x) = \frac{\lambda}{k_1}$. Then, for $\balpha$ such that $|\balpha| = 1$, $1 \in V_2(\balpha) $ and $1 \notin V_1(\balpha)$, we have 
	 \begin{equation*}
	 	\kappa_{\balpha}(x, X) \overset{ \eqref{eq: kappa-recursive-relation} } = \bbE[x X^{\balpha}] - \kappa_0(x, X) \bbE[X^{\balpha}] \overset{\eqref{eq: E-X-alpha-formula-biclustering} } = \lambda^{|\balpha| + 1} \left(\frac{1}{k_1}\right)^{ |V_1(\balpha)| + |V_2(\balpha)| } - \frac{\lambda}{k_1} \lambda^{|\balpha|}\left(\frac{1}{k_1}\right)^{ |V_1(\balpha)| + |V_2(\balpha)| - 1 }  =0.
	 \end{equation*}
	 
	 Now assume that given any $t \geq 2$ and any $\balpha$ such that $1 \in V_2(\balpha) $, $1 \notin V_1(\balpha)$ and $|\balpha| < t$, $\kappa_{\balpha}(x, X) = 0$. Then for any such $\balpha$ with $|\balpha| = t$, we have
	 \begin{equation*}
	 	\begin{split}
	 		\kappa_{\balpha}(x, X) &\overset{ \eqref{eq: kappa-recursive-relation} } =  \bbE(x X^{\balpha}) - \sum_{0 \leq \bbeta \lneq \balpha } \kappa_{\bbeta}(x, X) {\balpha \choose \bbeta } \bbE[X^{\balpha - \bbeta}] \\
	 		& \overset{(a)}= \bbE(x X^{\balpha}) -\kappa_0(x, X) \bbE[X^{\balpha }] = \lambda^{|\balpha| + 1} \left(\frac{1}{k_1}\right)^{ |V_1(\balpha)| + |V_2(\balpha)| } - \frac{\lambda}{k_1} \lambda^{|\balpha|}\left(\frac{1}{k_1}\right)^{ |V_1(\balpha)| + |V_2(\balpha)| - 1 } \\
	 		&  =0,
	 	\end{split}
	 \end{equation*} where in (a), for any $\bbeta$ such that $|\bbeta| \geq 1$, $1 \notin V_1(\bbeta)$ since $\bbeta$ is a subgraph of $\balpha$, and thus $\kappa_{\bbeta}(x, X) = 0$ for both the case $1 \notin V_2(\bbeta)$ by the result we have proved in part (i) and the case $1 \in V_2(\bbeta)$ by the induction assumption. 
This finishes the induction, and we have that for any $\balpha$ such that $|\balpha| \geq 1$, $1 \in V_2(\balpha) $ and $1 \notin V_1(\balpha)$, $\kappa_{\balpha}(x, X) = 0$. 
	
{\bf Proof of (iii).} First, for any connected subgraph $\bbeta$ of $\balpha$,
\begin{equation}\label{eq: E-X-alpha-beta-formula-biclustering}
\begin{split}
	 \bbE [X^{\balpha - \bbeta}] &= \lambda^{|\balpha - \bbeta|} \bbP(\text{each connected component in }\balpha - \bbeta \text{ belongs to the same community} ) \\
	& = \lambda^{|\balpha - \bbeta|} \left(\frac{1}{k_1}\right)^{ |V(\balpha-\bbeta)| - \cC(\balpha - \bbeta) },
\end{split}
\end{equation} where $\cC(\balpha - \bbeta)$ denotes the number of connected components in $\balpha - \bbeta$. 	
	
	Next, we prove the claim by induction. Recall that when $|\balpha| = 0$, $\kappa_0(x, X) = \frac{\lambda}{k_1}$. Then, for $\balpha$ such that $|\balpha| = 1$, $1 \in V_1(\balpha) $ and $1 \in V_2(\balpha)$, we have 
	 \begin{equation*}
	 \begin{split}
	 	\kappa_{\balpha}(x, X) &\overset{ \eqref{eq: kappa-recursive-relation} } = \bbE[x X^{\balpha}] - \kappa_0(x, X) \bbE[X^{\balpha}] \overset{\eqref{eq: E-X-alpha-formula-biclustering} } = \lambda^{|\balpha| + 1} \left(\frac{1}{k_1}\right)^{ |V(\balpha)|-1 } - \frac{\lambda}{k_1} \lambda^{|\balpha|}\left(\frac{1}{k_1}\right)^{ |V(\balpha)| -1 } \\
	 	& = \lambda^{|\balpha| + 1} \left(\frac{1}{k_1}\right)^{ |V(\balpha)|-1 } - \frac{\lambda}{k_1} \lambda^{|\balpha|}\left(\frac{1}{k_1}\right)^{ |V(\balpha)| -1 } = \lambda^{|\balpha| + 1} \left(\frac{1}{k_1}\right)^{ |V(\balpha)|-1 } ( 1- 1/k_1 ),
	 \end{split}
	 \end{equation*} thus $|\kappa_{\balpha}(x, X)| \leq \lambda^{|\balpha|+1} (1/k_1)^{|V(\balpha)| - 1} (|\balpha| +1 )^{|\balpha|}$ holds for $|\balpha| = 1$. 

Now assume that given any $t \geq 2$ and any $\balpha$ with $1 \in V_1(\balpha) $, $1 \in V_2(\balpha)$ and $|\balpha| < t$, $|\kappa_{\balpha}(x, X)| \leq \lambda^{|\balpha|+1} (1/k_1)^{|V(\balpha)| - 1} (|\balpha| +1 )^{|\balpha|}$. Then for any such $\balpha$ with $|\balpha| = t$, we have
\begin{equation*}
	\begin{split}
		|\kappa_{\balpha}(x, X)|  & \overset{ \eqref{eq: kappa-recursive-relation} } =  \big|\bbE(x X^{\balpha}) - \sum_{0 \leq \bbeta \lneq \balpha } \kappa_{\bbeta}(x, X) {\balpha \choose \bbeta } \bbE[X^{\balpha - \bbeta}]\big| \\
		& \leq \big|\bbE(x X^{\balpha})\big| + \sum_{0 \leq \bbeta \lneq \balpha } \big|\kappa_{\bbeta}(x, X) \big| {\balpha \choose \bbeta } \bbE[X^{\balpha - \bbeta}] \\
				& \overset{\text{Part }(i)}=  \big|\bbE(x X^{\balpha})\big| + \sum_{0 \leq \bbeta \lneq \balpha: \bbeta \text{ is connected }  } \big|\kappa_{\bbeta}(x, X) \big| {\balpha \choose \bbeta } \bbE[X^{\balpha - \bbeta}] \\
		& \overset{\eqref{eq: E-X-alpha-formula-biclustering}, \eqref{eq: E-X-alpha-beta-formula-biclustering} } = \lambda^{|\balpha| + 1} \left(\frac{1}{k_1}\right)^{ |V(\balpha)|-1 } + \sum_{0 \leq \bbeta \lneq \balpha: \bbeta \text{ is connected } } \big|\kappa_{\bbeta}(x, X) \big| {\balpha \choose \bbeta } \lambda^{|\balpha - \bbeta|} \left(\frac{1}{k_1}\right)^{ |V(\balpha-\bbeta)| - \cC(\balpha - \bbeta) } \\
		& \overset{(a)}= \lambda^{|\balpha| + 1} \left(\frac{1}{k_1}\right)^{ |V(\balpha)|-1 } +  \big|\kappa_{0}(x, X) \big|\lambda^{|\balpha|} \left(\frac{1}{k_1}\right)^{ |V(\balpha)| - 1 } \\
		 &+\sum_{ \substack{0 \lneq \bbeta \lneq \balpha,\\ \bbeta \text{ is connected },\\ 1 \in V_1(\bbeta), 1 \in V_2(\bbeta)} } \big|\kappa_{\bbeta}(x, X) \big| {\balpha \choose \bbeta } \lambda^{|\balpha - \bbeta|} \left(\frac{1}{k_1}\right)^{ |V(\balpha-\bbeta)| - \cC(\balpha - \bbeta) } \\
		 & \overset{(b)} = \lambda^{|\balpha| + 1} \left(\frac{1}{k_1}\right)^{ |V(\balpha)|-1 } + \lambda^{|\balpha| + 1} \left(\frac{1}{k_1}\right)^{ |V(\balpha)|} \\
		 & +\sum_{ \substack{0 \lneq \bbeta \lneq \balpha,\\ \bbeta \text{ is connected },\\ 1 \in V_1(\bbeta), 1 \in V_2(\bbeta)} } \lambda^{|\bbeta|+1} (1/k_1)^{|V(\bbeta)| - 1} (|\bbeta| +1 )^{|\bbeta|}  {\balpha \choose \bbeta } \lambda^{|\balpha - \bbeta|} \left(\frac{1}{k_1}\right)^{ |V(\balpha-\bbeta)| - \cC(\balpha - \bbeta) }\\
		 & \overset{(c) }\leq  2\lambda^{|\balpha| + 1} \left(\frac{1}{k}\right)^{ |V(\balpha)|-1 } + \lambda^{|\balpha| + 1}\left(\frac{1}{k}\right)^{ |V(\balpha)|-1 }\sum_{ \substack{0 \lneq \bbeta \lneq \balpha,\\ \bbeta \text{ is connected },\\ 1 \in V_1(\bbeta), 1 \in V_2(\bbeta)} }  (|\bbeta| +1 )^{|\bbeta|}  {\balpha \choose \bbeta } 
	\end{split}
\end{equation*} where in (a), we separate the term $\bbeta =0$ in the summation and then use the results proved in (i)(ii) of this proposition; (b) is because $\kappa_0(x, X) = \frac{\lambda}{k_1}$ and by the induction assumption; (c) is because of a bipartite version of Lemma \ref{lm: graph-subgraph-vertex-connection} and the proof is the same as Lemma \ref{lm: graph-subgraph-vertex-connection} by defining $|V(\balpha)| = |V_1(\balpha)| + |V_2(\balpha)|$.

	The rest of the proof for (iii) is the same as the one in \eqref{ineq: SBM-alpha-prop3-main-arg2} by replacing $k$ with $k_1$ and we omit it here for simplicity. This finishes the proof of this proposition.
\end{proof}

Next, we bound the number of $\balpha$ such that $\kappa_{\balpha}(x, X)$ is nonzero provided in Proposition \ref{prop: biclustering-alpha-property}.
\begin{Lemma}\label{lm: num-nonzero-alpha-biclustering}
	Given any $d \geq 1$, $0 \leq h \leq d-1$, the number of connected $\balpha$ such that $1 \in V_1(\balpha)$, $1 \in V_2(\balpha)$, $|\balpha| = d$ and $|V(\balpha)| = d+1 -h$ is at most $(n_1 \vee n_2)^{d-h-1} d^{d+h}$.
\end{Lemma}
\begin{proof}
	We view $\balpha$ as a multi-bipartite graph on two sets of nodes $[n_1]$ and $[n_2]$ and count the number of ways to construct such $\balpha$. The counting strategy is the following: we start with adding Vertex $1$ from Group 1 to $\balpha$ and then add $(d-h)$ edges such that for each edge there, it will introduce a new vertex either from Group 1 or Group 2; then in the second stage, we add the rest of the $h$ edges on these existing vertices. In the first stage, we can also count different cases by considering when will Vertex $1$ from Group 2 be introduced in adding new vertices.
	\begin{itemize}
		\item If Group 2 Vertex $1$ is the first vertex to be added after Group 1 Vertex $1$, then the number of such choices of $\balpha$ is at most $((n_1 \vee n_2)d)^{d-h-1} (d^2)^h$. Here $((n_1 \vee n_2)d)^{d-h-1}$ is because for each of the rest of $d-h-1$ edges, the number of choices for the starting vertex is at most $d$ since there are at most $(d+1)$ vertices in total and the number of choices for a newly introduced vertex is at most $(n_1 \vee n_2)$ since it is a bipartite graph. $(d^2)^h$ comes from that in the second stage, the choice of each extra edge is at most $ \lfloor \frac{d+1}{2} \rfloor \lceil \frac{d+1}{2} \rceil   \leq d^2$.
		\item By the same counting strategy, if Group 2 Vertex $1$ is the second vertex to be added in the first stage, then the number of such choices of $\balpha$ is at most $((n_1 \vee n_2)d)^{d-h-1} (d^2)^h$.
		\item $\cdots$
		\item If Group 2 Vertex $1$ is the $(d-h)$-th vertex to be added in the first stage, then the number of such choices of $\balpha$ is at most $((n_1 \vee n_2)d)^{d-h-1} (d^2)^h$.
	\end{itemize}
	  By adding them together, the number of choices of connected $\balpha$ such that $1 \in V_1(\balpha)$, $1 \in V_2(\balpha)$, $|\balpha| = d$ and $|V(\balpha)| = d+1 -h$ is at most 
	  \begin{equation*}
	  	(d-h) ((n_1 \vee n_2)d)^{d-h-1} (d^2)^h \leq d ((n_1 \vee n_2)d)^{d-h-1} (d^2)^h = (n_1 \vee n_2)^{d-h-1} d^{d+h}.
	  \end{equation*}
\end{proof}

Finally, we bound the $\sum_{ \balpha \in \bbN^N, 0 \leq | \balpha | \leq D } \frac{\kappa_{\balpha}^2(x, X)}{ \balpha! }$ in the following Proposition \ref{prop: biclustering-bound-sum-of-alpha}.
\begin{Proposition}\label{prop: biclustering-bound-sum-of-alpha}
	Under the same setting as in Proposition \ref{prop: biclustering-alpha-property}, we have for any $D \geq 1$, we have
	\begin{equation*}
	\begin{split}
		 & \sum_{\balpha \in \bbN^N, 0\leq|\balpha| \leq D }  \frac{\kappa_{\balpha}^2(x,X)}{\balpha!} \\
		& \leq \frac{\lambda^2}{k_1^2} - \frac{\lambda^2}{(n_1 \vee n_2)} +\frac{\lambda^2}{(n_1 \vee n_2)} \sum_{h=0}^d \left( D^2 (D+1)^2 \lambda^2 \right)^h \sum_{d=h}^D \left( D (D+1)^2 \frac{(n_1 \vee n_2) \lambda^2}{k_1^2} \right)^{d-h}.
	\end{split}
	\end{equation*}
	
	In particular, for any $0 < r < 1$, if $\lambda^2 \leq \frac{r}{(D(D+1))^2} \min\left(1, \frac{k_1^2}{n_1 \vee n_2} \right)$, then we have 
	\begin{equation*}
		\sum_{\balpha \in \bbN^N, 0 \leq |\balpha| \leq D }  \frac{\kappa_{\balpha}^2(x,X)}{\balpha!} \leq \frac{\lambda^2}{k_1^2} + \frac{r(2-r)\lambda^2}{(1-r)^2(n_1 \vee n_2)}.
	\end{equation*}
\end{Proposition} \begin{proof}
	The proof of this proposition is almost the same as the proof of Proposition \ref{prop: sum-kappa-bound-SBM} by replacing $k$ with $k_1$, $n$ with $(n_1 \vee n_2)$ and $\frac{(p-q)^2}{q(1-p)} $ with $\lambda^2$. We omit the proof here for simplicity.
\end{proof}

\subsection{Proof of Corollary \ref{coro: biclustering-estimation-final-lower-bound}  }
	Without loss of generality, we assume $k_1 \leq k_2$. Since $k_1 \geq 2$ and $n_1 \vee n_2 \geq k_1^2 \geq 2k_1$, by Theorem \ref{th: biclustering-comp-lower-bound} we have there exists a small enough $r $ and $c_r > 0$ such that when $\lambda^2 = \frac{r}{(D(D+1))^2} \min(1,\frac{k_1^2}{n_1 \vee n_2}) $, we have
	\begin{equation*}
	\begin{split}
		\inf_{\hM \in \bbR[Y]^{n_1 \times n_2}_{\leq D} } \sup_{M \in \cM_{k_1, k_2}} \bbE( \ell(M, \hM) ) &\geq \inf_{\hM \in \bbR[Y]^{n_1 \times n_2}_{\leq D} } \bbE_{Y, M \sim \bbP_{\BC (\lambda)}} \ell(\hM, M) \\
		 &\geq c_r \frac{\lambda^2}{k_1} = c \frac{r}{D^4} \left( \frac{k_1}{n_1 \vee n_2} \wedge \frac{1}{k_1} \right),
	\end{split}
	\end{equation*}where $c$ depends $r$ only.

\end{sloppypar}

\end{document}